\documentclass{article}
\usepackage[utf8]{inputenc}
\usepackage{fullpage}
\usepackage{amsmath}
\usepackage{amssymb}
\usepackage{amsthm}
\usepackage{epsfig}
\usepackage{epstopdf}
\usepackage{mathtools}
\usepackage{relsize}
\usepackage{hyperref}
\hypersetup{
    colorlinks=true,     
    linkcolor=blue,     
    citecolor=blue,     
}
\usepackage[numbers,sort&compress]{natbib}
\usepackage{enumitem}
\usepackage{algorithm,algpseudocode}
\usepackage{setspace}
\usepackage{caption}
\usepackage{subcaption}
\usepackage{xspace}
\usepackage{authblk}
\usepackage{dsfont}

\title{Residuals-based distributionally robust optimization \\ with covariate information}

\date{Version 2 (this document): May 3, 2022 \\[0.1in] \hspace*{-0.85in} Version 1: December 2, 2020}

\author[1]{Rohit Kannan}
\author[2]{G{\"u}zin Bayraksan}
\author[3]{James R. Luedtke}
\affil[1]{Center for Nonlinear Studies (T-CNLS) and Applied Mathematics \& Plasma Physics (T-5), \protect\\ Los Alamos National Laboratory, Los Alamos, NM, USA. E-mail: rohitk@alum.mit.edu}
\affil[2]{Department of Integrated Systems Engineering, The Ohio State University, Columbus, OH, USA. \protect\\ E-mail: bayraksan.1@osu.edu}
\affil[3]{Department of Industrial \& Systems Engineering and Wisconsin Institute for Discovery, \protect\\ University of Wisconsin-Madison, Madison, WI, USA. E-mail: jim.luedtke@wisc.edu}

\newcommand{\Set}[2]{\left\lbrace #1 : #2 \right\rbrace}

\DeclareMathOperator*{\argmin}{arg\,min}

\newcommand{\tr}[1]{\ensuremath{{#1}^\text{T}}}

\newcommand{\uset}[2]{\ensuremath{\underset{#1}{#2}}}

\newcommand{\1}{\mathds{1}}

\DeclarePairedDelimiter\abs{\lvert}{\rvert}%
\DeclarePairedDelimiter\norm{\lVert}{\rVert}%

\newcommand{\prob}[1]{\mathbb{P}\left\lbrace{#1}\right\rbrace}
\newcommand{\expect}[1]{\mathbb{E}\left[{#1}\right]}
\newcommand{\expectation}[2]{\mathbb{E}_{#1}\left[{#2}\right]}

\newcommand{\dev}[2]{\mathbb{D}\left({#1},{#2}\right)}

\newcommand{\convinprob}{\xrightarrow{p}}

\newcommand{\proj}[2]{\operatorname{proj}_{#1}(#2)}

\newcommand{\D}{\mathcal{D}}
\newcommand{\F}{\mathcal{F}}

\newcommand{\R}{\mathbb{R}}

\newcommand{\hS}{\hat{S}}
\newcommand{\X}{\mathcal{X}}
\newcommand{\Y}{\mathcal{Y}}
\newcommand{\Z}{\mathcal{Z}}

\newcommand{\hf}{\hat{f}}
\newcommand{\hg}{\hat{g}}

\newcommand{\heps}{\hat{\varepsilon}}
\newcommand{\teps}{\tilde{\varepsilon}}

\newcommand{\hv}{\hat{v}}

\newcommand{\hz}{\hat{z}}

\newcommand{\hP}{\hat{\mathcal{P}}}
\newcommand{\Pf}{\mathfrak{P}}

\newcommand{\tE}{\texttt{E}}
\newcommand{\tJ}{\texttt{J}}
\newcommand{\tW}{\texttt{W}}
\newcommand{\tS}{\texttt{S}}
\newcommand{\tH}{\texttt{H}}

\newcommand{\tone}{\texttt{1}}
\newcommand{\ttwo}{\texttt{2}}
\newcommand{\tthree}{\texttt{3}}

\providecommand{\keywords}[1]
{
  \small	
  \textbf{Key words:} #1
}

\newtheorem{theorem}{Theorem}[]
\newtheorem{lemma}[theorem]{Lemma}
\newtheorem{assumption}{Assumption}[]
\newtheorem{remark}{Remark}
\newtheorem{definition}{Definition}
\newtheorem{corollary}[theorem]{Corollary}
\newtheorem{proposition}[theorem]{Proposition}
\newtheorem{example}{Example}
\newtheorem{conjecture}[theorem]{Conjecture}

\let\oldtheorem\theorem
\renewcommand{\theorem}{\oldtheorem\normalfont}

\let\oldlemma\lemma
\renewcommand{\lemma}{\oldlemma\normalfont}

\let\oldassumption\assumption
\renewcommand{\assumption}{\oldassumption\normalfont}

\let\oldremark\remark
\renewcommand{\remark}{\oldremark\normalfont}

\let\olddefinition\definition
\renewcommand{\definition}{\olddefinition\normalfont}

\let\oldcorollary\corollary
\renewcommand{\corollary}{\oldcorollary\normalfont}

\let\oldproposition\proposition
\renewcommand{\proposition}{\oldproposition\normalfont}

\let\oldexample\example
\renewcommand{\example}{\oldexample\normalfont}

\let\oldconjecture\conjecture
\renewcommand{\conjecture}{\oldconjecture\normalfont}

\begin{document}

\maketitle

\begin{abstract}
We consider data-driven approaches that integrate a machine learning prediction model within distributionally robust optimization (DRO) given limited joint observations of uncertain parameters and covariates.
Our framework is flexible in the sense that it can accommodate a variety of regression setups and DRO ambiguity sets. 
We investigate asymptotic and finite sample properties of solutions obtained using Wasserstein, sample robust optimization, and phi-divergence-based ambiguity sets within our DRO formulations, and explore cross-validation approaches for sizing these ambiguity sets.
Through numerical experiments, we validate our theoretical results, study the effectiveness of our approaches for sizing ambiguity sets, and illustrate the benefits of our DRO formulations in the limited data regime even when the prediction model is misspecified. \\[0.1in]
\keywords{Data-driven stochastic programming, distributionally robust optimization, Wasserstein distance, phi-divergences, covariates, machine learning, convergence rate, large deviations}
\end{abstract}

\section{Introduction}
\label{sec:intro}

Stochastic programming~\cite{shapiro2009lectures} is a powerful modeling framework for decision-making under uncertainty that finds applications in engineering, operations research, and economics.
A standard formulation of a stochastic program is
\begin{alignat}{2}
\label{eqn:tradsp}
&\uset{z \in \Z}{\min} \: && \expect{c(z,Y)},
\end{alignat}
where $z$ denotes the decision vector, $\Z \subseteq \R^{d_z}$ is the set of feasible decisions, $Y$ denotes a random vector of model parameters with support $\Y \subseteq \R^{d_y}$, and $c: \Z \times \Y \to \R$ is the objective function.
Because the distribution of the random vector~$Y$ is typically unknown, popular data-driven approaches for solving problem~\eqref{eqn:tradsp}, such as sample average approximation (SAA)~\cite{shapiro2009lectures,homem2014monte}, only assume access to a finite sample of~$Y$.
Often, in real-world applications, the random vector~$Y$ (e.g., demand for a new product) can be predicted using knowledge of covariates~$X$ (e.g., web chatter and historical demands for similar existing products).
In our previous work~\cite{kannan2020data}, we investigated extensions of SAA that can incorporate covariate information in problem~\eqref{eqn:tradsp} and studied the asymptotic and finite sample properties of the resulting solutions (see Section~\ref{subsec:ddsaaform}).
Despite its favorable theoretical guarantees~\cite{homem2014monte,shapiro2009lectures,kannan2020data}, a limitation of the SAA approach is that its solutions may exhibit disappointing out-of-sample performance in the small sample size regime~\cite{bertsimas2018robust,esfahani2018data}.

Distributionally robust optimization (DRO)~\cite{rahimian2019distributionally} is a  framework for addressing ambiguity in the distribution of~$Y$.
The DRO counterpart of problem~\eqref{eqn:tradsp} can be formulated as
\begin{alignat}{2}
\label{eqn:traddro}
&\uset{z \in \Z}{\min} \: \uset{Q \in \hP}{\sup} \: && \expectation{Y \sim Q}{c(z,Y)},
\end{alignat}%
where we minimize the worst-case expected objective over an ambiguity set~$\hP$ of distributions.
Several studies have shown that the DRO problem~\eqref{eqn:traddro} can regularize a small-sample SAA of problem~\eqref{eqn:tradsp} and its solutions can mitigate the out-of-sample disappointment of decisions determined using the SAA approach (see the reviews~\cite{kuhn2019wasserstein,rahimian2019distributionally,blanchet2021statistical}).

We introduce a DRO framework for decision-making under uncertainty in the presence of covariate information and study its theoretical properties.
We first consider the setup in~\cite{ban2018big,bertsimas2014predictive,sen2018learning} for incorporating covariate information in problem~\eqref{eqn:tradsp}.
Suppose we have access to joint observations of the random vector $Y$ and random covariates~$X$.
Given a new random observation $X = x$, our goal is to approximate the solution to the {\it conditional stochastic program}
\begin{alignat}{2}
\label{eqn:sp}
v^*(x) &:= \uset{z \in \Z}{\min} \: && \expect{c(z,Y) \mid X = x}. \tag{SP}
\end{alignat}
Applications of this framework include 
shipment planning under demand uncertainty~\cite{bertsimas2014predictive,bertsimas2019dynamic}, where product demands can be predicted using past demands, location, and web search results before making production and inventory decisions, 
and portfolio optimization under market uncertainty~\cite{dou2019distributionally}, where stock prices can be predicted using economic indicators and historical stock data before making investment decisions.

Motivated by applications where we may only have access to limited data, we consider data-driven DRO formulations that incorporate a regression model within a DRO framework in a bid to construct estimators for~\eqref{eqn:sp} with better out-of-sample performance.
Our data-driven DRO formulations are built around the residuals-based SAA formulations that we studied in~\cite{kannan2020data}.
We  define our DRO frameworks in Section~\ref{sec:drsaa}, and analyze their asymptotic and finite sample properties in Sections~\ref{sec:wassdro} and~\ref{sec:samprobustopt}.
Section~\ref{sec:wassdro} focuses on ambiguity sets defined using Wasserstein distances, whereas Section~\ref{sec:samprobustopt} studies a family of ambiguity sets with discrete support.
The analysis in Sections~\ref{sec:wassdro} and~\ref{sec:samprobustopt} builds on our previous analysis of residuals-based SAA formulations in~\cite{kannan2020data}, but makes several new contributions, including analysis of the average rates of convergence of our DRO formulations and finite sample solution guarantees of Wasserstein DRO solutions.
We also investigate data-driven methods for choosing the radii of these ambiguity sets in the presence of covariate observations in Section~\ref{subsec:wass_radius}
without access to samples from the conditional distribution of $Y$ given $X = x$.
Numerical experiments in Section~\ref{sec:computexp} demonstrate the potential benefits of our modular data-driven DRO framework  in the limited data regime.

\subsection{Related work}

We begin by reviewing related work that aims to solve the conditional stochastic program~\eqref{eqn:sp} without using DRO.
Ban and Rudin~\cite{ban2018big} and Bertsimas and Kallus~\cite{bertsimas2014predictive} study policy-based empirical risk minimization and nonparametric regression-based reweighted SAA approaches for solving~\eqref{eqn:sp}.
Bertsimas and Kallus~\cite{bertsimas2014predictive} establish asymptotic optimality of their data-driven decisions, whereas Ban and Rudin~\cite{ban2018big} also present finite sample guarantees in the context of the data-driven newsvendor problem. 
Bazier-Matte and Delage~\cite{bazier2020generalization} explore linear decision rules for a regularized portfolio selection problem given side information.
They derive finite sample and suboptimality performance guarantees for their solutions.
Ban et al.\ \cite{ban2018dynamic} and Sen and Deng~\cite{sen2018learning} use parametric regression methods along with their empirical residuals to generate scenarios of the random variables given covariate information.
Ban et al.\ \cite{ban2018dynamic} prove asymptotic optimality of their decisions for their particular application.
Kannan et al.\ \cite{kannan2020data} introduce two new SAA formulations that use leave-one-out residuals.
They identify conditions under which solutions to their data-driven SAAs possess asymptotic and finite sample guarantees.
Kannan et al.\ \cite{kannan2020data} also review other data-driven approximations to~\eqref{eqn:sp} that do not use DRO.

Solutions to the above approximations to~\eqref{eqn:sp} might display poor out-of-sample performance when we only have access to limited joint data on the random variables and covariates.
DRO offers a structured framework for determining solutions with better out-of-sample performance in such situations.
Next, we review related work that attempts to solve~\eqref{eqn:sp} using DRO.

Hanasusanto and Kuhn~\cite{hanasusanto2013robust} study multi-stage stochastic programs with time series data.
They propose a $\chi^2$-distance-based DRO formulation that uses Nadaraya-Watson regression estimates to approximate value functions, and solve it using an approximate dynamic programming method.
Bertsimas et al. \cite{bertsimas2019dynamic} consider a multi-stage DRO extension of the approach in~\cite{bertsimas2014predictive} using the sample robust optimization method of~\cite{bertsimas2018data}.
They demonstrate asymptotic optimality of their decisions and develop an approximate solution method using linear decision rules.
Bertsimas and Van Parys~\cite{bertsimas2017bootstrap} propose a notion of `bootstrap robustness'.
They define DRO extensions of the Nadaraya-Watson and $k$-nearest neighbors formulations in~\cite{bertsimas2014predictive} using ambiguity sets based on discrepancy measures and study their theoretical properties.

Blanchet et al.\ \cite{blanchet2019robust} and Nguyen et al.\ \cite{nguyen2020distributionally} consider Wasserstein DRO formulations of single-stage stochastic programs arising in statistics or machine learning applications.
Blanchet et al.\ \cite{blanchet2019robust} study how to optimally size their ambiguity sets.
Boskos et al.\ \cite{boskos2020data,boskos2020data2,boskos2021high} explore the construction of Wasserstein ambiguity sets for noisy observations of dynamically evolving random variables with \textit{known dynamics} within a control setting.
They also consider the case when the random variables may only be estimated using noisy observations of outputs of a linear time-varying system.
Similar to our Wasserstein ambiguity sets in Section~\ref{sec:wassdro}, they propose to enlarge the radius of the ambiguity set to account for errors in the estimates of the random variables due to measurement noise.
Dou and Anitescu~\cite{dou2019distributionally} consider a tailored Wasserstein DRO formulation of single-stage stochastic convex programs when the data obeys a linear vector autoregressive model and derive its tractable dual.
Finally, Esteban-P{\'e}rez and Morales~\cite{esteban2020distributionally} construct a Wasserstein DRO extension of~\eqref{eqn:sp} by linking trimmings of probability distributions with the partial mass transportation problem.
They show that their approach naturally produces DRO extensions of formulations based on some nonparametric regression techniques.
They also allow for the available data to be contaminated, and establish asymptotic and finite sample guarantees for their solutions.

We consider a flexible data-driven DRO extension of~\eqref{eqn:sp} that integrates a regression model within a DRO framework.
Our work is similar in spirit to~\cite{dou2019distributionally,esteban2020distributionally}, but we consider more general formulations~\eqref{eqn:sp}, including two-stage stochastic programs, generic regression models, and more general DRO setups, including ones based on Wasserstein distances, sample robust optimization, and phi-divergences.
A key difference between our Wasserstein DRO formulation in Section~\ref{sec:wassdro} and the formulation in~\cite{dou2019distributionally} is that we consider an ambiguity set for the residuals of the regression model, but do not consider one for its coefficients for the sake of generality.
We investigate the theoretical properties of our residuals-based DRO formulations in Sections~\ref{sec:wassdro} and~\ref{sec:samprobustopt}.
The case study in Section~\ref{sec:computexp} demonstrates the modularity benefit of our formulations.

\subsection{Summary of main contributions}

The following summarizes the main contributions of this paper:
\begin{enumerate}
\item We introduce a general residuals-based DRO framework for approximating the solution to problem~\eqref{eqn:sp}
based on the residuals-based SAA framework in~\cite{kannan2020data}.
Our DRO framework is flexible
in the sense that it can accommodate side information effectively using a variety of regression setups and ambiguity sets.
It also seamlessly extends existing DRO formulations that do not utilize covariate information.

\item We study asymptotic optimality, pointwise and average rates of convergence, and finite sample guarantees of solutions determined using Wasserstein ambiguity sets.

\item We consider a family of ambiguity sets with only discrete distributions and study the asymptotic and finite sample properties of resulting solutions.

\item We investigate three data-driven approaches for choosing the radii of ambiguity sets for our residuals-based DRO formulations.
An important difference compared to traditional DRO \textit{without} covariate information is that we cannot assume access to samples from the conditional distribution of $Y$ given $X = x$.
Consequently, radius selection strategies used in traditional DRO (that do not use covariate information) may no longer yield the best radius for our ambiguity sets.
We empirically demonstrate that our new radius selection strategies yield good-quality decisions that work better for our setting than the strategies in the traditional DRO literature.

\item 
Finally, our numerical experiments investigate the effectiveness of proposed approaches for sizing ambiguity sets, validate our theoretical results, and demonstrate the advantages of our data-driven DRO formulations in the limited data regime even when the prediction model is misspecified. 
These experiments also provide insight into the relative performance of different DRO formulations and illustrate the benefit of our framework's modularity.

\end{enumerate}

\paragraph{\bf Notation.}
Let $[n] := \{1,\dots,n\}$, $\norm{\cdot}_p$ denote the $\ell_p$-norm for $p \in [1,+\infty]$, $\proj{S}{v}$ denote the orthogonal projection of $v$ onto a nonempty closed convex set~$S$, and $\delta$ denote the Dirac measure.
We write $\norm{\cdot}$ as shorthand for $\norm{\cdot}_2$.
Let $\mathcal{P}(S)$ denote the space of probability distributions with support contained in the set~$S \subseteq \R^{d_y}$.
Given $Q_1,Q_2 \in \mathcal{P}(S)$, let $\Pi(Q_1,Q_2)$ denote the set of joint distributions with marginals $Q_1$ and $Q_2$. The $p$-Wasserstein distance $d_{W,p}(Q_1,Q_2)$ between $Q_1$ and $Q_2$ with respect to the $\ell_2$-norm\footnote{Our results can be extended to Wasserstein distances defined using $\ell_q$-norms with $q \neq 2$.} is given by
\begin{align*}
d_{W,p}(Q_1,Q_2) &:= \bigl(\inf_{\pi \in \Pi(Q_1,Q_2)} \: \int_{S^2} \norm{y_1 - y_2}^p d\pi(y_1,y_2)\bigr)^{1/p}, \quad \text{if } p \in [1,+\infty), \\
d_{W,\infty}(Q_1,Q_2) &:= \inf_{\pi \in \Pi(Q_1,Q_2)} \: \uset{S \times S}{\pi\text{-ess sup}} \norm{y_1 - y_2},
\end{align*}
where $\pi\text{-ess sup}_{S \times S} \norm{y_1 - y_2} := \inf\{C : \pi(\norm{y_1 - y_2} > C) = 0\}$ 
denotes the essential supremum with respect to the measure~$\pi$.
Let $(S,\Sigma,\mu)$ be a measure space.
Given $q \in [1,+\infty]$, we write $\norm{F}_{L^q}$ to denote the $L^q$-norm of a measurable function $F : S \to \R^{d_F}$, i.e.,
$\norm{F}_{L^q} := \bigl(\int_S \norm{F}^q d\mu \bigr)^{1/q}$.
For any $S \subseteq \R^{d_z}$, let $C(S)$ denote the Banach space of real-valued continuous functions on $S$ equipped with the supremum norm.
For sets $A, B \subseteq \R^{d_z}$, let $\dev{A}{B} := \sup_{v \in A} \text{dist}(v,B)$ denote the deviation of~$A$ from~$B$, where $\text{dist}(v,B) := \inf_{w \in B} \norm{v - w}$.

The abbreviations `a.e.', `a.s.', `LLN', `i.i.d.', and `r.h.s.' are shorthand for `almost everywhere', `almost surely', `law of large numbers', `independent and identically distributed', and `right-hand side'. 
For a random vector~$V$ with probability measure~$P_V$, we write a.e.\ $v \in V$ to denote $P_V$-a.e.\ $v \in V$.
The symbols~$\xrightarrow{p}$,~$\xrightarrow{a.s.}$, and~$\xrightarrow{d}$ denote convergence in probability, almost surely, and in distribution with respect to the probability measure generating the joint data on~$Y$ and~$X$.
For random sequences~$\{V_n\}$ and~$\{W_n\}$, we write $V_n = o_p(W_n)$ and $V_n = O_p(W_n)$ to  convey that~$V_n = R_n W_n$ with $\{R_n\}$ converging in probability to zero, or being bounded in probability, respectively.
We write 
$O(1)$ to denote generic constants and 
$v_n = \Theta(w_n)$ to mean that the sequence $\{v_n\}$ is asymptotically bounded both above and below by the sequence $\{w_n\}$.
We assume that all functions, sets and selections are measurable (see~\cite{vaart1996weak,shapiro2009lectures} for detailed consideration of these issues).

\section{Preliminaries}
\label{sec:prelim}

\subsection{Framework}

We assume throughout that the random vector~$Y$ (commonly referred to as ``dependent variables'') is related to the random covariates~$X$ (commonly referred to as ``independent variables'') as $Y = f^*(X) + \varepsilon$, where $f^*(x) := \expect{Y \mid X = x}$ is the regression function and the random vector~$\varepsilon$ is the associated regression error.
We also assume that the zero-mean errors~$\varepsilon$ are independent\footnote{We investigate extensions of our framework that can adapt to heteroscedasticity in~\cite{kannan2021heteroscedasticity}, where we assume that $Y = f^*(X) + Q^*(X)\varepsilon$ with $X$ and $\varepsilon$ independent (here, $Q^*(X)$ denotes the covariate-dependent covariance matrix of the errors).} of the covariates~$X$, and that $f^*$ is known to belong to a class of functions\footnote{See Remark~\ref{rem:modelclass} at the end of this section for a discussion of the case when $f^* \not\in \F$.} $\F$.
The model class~$\F$ can be infinite-dimensional and can depend on the sample size~$n$.
Let $\Y \subseteq \R^{d_y}$, $\X \subseteq \R^{d_x}$, and $\Xi \subseteq \R^{d_y}$ denote the supports of $Y$, $X$, and~$\varepsilon$, respectively.
Additionally, let $P_{Y \mid X=x}$ denote the conditional distribution of $Y$ given $X = x$ and $P_X$ and $P_{\varepsilon}$ denote the distribution of $X$ and~$\varepsilon$. 
Finally, we assume that the support~$\Y$ is nonempty and convex, which ensures that the orthogonal projection onto~$\Y$ is unique and Lipschitz continuous. If~$\Y$ is not convex (e.g., if it is discrete), one option is to instead project onto its convex hull, $\text{conv}(\Y)$, and replace $\Y$ by $\text{conv}(\Y)$ in our formulations, assumptions, and results.

Under the above assumptions, problem~\eqref{eqn:sp} is equivalent to
\begin{alignat}{2}
\label{eqn:speq}
v^*(x) &= \uset{z \in \Z}{\min} && \left\lbrace g(z;x) := \expect{c(z,f^*(x)+\varepsilon)} \right\rbrace,
\end{alignat}
where the expectation is computed with respect to the distribution $P_{\varepsilon}$ of~$\varepsilon$.
We refer to problem~\eqref{eqn:speq} as \textit{the true problem}.
We assume throughout that the set $\Z$ is nonempty and compact, $\expect{\abs{c(z,f^*(x)+\varepsilon)}} < +\infty$ for each $z \in \Z$ and a.e.\ $x \in \X$, and the function $g(\cdot;x)$ is lower semicontinuous on~$\Z$ for a.e.\ $x \in \X$.
These assumptions ensure that problem~\eqref{eqn:speq} is well-defined  and the set $S^*(x)$ of optimal solutions to problem~\eqref{eqn:speq} is nonempty for a.e.\ $x \in \X$.

\subsection{Review of data-driven SAA formulations}
\label{subsec:ddsaaform}

We now summarize the residuals-based SAA formulations considered in~\cite{kannan2020data}.
Let~$\D_n := \{(y^i,x^i)\}_{i=1}^{n}$ denote the joint observations of $(Y,X)$ and $\{\varepsilon^i\}_{i=1}^{n}$, with $\varepsilon^i := y^i - f^*(x^i)$, $i \in [n]$, denote the corresponding realizations of the errors.
If we know the regression function~$f^*$, then we can construct the following \textit{full-information SAA} (FI-SAA) to problem~\eqref{eqn:speq} using the data~$\D_n$:
\begin{align}
\label{eqn:fullinfsaa}
&\uset{z \in \Z}{\min} \Big\{ g^*_n(z;x) := \dfrac{1}{n} \displaystyle\sum_{i=1}^{n} c(z,f^*(x)+\varepsilon^i) \Big\}.
\end{align} 
Because~$f^*$ is unknown, we first estimate it by $\hf_n$ using a regression method on the data~$\D_n$.
We then use $\hf_n$ and its residuals on the training data $\heps^i_{n} := y^i - \hat{f}_n(x^i)$, $i \in [n]$, to construct the following \textit{empirical residuals-based SAA} (ER-SAA) to problem~\eqref{eqn:speq}:
\begin{alignat}{2}
\label{eqn:app}
\hv^{ER}_n(x) &:= \uset{z \in \Z}{\min} && \Big\{ \hg^{ER}_n(z;x) := \dfrac{1}{n}\displaystyle\sum_{i=1}^{n} c\bigl(z,\proj{\Y}{\hf_n(x) + \heps^i_{n}}\bigr) \Big\}.
\end{alignat}
In contrast with~\cite{kannan2020data}, we project the points $\{\hf_n(x) + \heps^i_{n}\}_{i\in [n]}$ onto the support~$\Y$ in this work. 
This projection step may be helpful when the ER-SAA scenarios $\{\hf_n(x) + \heps^i_{n}\}_{i \in [n]}$ lie outside the support $\Y$ even though the ``true'' FI-SAA scenarios $\{f^*(x) + \varepsilon^i\}_{i \in [n]}$ are elements of $\Y$.
It also ensures that the Wasserstein and sample robust optimization-based DRO formulations considered in Section~\ref{sec:drsaa} are tractable under suitable assumptions on the true problem~\eqref{eqn:speq}. 
This is because the ambiguity set of these DRO formulations then becomes a ball around the ER-SAA distribution defined below, with respect to a suitable Wasserstein metric (cf.\ Section~4 of~\cite{esfahani2018data}). We stick with this modification of the ER-SAA formulation~\eqref{eqn:app} throughout for uniformity.

When the sample size $n$ is small relative to the complexity of the regression method, the empirical residuals $\{\heps^i_{n}\}_{i=1}^{n}$ may be optimistically biased and provide a poor estimate of the samples $\{\varepsilon^i\}_{i=1}^{n}$ of~$\varepsilon$.
This motivated our construction in~\cite{kannan2020data} of two alternative SAA formulations that instead use leave-one-out (jackknife) residuals to construct scenarios of $Y$ given $X = x$.

Let $P^*_n(x)$ denote the \textit{true empirical distribution} of $Y$ given $X = x$ corresponding to the FI-SAA problem~\eqref{eqn:fullinfsaa}
and $\hat{P}^{ER}_n(x)$ denote the \textit{estimated empirical distribution} 
corresponding to the ER-SAA problem~\eqref{eqn:app}, i.e., 
\[
P^*_n(x) := \frac{1}{n} \sum_{i=1}^{n} \delta_{f^*(x)+\varepsilon^i}, \quad\quad \hat{P}^{ER}_n(x) := \dfrac{1}{n} \sum_{i=1}^{n} \delta_{\proj{\Y}{\hf_n(x) + \heps^i_{n}}}.
\]
A main component of the analysis conducted in this paper is controlling the distance between the estimated empirical distribution~$\hat{P}^{ER}_n(x)$ and the true empirical distribution~$P^*_n(x)$.
To enable this, note that the Lipschitz continuity of orthogonal projections\footnote{For any $u,v \in \R^{d_y}$, $\norm{\proj{\Y}{u} - \proj{\Y}{v}} \leq \norm{u-v}$.} implies that for each $x \in \X$
\begin{align}
\label{eqn:projlipschitz}
\norm{\proj{\Y}{\hf_n(x) + \heps^i_{n}} - (f^*(x) + \varepsilon^i)} \leq \norm{\teps^{i}_{n}(x)}, \quad \forall i \in [n],
\end{align}
where
$\teps^{i}_{n}(x) := (\hf_n(x) + \heps^i_{n}) - (f^*(x) + \varepsilon^i) = \bigl( \hf_n(x) - f^*(x) \bigr) + \bigl( f^*(x^i) - \hf_n(x^i) \bigr)$.
Note that $\teps^{i}_{n}(x)$ equals the sum of the \textit{prediction error} at the new covariate realization $x \in \X$ and the \textit{estimation error} at the training point $x^i \in \X$.

\begin{remark}
\label{rem:modelclass}
Although we assume that the regression function~$f^*$ belongs to the model class~$\F$ to establish our theoretical guarantees, our data-driven approximations of~\eqref{eqn:speq} are well defined even when $f^* \not\in \F$, i.e., when the regression model is misspecified.
In this setting, under mild assumptions the regression estimates $\hf_n$ converge  to the best approximation to $f^*$ in the model class $\F$, denoted by $\bar{f}$, and
residuals-based SAA and DRO formulations then yield solutions converging to the optimal solution of the problem
\[
\uset{z \in \Z}{\min} \: \dfrac{1}{n} \displaystyle\sum_{i=1}^{n} c(z,\bar{f}(x)+\bar{\varepsilon}^i),
\]
where $\bar{\varepsilon}^i := f^*(x^i) - \bar{f}(x^i) + \varepsilon^i$.
Therefore, we can replace $f^*$ by $\bar{f}$ and $\{\varepsilon^i\}$ by $\{\bar{\varepsilon}^i\}$ in our assumptions and results to characterize the asymptotic and finite sample properties of our data-driven approximations in this case.
\end{remark}

\section{Residuals-based DRO formulations}
\label{sec:drsaa}

We consider the following DRO extension of the data-driven SAA formulations reviewed in Section~\ref{subsec:ddsaaform} to approximate the solution to problem~\eqref{eqn:speq}:
\begin{alignat}{2}
\label{eqn:dro}
\hv^{DRO}_n(x) &= \uset{z \in \Z}{\min} \: \uset{Q \in \hP_n(x)}{\sup} \: && \expectation{Y \sim Q}{c(z,Y)},
\end{alignat}
where $\hP_n(x)$ is a data-driven ambiguity set for the distribution of $Y$ given $X = x$ that is centered at $\hat{P}^{ER}_n(x)$.
Let $\hz^{DRO}_n(x)$ denote an optimal solution to problem~\eqref{eqn:dro} and $\hS^{DRO}_n(x)$ denote its set of optimal solutions.
We assume throughout that the objective function of problem~\eqref{eqn:dro} is real-valued and lower semicontinuous on~$\Z$ for each $x \in \X$.
This ensures that its optimal solution set~$\hS^{DRO}_n(x)$ is nonempty for each $x \in \X$.

We seek to derive DRO formulations~\eqref{eqn:dro} that obtain a solution~$\hz^{DRO}_n(x)$ with good out-of-sample performance $g(\hz^{DRO}_n(x);x)$ for relatively small sample sizes~$n$.
To support our investigation of such formulations, we consider different desirable properties they may have.
Given a risk level $\alpha \in (0,1)$, we wish to construct the ambiguity set $\hP_n(x)$ such that one or more of the following properties hold for a.e.\ $x \in \X$ (cf.\ \cite{bertsimas2018robust,esfahani2018data}):
\begin{enumerate}
\item  \textbf{Consistency and asymptotic optimality:} the optimal value $\hv^{DRO}_n(x)$ and solution $\hz^{DRO}_n(x)$ of the residuals-based DRO problem~\eqref{eqn:dro} satisfy
\[
\hv^{DRO}_n(x) \xrightarrow{p} v^*(x), \:\: \text{dist}(\hz^{DRO}_n(x),S^*(x)) \convinprob 0, \:\: g(\hz^{DRO}_n(x);x) \convinprob v^*(x).
\]

\item \textbf{Rate of convergence:} for some constant $r \in (0,1]$ (ideally close to one), the optimal value $\hv^{DRO}_n(x)$ and solution $\hz^{DRO}_n(x)$ satisfy\footnote{In special cases (e.g., smooth unconstrained problems, see~\cite[Section~5]{gotoh2021calibration}), it may be possible to establish the sharper convergence rate $\abs*{g(\hz^{DRO}_n(x);x) - v^*(x)} = O_p\big(n^{-r}\big)$.}
\[
\abs*{\hv^{DRO}_n(x) - v^*(x)} = O_p\big(n^{-r/2}\big), \quad \abs*{g(\hz^{DRO}_n(x);x) - v^*(x)} = O_p\big(n^{-r/2}\big).
\] 

\item \textbf{Finite sample certificate guarantee:} the optimal value $\hv^{DRO}_n(x)$ provides the following certificate on the out-of-sample cost of $\hz^{DRO}_n(x)$:
\[
\prob{g(\hz^{DRO}_n(x);x) \leq \hv^{DRO}_n(x)} \geq 1 - \alpha.
\]

\end{enumerate}
We would also like the solution~$\hz^{DRO}_n(x)$ to possess the following guarantee:
\begin{enumerate}
\setcounter{enumi}{3}
\item \textbf{Finite sample solution guarantee:} for a.e.\ $x \in \X$ and any $\eta > 0$, there exist positive constants $\Gamma(\eta,x)$ and $\gamma(\eta,x)$ such that the solution~$\hz^{DRO}_n(x)$ of the DRO problem~\eqref{eqn:dro} with a suitable specification of the radius of the ambiguity set $\hP_n(x)$ satisfies
\[
\prob{\text{dist}(\hz^{DRO}_n(x),S^*(x)) \geq \eta} \leq \Gamma(\eta,x) \exp(-n\gamma(\eta,x)).
\]
\end{enumerate}
Finally, we would also like problem~\eqref{eqn:dro} to be efficiently solvable in practice.
Although our asymptotic guarantees are stated in terms of convergence in probability, they can be naturally extended to consider almost sure convergence under stronger assumptions on problem~\eqref{eqn:speq} and the regression estimate~$\hf_n$.

We call problem~\eqref{eqn:dro} with the ambiguity set~$\hP_n(x)$ centered at~$\hat{P}^{ER}_n(x)$ the \textit{empirical residuals-based DRO} (ER-DRO) problem.
While in this paper we focus our attention on ER-DRO formulations, note that the ambiguity set~$\hP_n(x)$ can also be centered at the estimated empirical distributions corresponding to its jackknife-based counterparts introduced in \cite{kannan2020data}.
The analysis in~\cite[Appendix~EC.1]{kannan2020data} can be used to extend this paper's results for ER-DRO to its jackknife-based variants.

In the remainder of this work, we focus on the use of the following data-driven ambiguity sets $\hP_n(x)$ in the construction of ER-DRO problem~\eqref{eqn:dro}.
Unlike the classical DRO setting~\cite{rahimian2019distributionally}, we allow the radius of these ambiguity sets~$\hP_n(x)$ to depend not only on the sample size~$n$ and the risk level~$\alpha$ that, e.g., shows up in the finite sample certificate, 
but also on the covariate realization~$x \in \X$; see~$\zeta_n(x)$ and~$\mu_n(x)$ below.
We often omit the dependence of the radius on $\alpha$ to simplify notation.
\begin{enumerate}
\item Wasserstein-based ambiguity sets (cf.\ \cite{pflug2007ambiguity,esfahani2018data,gao2016distributionally}): given radius $\zeta_n(x) \geq 0$ and order $p \in [1,+\infty]$, set
\[
\hP_n(x) = \big\{Q \in \mathcal{P}(\Y) : d_{W,p}(Q,\hat{P}^{ER}_n(x)) \leq \zeta_n(x)\big\}.
\]

\item Sample robust optimization-based ambiguity sets (cf.\ \cite{xu2012distributional,bertsimas2018data}): given radius~$\mu_n(x) \geq 0$ and parameter $p \in [1,+\infty]$, set\footnote{We use $\mu_n(x)$ to avoid a clash with the notation~$\zeta_n(x)$ for the radius of ambiguity sets with the same support as $\hat{P}^{ER}_n(x)$. Having different notation for these two radii will prove useful in our unified analysis of the corresponding ER-DRO problems in Section~\ref{sec:samprobustopt}.}
\[
\displayindent0pt
\displaywidth\textwidth
\hP_n(x) = \Big\{Q = \frac{1}{n}\sum_{i=1}^{n} \delta_{\bar{y}^i} : \norm{\bar{y}^i - \proj{\Y}{\hf_n(x) + \heps^i_{n}}}_p \leq \mu_n(x), \bar{y}^i \in \Y, \forall i \in [n]\Big\}.
\]
We focus on ambiguity sets constructed using $p = 2$ to keep the exposition simple, but our analysis also extends to ambiguity sets with $p \neq 2$.

\item Ambiguity sets with the same support as $\hat{P}^{ER}_n(x)$ (cf.\ \cite{ben2013robust,bayraksan2015data}, for instance): 
given radius $\zeta_n(x) \geq 0$, set
\[
\hP_n(x) = \Big\{Q = \sum_{i=1}^{n} p_i\delta_{\proj{\Y}{\hf_n(x) + \heps^i_{n}}} : p \in \mathfrak{P}_n(x;\zeta_n(x))\Big\},
\]
where $\mathfrak{P}_n(x;\zeta_n(x))$ is a generic ambiguity set for the $n$-dimensional vector of probabilities~$p$.
We focus on sets~$\mathfrak{P}_n(x;\zeta_n(x))$ that satisfy for each $x \in \X$
\begin{equation}
\begin{aligned}
\label{eqn:ambiguitysetconsistency}
p \in \R^n_+ \:\: \text{and} \:\: \sum_{i=1}^{n} p_i = 1, \quad \forall p \in \mathfrak{P}_n(x;\zeta_n(x)), & \\
\lim_{\zeta \downarrow 0} \mathfrak{P}_n(x;\zeta) = \mathfrak{P}_n(x;0) = \biggl\{\biggl(\frac{1}{n},\frac{1}{n},\dots,\frac{1}{n}\biggr)\biggr\}. &
\end{aligned}
\end{equation}
The above family of ambiguity sets---that use the same support as $\hat{P}^{ER}_n(x)$---result in tractable ER-DRO formulations~\eqref{eqn:dro} under milder assumptions on the true problem~\eqref{eqn:speq} compared to Wasserstein and sample robust optimization ambiguity sets, which go beyond the support of~$\hat{P}^{ER}_n(x)$.

\end{enumerate}

We now provide two examples of the last category of ambiguity sets. Appendix~\ref{app:ambiguitysetrate} includes a third example based on mean-upper semideviations.

\begin{example}
\label{exm:cvar}
CVaR-based ambiguity set~\cite{rockafellar2000optimization,shapiro2009lectures}: given radius $\zeta_n(x) \in [0,1)$, set
\[
\Pf_n(x;\zeta_n(x)) := \bigg\{p \in \R^{n}_+ : \sum_{i=1}^{n} p_i = 1, \: p_i \leq \frac{1}{n(1-\zeta_n(x))}, \: \forall i \in [n] \bigg\}.
\]
Observe that~$\zeta_n(x)$ enters the ambiguity set~$\Pf_n(x;\zeta_n(x))$ through the CVaR risk parameter.
\end{example}

\begin{example}
\label{exm:phidivergence}
Phi-divergence-based ambiguity sets~\cite{ben2013robust,bayraksan2015data}: Let $\phi:\R_+ \to \overline{\R}_+$ be a lower semicontinuous, convex phi-divergence function with a unique minimum at $1$ and $\phi(1) = 0$.
Given radius $\zeta_n(x) \geq 0$, define~$\hP_n(x)$ using 
\[
\mathfrak{P}_n(x;\zeta_n(x)) := \bigg\{p \in \R^n_+ : \sum_{i=1}^{n} p_i = 1, \: \frac{1}{n} \sum_{i=1}^{n} \phi(n p_i) \leq \zeta_n(x)\bigg\}.
\]
Particular instances include Kullback Leibler divergence, variation distance, and Hellinger distance-based ambiguity sets.
\end{example}

In the next section, we investigate the theoretical properties of using Wasserstein ambiguity sets within the ER-DRO problem.
In Section~\ref{sec:samprobustopt}, we present a unified analysis of the theoretical properties of using both sample robust optimization ambiguity sets and ambiguity sets with the same support as $\hat{P}^{ER}_n(x)$.
Hereafter, we often write $\hP_n(x;\zeta_n(x))$ instead of $\hP_n(x)$ to make its dependence on the radius~$\zeta_n(x)$ explicit.
We also write~$\zeta_n(\alpha,x)$ instead of~$\zeta_n(x)$ when we want to emphasize the dependence of the radius on the risk level~$\alpha$.

\section{Wasserstein-based ambiguity sets}
\label{sec:wassdro}

We now establish asymptotic optimality, rates of convergence, and finite sample guarantees for ER-DRO formulations defined using $p$-Wasserstein distance-based ambiguity sets with $p \in [1,+\infty)$.
Section~\ref{sec:samprobustopt} presents analysis for ambiguity sets defined using the $\infty$-Wasserstein distance by exploiting a link with sample robust optimization~\cite{bertsimas2018data}.
Sections~4.1 and~5 of~\cite{esfahani2018data} and Section~2.2 of~\cite{kuhn2019wasserstein} identify conditions under which the resulting ER-DRO formulation~\eqref{eqn:dro} is computationally tractable.
References~\cite{gao2016distributionally,hanasusanto2018conic,bansal2018decomposition} also consider solution approaches for the setting where problem~\eqref{eqn:speq} is a two-stage stochastic program.

We begin with a light-tail assumption on the distribution~$P_{\varepsilon}$ of the errors~$\varepsilon$.

\begin{assumption}
\label{ass:lighttailerr}
There is a constant $a > p$ such that $\expect{\exp(\norm{\varepsilon}^a)} < +\infty$.
\end{assumption}

Assumption~\ref{ass:lighttailerr} (cf.\ \citep[Assumption~3.3]{esfahani2018data}) may not hold for sub-exponential distributions.
Additionally, when $p \geq 2$, it requires the tails of $\varepsilon$ to decay at a faster rate than Gaussian tails.
However, sub-Gaussian errors (see Definition~\ref{def:subgaussian} below) can be handled using $p \in [1,2)$.
Our first result identifies sufficient conditions under which sub-Gaussian errors satisfy Assumption~\ref{ass:lighttailerr} for $p \in [1,2)$.

\begin{definition}
\label{def:subgaussian}
A random vector $V \in \R^{d_v}$ is said to be sub-Gaussian with variance proxy $\sigma^2$ if $\expect{V} = 0$ and
\[
\expect{\exp(s\tr{u}V)} \leq \exp(0.5\sigma^2s^2), \:\: \forall s \in \R \text{ and } u \in \R^{d_v} \text{ s.t. } \norm{u} = 1.
\]
\end{definition}

Definition~\ref{def:subgaussian} implies that the class of sub-Gaussian random vectors includes zero-mean Gaussian random vectors.

\begin{proposition}
\label{prop:subgaussian}
Suppose $\varepsilon$ is a sub-Gaussian random vector with independent components and variance proxy $\sigma^2$.
Then $\expect{\exp(\norm{\varepsilon}^a)} < +\infty$, $\forall a \in (1,2)$.
\end{proposition}
\begin{proof}
{See Appendix~\ref{app:subgaussian}.}
\end{proof}

Next, we make a finite sample assumption on the regression estimate~$\hf_n$.

\begin{assumption}
\label{ass:reglargedevwass}
The regression estimate~$\hf_n$ possesses the following finite sample property:
for a.e.\ $x \in \X$ and any risk level $\alpha \in (0,1)$, there exists a constant $\kappa_{p,n}(\alpha,x) > 0$ such that 
\begin{align*}
\mathbb{P}\big\{\norm{f^*(x) - \hf_n(x)}^p > \kappa^p_{p,n}(\alpha,x)\big\} &\leq \alpha, \quad \text{and}\\
\mathbb{P}\biggl\{\dfrac{1}{n} \displaystyle\sum_{i=1}^{n} \norm{f^*(x^i) - \hf_n(x^i)}^p > \kappa^p_{p,n}(\alpha,x)\biggr\} &\leq \alpha.
\end{align*}
\end{assumption}

We use same the constant $\kappa_{p,n}(\alpha,x)$ for both the prediction error at the new covariate realization $x \in \X$ and the power-mean estimation error on the training data points $\{x^i\}_{i=1}^{n}$ to keep the notation simple even though the latter does not depend on $x$.
Appendix~EC.3 of~\cite{kannan2020data} identifies conditions under which parametric regression methods such as ordinary least squares (OLS) and Lasso regression satisfy Assumption~\ref{ass:reglargedevwass} for the case $p = 2$ with constants $\kappa^2_{2,n}(\alpha,x) = O(n^{-1} \log(\alpha^{-1}))$; {we omit the dependence on $x$ here for simplicity}.
Nonparametric regression methods, on the other hand, typically only satisfy Assumption~\ref{ass:reglargedevwass} with constants 
$\kappa^p_{p,n}(\alpha,x) = O(n^{-1}\log(\alpha^{-1}))^{O(1)/d_x}$.
Similar bounds readily hold for $p \neq 2$, e.g., if the support~$\X$ of the covariates is compact.
If Assumption~\ref{ass:reglargedevwass} holds for $p = 2$, the power mean inequality implies that it also holds for any $p \in [1,2)$ with $\kappa_{p,n}(\alpha,x) = \kappa_{2,n}(\alpha,x)$.

We make the light-tail Assumption~\ref{ass:lighttailerr} on the distribution~$P_{\varepsilon}$ of the errors~$\varepsilon$ to invoke the concentration inequality in Lemma~\ref{lem:wassmeasureconc} for the true empirical distribution~$P^*_n(x)$.
Throughout, we assume $p \neq d_y/2$ for a slightly simpler form of this concentration inequality; see~\citep[Theorem~2]{fournier2015rate} for the case $p = d_y/2$.
Lemma~\ref{lem:wassmeasureconc} also applies to non-i.i.d.\ data~$\D_n$ such as time series data (cf.\ \cite{dou2019distributionally}).

\begin{lemma}[Theorem~2 of~\cite{fournier2015rate}]
\label{lem:wassmeasureconc}
Suppose Assumption~\ref{ass:lighttailerr} holds, $p \neq d_y/2$, and the samples $\{\varepsilon^{i}\}_{i=1}^{n}$ are i.i.d.
Then, for all $\kappa > 0$, $n \in \mathbb{N}$, and $x \in \X$
\[
\mathbb{P}\big\{d_{W,p}( P^*_n(x), P_{Y \mid X=x} ) \geq \kappa\big\} \leq \begin{cases} O(1) \exp(-O(1) n \kappa^{\max\{d_y/p,2\}}) &\mbox{if } \kappa \leq 1 \\
O(1) \exp(-O(1) n \kappa^{a/p}) & \mbox{if } \kappa > 1 \end{cases}.
\]
\end{lemma}

{The $O(1)$ constants in Lemma~\ref{lem:wassmeasureconc} only depend on $a$, $d_y$, and $\expect{\exp(\norm{\varepsilon}^a)}$ (see~\citep[Theorem~2]{fournier2015rate}).}
We require a few intermediate results before we can establish a finite sample certificate guarantee for Wasserstein ER-DRO estimators in Theorem~\ref{thm:wassfinitesampcert} (cf.\ \cite[Theorem~19]{kuhn2019wasserstein}).
The first result bounds the $p$-Wasserstein distance between the estimated empirical distribution $\hat{P}^{ER}_n(x)$ and the conditional distribution $P_{Y \mid X = x}$ of $Y$ given $X = x$.

\begin{lemma}
\label{lem:wass_dist_bound}
For each $x \in \X$
\[
d_{W,p}( \hat{P}^{ER}_n(x), P_{Y \mid X=x} ) \leq \biggl(\dfrac{1}{n} \displaystyle\sum_{i=1}^{n} \norm{\teps^{i}_{n}(x)}^{{p}}\biggr)^{1/p} + d_{W,p}( P^*_n(x), P_{Y \mid X=x} ).
\]
\end{lemma}
\begin{proof}
The triangle inequality for the $p$-Wasserstein distance yields
\[
d_{W,p}( \hat{P}^{ER}_n(x), P_{Y \mid X=x} ) \leq d_{W,p}( \hat{P}^{ER}_n(x), P^*_n(x) ) + d_{W,p}( P^*_n(x), P_{Y \mid X=x} ).
\]
The stated result then follows from the definition of the $p$-Wasserstein distance
and inequality~\eqref{eqn:projlipschitz} since
\begin{align*}
d_{W,p}( \hat{P}^{ER}_n(x), P^*_n(x) ) &\leq \biggl(\frac{1}{n} \sum_{i=1}^{n} \norm{\proj{\Y}{\hf_n(x) + \heps^i_{n}} - (f^*(x) + \varepsilon^i)}^p\biggr)^{1/p} \\
&\leq \biggl(\frac{1}{n} \sum_{i=1}^{n} \norm{\teps^{i}_{n}(x)}^p\biggr)^{1/p}. \qedhere
\end{align*}
\end{proof}

The next result bounds the power mean deviation $\bigl(\frac{1}{n}\sum_{i=1}^{n} \norm{\teps^{i}_{n}(x)}^p\bigr)^{1/p}$.

\begin{lemma}
\label{lem:boundregerror}
For each $x \in \X$
\[
\biggl(\dfrac{1}{n}\displaystyle\sum_{i=1}^{n} \norm{\teps^{i}_{n}(x)}^p\biggr)^{1/p} \leq \norm{f^*(x) - \hf_n(x)} + \biggl(\dfrac{1}{n} \displaystyle\sum_{i=1}^{n} \norm{f^*(x^i) - \hf_n(x^i)}^p\biggr)^{1/p}.
\]
\end{lemma}
\begin{proof}
We have from the definition of $\teps^{i}_{n}(x)$ that
\begin{align*}
\biggl(\dfrac{1}{n}\displaystyle\sum_{i=1}^{n} \norm{\teps^{i}_{n}(x)}^p\biggr)^{1/p} \leq& \biggl(\dfrac{1}{n}\displaystyle\sum_{i=1}^{n} \bigl(\norm{f^*(x) - \hf_n(x)} + \norm{f^*(x^i) - \hf_n(x^i)}\bigr)^p\biggr)^{1/p} \\
\leq& \norm{f^*(x) - \hf_n(x)} + \biggl(\dfrac{1}{n} \displaystyle\sum_{i=1}^{n} \norm{f^*(x^i) - \hf_n(x^i)}^p\biggr)^{1/p},
\end{align*}
where the first step follows from the triangle inequality for the $\ell_2$-norm, and the second step follows from the triangle inequality for the $\ell_p$-norm.
\end{proof}

{We also require the following simple inequality.}

\begin{lemma}
\label{lem:probineq}
{Let $V$ and $W$ be random variables and $c_1, c_2 \in \R$.
Then}
\[
\mathbb{P}(V+W > c_1 + c_2) \leq \mathbb{P}(V > c_1) + \mathbb{P}(W > c_2).
\]
\end{lemma}

\noindent {We are now ready to} derive a finite sample guarantee for $\bigl(\frac{1}{n}\sum_{i=1}^{n} \norm{\teps^{i}_{n}(x)}^p\bigr)^{1/p}$.

\begin{lemma}
\label{lem:wassfinitesamp}
Suppose Assumption~\ref{ass:reglargedevwass} holds and $\alpha \in (0,1)$.
Then for a.e.\ $x \in \X$
\[
\mathbb{P}\bigg\{\biggl(\dfrac{1}{n}\displaystyle\sum_{i=1}^{n} \norm{\teps^{i}_{n}(x)}^p\biggr)^{1/p} > 2\kappa_{p,n}\Bigl(\frac{\alpha}{4},x\Bigr)\bigg\} \leq \frac{\alpha}{2}.
\] 
\end{lemma}
\begin{proof}
We have for a.e.\ $x \in \X$
\begin{align*}
& \mathbb{P}\bigg\{\biggl(\dfrac{1}{n}\displaystyle\sum_{i=1}^{n} \norm{\teps^{i}_{n}(x)}^p\biggr)^{1/p} > 2\kappa_{p,n}\Bigl(\frac{\alpha}{4},x\Bigr)\bigg\} \\
\leq & \mathbb{P}\bigg\{\norm{f^*(x) - \hf_n(x)} + \biggl(\dfrac{1}{n} \displaystyle\sum_{i=1}^{n} \norm{f^*(x^i) - \hf_n(x^i)}^p\biggr)^{1/p} > 2\kappa_{p,n}\Bigl(\frac{\alpha}{4},x\Bigr)\bigg\} \\
\leq & \prob{\norm{f^*(x) - \hf_n(x)} > \kappa_{p,n}\Bigl(\frac{\alpha}{4},x\Bigr)} \\
& \qquad + \mathbb{P}\bigg\{\dfrac{1}{n} \displaystyle\sum_{i=1}^{n} \norm{f^*(x^i) - \hf_n(x^i)}^p > \kappa^p_{p,n}\Bigl(\frac{\alpha}{4},x\Bigr)\bigg\} \\
\leq & \frac{\alpha}{4} + \frac{\alpha}{4} = \frac{\alpha}{2},
\end{align*}
where the first step follows by Lemma~\ref{lem:boundregerror}, {the second step follows from Lemma~\ref{lem:probineq}},
and the last step holds by Assumption~\ref{ass:reglargedevwass}.
\end{proof}

To establish our asymptotic and finite sample guarantees,
we enlarge the radius of the Wasserstein ambiguity set that is used in the absence of covariate information~\cite{esfahani2018data,kuhn2019wasserstein}.  This enlargement accounts for the error in estimating the regression function~$f^*$.
In particular, for a given covariate realization $x \in \X$ and risk level $\alpha \in (0,1)$, we use 
\begin{align}
\label{eqn:wassradius}
\zeta_n(\alpha,x) := \kappa^{(1)}_{p,n}(\alpha,x) +  \kappa^{(2)}_{p,n}(\alpha)
\end{align}
as the radius of the ambiguity set,
where $\kappa^{(1)}_{p,n}(\alpha,x) := 2\kappa_{p,n}\bigl(\frac{\alpha}{4},x\bigr)$ and
\begin{align*}
\kappa^{(2)}_{p,n}(\alpha) &:= \begin{cases} \left( \frac{O(1)\log(O(1)\alpha^{-1})}{n} \right)^{\min\{p/d_y,1/2\}} &\mbox{if } n \geq O(1) \log(O(1) \alpha^{-1}) \\
\left( \frac{O(1)\log(O(1)\alpha^{-1})}{n} \right)^{p/a} & \mbox{if } n < O(1) \log(O(1) \alpha^{-1}) \end{cases}.
\end{align*}
The constants $a$ and $\kappa_{p,n}$ above are defined in Assumptions~\ref{ass:lighttailerr} and~\ref{ass:reglargedevwass}.
{The term $\kappa^{(2)}_{p,n}(\alpha)$ is obtained by setting the r.h.s.\ of the inequality in Lemma~\ref{lem:wassmeasureconc} to $\alpha/2$.}
While this choice of~$\zeta_n$ helps us derive our theoretical guarantees, it involves unknown constants and is often conservative in practice (see Remark~\ref{rem:wassdrorate}).
We investigate practical data-driven approaches for choosing~$\zeta_n$
in Section~\ref{subsec:wass_radius}.

\begin{theorem}[Finite sample certificate guarantee]
\label{thm:wassfinitesampcert}
Suppose Assumptions~\ref{ass:lighttailerr} and~\ref{ass:reglargedevwass} hold, $\alpha \in (0,1)$ is a given risk level, and the samples $\{\varepsilon^{i}\}_{i=1}^{n}$ of the errors are i.i.d. 
Then, for a.e.\ $x \in \X$, the finite sample certificate guarantee $\prob{g(\hz^{DRO}_n(x);x) \leq \hv^{DRO}_n(x)} \geq 1 - \alpha$ holds for the ER-DRO problem~\eqref{eqn:dro} with radius $\zeta_n(\alpha,x)$ of the ambiguity set $\hP_n(x;\zeta_n(\alpha,x))$ specified by equation~\eqref{eqn:wassradius}.
\end{theorem}
\begin{proof}
Lemma~\ref{lem:wassfinitesamp} and Lemma~\ref{lem:wassmeasureconc}
imply that
\begin{align*}
\mathbb{P}\bigg\{\biggl(\frac{1}{n}\sum_{i=1}^{n} \norm{\teps^{i}_{n}(x)}^p\biggr)^{1/p} > \kappa^{(1)}_{p,n}(\alpha,x)\bigg\} &\leq \frac{\alpha}{2}, \quad \text{for a.e. } x \in \X, \\
\mathbb{P}\big\{d_{W,p}( P^*_n(x), P_{Y \mid X=x} ) > \kappa^{(2)}_{p,n}(\alpha)\big\} &\leq \frac{\alpha}{2}, \quad \forall x \in \X.
\end{align*}
Consequently, equation~\eqref{eqn:wassradius},
Lemma~\ref{lem:wass_dist_bound} {and Lemma~\ref{lem:probineq}} 
imply that
\[
\mathbb{P}\big\{d_{W,p}( \hat{P}^{ER}_n(x), P_{Y \mid X=x} ) > \zeta_n(\alpha,x)\big\} \leq \alpha \:\: \text{for a.e. } x \in \X.
\]
The stated result follows from the definition of the ER-DRO problem~\eqref{eqn:dro}.
\end{proof}

We now make the following assumption along the lines of~\cite{esfahani2018data,kuhn2019wasserstein,gao2016distributionally} to show in Theorem~\ref{thm:wassaympconsist} that solutions to the ER-DRO problem~\eqref{eqn:dro} with radii $\zeta_n(\alpha_n,x)$ 
are 
asymptotically optimal for a suitable sequence of risk levels $\{\alpha_n\}$.

\begin{assumption}
\label{ass:growthrate}
The function $c(\cdot,Y)$ is lower semicontinuous on $\Z$ for each $Y \in \Y$ and the function~$c(z,\cdot)$ is continuous on~$\Y$ for each~$z \in \Z$. Furthermore, there exists a constant $B_{c,p} \geq 0$ such that 
\[
\abs{c(z,Y)} \leq B_{c,p} (1 + \norm{Y}^p), \quad \forall z \in \Z, \: Y \in \Y.
\]
\end{assumption}

We also make either of the following assumptions on the function~$c$ to establish a rate of convergence of the ER-DRO estimator in Theorem~\ref{thm:wassconvrate}.

\begin{assumption}
\label{ass:lipschitzy}
For each $z \in \Z$, the function $c(z,\cdot)$ is Lipschitz continuous on $\Y$ with Lipschitz constant~$L_1(z)$. 
\end{assumption}

\begin{assumption}
\label{ass:lipschitzy2}
The Wasserstein order is $p \geq 2$.
Furthermore, for each $z \in \Z$, the function $c(z,\cdot)$ is differentiable on~$\Y$ with $\expect{\norm{\nabla c(z,Y)}^2} < +\infty$ and
\[
\norm{\nabla c(z,\bar{y}) -\nabla c(z,y)} \leq L_2(z) \norm{\bar{y} - y}, \quad \forall y, \bar{y} \in \Y.
\]
\end{assumption}

Assumptions~\ref{ass:growthrate},~\ref{ass:lipschitzy}, and~\ref{ass:lipschitzy2} hold for broad classes of stochastic programs, including two-stage stochastic mixed-integer linear programs (MIPs) with continuous recourse~\citep[Appendix~EC.2]{kannan2020data}.

\begin{assumption}
\label{ass:wassrisklevelseq}
The sequence of risk levels $\{\alpha_n\} \subset (0,1)$ satisfies $\sum_{n} \alpha_n < \infty$ and $\uset{n \to \infty}{\lim} \zeta_n(\alpha_n,x) = 0$ for a.e.\ $x \in \X$ with the radius~$\zeta_n$ defined in~\eqref{eqn:wassradius}.
\end{assumption}

{We have the following useful result.}

\begin{lemma}
\label{lem:wassintres}
{Suppose Assumptions~\ref{ass:lighttailerr},~\ref{ass:reglargedevwass}, and~\ref{ass:wassrisklevelseq} hold and the samples $\{\varepsilon^{i}\}_{i=1}^{n}$ are i.i.d.
Then, a.s.\ for $n$ large enough
\[
v^*(x) \leq g(\hz^{DRO}_n(x);x) \leq \hv^{DRO}_n(x), \quad \text{for a.e. } x \in \X.
\]
Furthermore, let $z^*(x) \in S^*(x)$ be an optimal solution to problem~\eqref{eqn:speq}. Then}
\begin{enumerate}[label=\Alph*.]
\item {If Assumption~\ref{ass:lipschitzy} holds, we a.s.\ have for a.e.\ $x \in \X$ and $n$ large enough}
\[
\hv^{DRO}_n(x) \leq v^*(x) + 2L_1(z^*(x))\zeta_n(\alpha_n,x).
\]

\item {If Assumption~\ref{ass:lipschitzy2} holds, we a.s.\ have for a.e.\ $x \in \X$ and $n$ large enough}
\[
\hspace*{-0.2in} \hv^{DRO}_n(x) \leq v^*(x) + 2\bigl(\expect{\norm{\nabla c(z^*(x),Y)}^2}\bigr)^{1/2} \zeta_n(\alpha_n,x) + 4L_2(z^*(x))\zeta^2_n(\alpha_n,x).
\]
\end{enumerate}
\end{lemma}
\begin{proof}
{See Appendix~\ref{app:wassintres}.} 
\end{proof}

{We now state our asymptotic guarantees for Wasserstein ER-DRO.}

\begin{theorem}[Consistency and asymptotic optimality]
\label{thm:wassaympconsist}
Suppose Assumptions~\ref{ass:lighttailerr}, \ref{ass:reglargedevwass},~\ref{ass:growthrate}, and~\ref{ass:wassrisklevelseq} hold and the samples $\{\varepsilon^{i}\}_{i=1}^{n}$ are i.i.d.
Then, for a.e.\ $x \in \X$, the optimal value and solution of the ER-DRO problem~\eqref{eqn:dro} with ambiguity set $\hP_n(x;\zeta_n(\alpha_n,x))$ are consistent and asymptotically optimal, i.e.,
\[
\hv^{DRO}_n(x) \xrightarrow{p} v^*(x), \:\: \textup{dist}(\hz^{DRO}_n(x),S^*(x)) \convinprob 0, \:\: g(\hz^{DRO}_n(x);x) \convinprob v^*(x).
\]
\end{theorem}

\begin{theorem}[Rate of convergence]
\label{thm:wassconvrate}
Suppose the assumptions of Theorem~\ref{thm:wassaympconsist} and either Assumption~\ref{ass:lipschitzy} or Assumption~\ref{ass:lipschitzy2} hold.
Then, for a.e.\ $x \in \X$, 
the ER-DRO problem~\eqref{eqn:dro} with ambiguity set $\hP_n(x;\zeta_n(\alpha_n,x))$ satisfies
\[
\hspace*{-0.05in}\abs*{\hv^{DRO}_n(x) - v^*(x)} = O_p\big(\zeta_n(\alpha_n,x)\big), \:\: \abs*{g(\hz^{DRO}_n(x);x) - v^*(x)} = O_p\big(\zeta_n(\alpha_n,x)\big).
\]
\end{theorem}

Proofs of Theorems~\ref{thm:wassaympconsist} and~\ref{thm:wassconvrate} are in Appendices~\ref{app:wassaympconsist} and~\ref{app:wassconvrate}.
The proof of Theorem~\ref{thm:wassaympconsist} mirrors that of~\cite[Theorem~3.6]{esfahani2018data} (it shows that the conclusions in fact hold almost surely).
Similar to the setting without covariate information~\cite{esfahani2018data}, we can typically choose the sequence of risk levels~$\{\alpha_n\}$ in {Assumption~\ref{ass:wassrisklevelseq} to be any sequence converging at a slower rate than $\{\exp(-n)\}$} 
when the errors~$\varepsilon$ are sub-Gaussian (see the discussion following Assumption~\ref{ass:reglargedevwass}).

\begin{remark}
Assumption~\ref{ass:lipschitzy2} can be weakened to consider functions~$c$ that satisfy
\[
\norm{\nabla c(z,\bar{y}) -\nabla c(z,y)} \leq L_2(z,y) \norm{\bar{y} - y}^{\kappa}, \quad \forall z \in \Z, \: y, \bar{y} \in \Y,
\]
and $\expect{\norm{L_2(z,Y)}^{p/(p-1)}} < +\infty$, $\forall z \in \Z$, for some constant $\kappa \in (0,1]$ and orders $p \geq 1+\kappa$, see~\cite[Proposition~1]{gao2017wasserstein}.
Furthermore, Assumption~\ref{ass:lipschitzy2} can also be weakened to consider functions~$c$  of the form $c(z,Y) = \max_{j \in [N_c]} c_j(z,Y)$, where $N_c \in \mathbb{N}$ and for each $z \in \Z$, the constituent functions $c_j(z,\cdot)$ are differentiable on~$\Y$ and satisfy $\mathbb{E}\bigl[\max_{j \in [N_c]} \norm{\nabla c_j(z,Y)}^2\bigr] < +\infty$ and
\[
\norm{\nabla c_j(z,\bar{y}) -\nabla c_j(z,y)} \leq L_{j,2}(z) \norm{\bar{y} - y}, \quad \forall y, \bar{y} \in \Y, \: j \in [N_c].
\]
The above weakening of Assumption~\ref{ass:lipschitzy2} makes it applicable to a larger class of 
stochastic programs.
We stick with Assumption~\ref{ass:lipschitzy2} for simplicity.
\end{remark}

\begin{remark}
\label{rem:wassdrorate}
Recall the radius given in \eqref{eqn:wassradius} consists of two parts. For the part that relates to the Wasserstein ambiguity set without covariate information, 
because the rate $d_{W,p}( P^*_n(x), P_{Y \mid X=x} ) = O_p(n^{-p/d_y})$ cannot be improved in general (see~\cite[Example~3]{kuhn2019wasserstein}), we usually  have $\kappa^{(2)}_{p,n}(\alpha_n)$ converging to zero only at the slow rate $\Theta(n^{-p/d_y})$.
Therefore, the convergence rate afforded by Theorem~\ref{thm:wassconvrate} suffers from the curse of dimensionality even when we use parametric regression methods, which typically exhibit better rates of convergence on the part of the radius that relates to the estimation of $f^*$ (cf.\ \cite[Theorem~2]{kannan2020data}). 
The analysis in Gao~\cite{gao2020finite} {and Blanchet et al.\ \cite{blanchet2019confidence}} implies that, under certain assumptions, using the {smaller} radius $\zeta_n(\alpha,x) := \max\{ \kappa^{(1)}_{p,n}(\alpha,x), \bar{\kappa}^{(2)}_{p,n}(\alpha) \}$ with suitably chosen $\bar{\kappa}^{(2)}_{p,n}(\alpha) = O(n^{-1/2})$ 
results in estimators with a finite sample certificate-type guarantee {and sharper convergence rates}.
This {smaller} choice of the radius $\zeta_n$ also yields estimators with the conventional $O_p(n^{-1/2})$ rate of convergence when we use parametric regression methods to estimate the function~$f^*$.
{Consequently, if sharper finite sample guarantees such as those in~\cite{gao2020finite,blanchet2019confidence} apply, then Theorem~\ref{thm:wassconvrate} can be readily adapted to derive sharper convergence rates.}
{However, the assumptions in~\cite{gao2020finite,blanchet2019confidence} may exclude some formulations of interest, such as two-stage stochastic programs (see, e.g.,~\citep[Assumption~(A3)]{blanchet2019confidence} and~\citep[Assumption~2]{gao2020finite}), or may be difficult to verify in general (see~\citep[Section~5]{gao2020finite})}.
\end{remark}

{Next, we identify conditions under which the optimal objective value of the Wasserstein ER-DRO problem~\eqref{eqn:dro} converges to $v^*(x)$ at a suitable rate with respect to the $L^q$-norm on~$\X$ for $q \in [1,\infty]$.
We make the following stronger form of Assumption~\ref{ass:reglargedevwass} for simplicity.}

\begin{assumption}
\label{ass:reglargedevunif}
{The regression estimate~$\hf_n$ possesses the following finite sample property: for any risk level $\alpha \in (0,1)$, there exists a positive constant $\kappa_n(\alpha)$ such that
$\mathbb{P}\bigl\{\sup_{x \in \X} \norm{f^*(x) - \hf_n(x)} > \kappa_n(\alpha)\bigr\} \leq \alpha$.}
\end{assumption}

{
Appendix~EC.3 of~\cite{kannan2020data} verifies that Assumption~\ref{ass:reglargedevunif} holds for some parametric and nonparametric regression methods such as OLS, Lasso, and kNN regression when the support~$\X$ of the covariates is compact.
When Assumption~\ref{ass:reglargedevunif} holds, we write~$\zeta_n(\alpha)$ instead of~$\zeta_n(\alpha,x)$ for the radius specified by~\eqref{eqn:wassradius}.}

\begin{theorem}[Mean convergence rate]
\label{thm:wassconvratelq}
{Suppose Assumptions~\ref{ass:lighttailerr}, \ref{ass:growthrate},~\ref{ass:wassrisklevelseq}, and~\ref{ass:reglargedevunif} hold, and $\{\varepsilon^{i}\}_{i=1}^{n}$ are i.i.d.
Let $q \in [1,+\infty]$ and suppose either Assumption~\ref{ass:lipschitzy} holds with $\norm*{L_1(z^*(\cdot))}_{L^q} < +\infty$, or Assumption~\ref{ass:lipschitzy2} holds with $\norm{\bigl(\mathbb{E}_{Y}\bigl[\norm{\nabla c(z^*(\cdot),Y)}^2\bigr]\bigr)^{1/2}}_{L^q} < +\infty$ and $\norm*{L_2(z^*(\cdot))}_{L^q} < +\infty$.
Then, 
the ER-DRO problem~\eqref{eqn:dro} with ambiguity set $\hP_n(x;\zeta_n(\alpha_n,x))$ satisfies}
\begin{align*}
\norm*{\hv^{DRO}_n(X) - v^*(X)}_{L^q} &= O_p\bigl(\zeta_n(\alpha_n)\bigr), \\ 
\norm*{g(\hz^{DRO}_n(X);X) - v^*(X)}_{L^q} &= O_p\bigl(\zeta_n(\alpha_n)\bigl).
\end{align*}
\end{theorem}
\begin{proof}
{Following the proof of Theorem~\ref{thm:wassfinitesampcert}, we have
\[
\mathbb{P}\big\{d_{W,p}( \hat{P}^{ER}_n(x), P_{Y \mid X=x} ) > \zeta_n(\alpha_n), \:\: \forall x \in \X\big\} \leq \alpha_n.
\]
Note that we are able to make the above assertion \textit{jointly} over all $x \in \X$ because the radius~$\zeta_n(\alpha_n)$ is independent of $x$ by Assumption~\ref{ass:reglargedevunif}.
Following the proof of Lemma~\ref{lem:wassintres}, we then a.s.\ have for all $n$ large enough:
\[
v^*(x) \leq g(\hz^{DRO}_n(x);x) \leq \hv^{DRO}_n(x), \quad \forall x \in \X.
\]
Suppose Assumption~\ref{ass:lipschitzy} holds.
Let $z^*(x) \in S^*(x)$.
From part A of Lemma~\ref{lem:wassintres}, the above inequalities a.s.\ imply for all $n$ large enough:
\[
\hv^{DRO}_n(x) - v^*(x) \leq 2L_1(z^*(x))\zeta_n(\alpha_n), \quad \forall x \in \X.
\]
Consequently, when Assumption~\ref{ass:lipschitzy} holds
\[
\norm*{\hv^{DRO}_n(X) - v^*(X)}_{L^q} = \norm*{L_1(z^*(\cdot))}_{L^q} \: O_p(\zeta_n(\alpha_n)).
\]
Suppose instead that Assumption~\ref{ass:lipschitzy2} holds.
From part B of Lemma~\ref{lem:wassintres}, the above inequalities a.s.\ imply for all $n$ large enough and each $x \in \X$:
\[
\hv^{DRO}_n(x) - v^*(x) \leq 2\bigl(\mathbb{E}_Y\bigl[\norm{\nabla c(z^*(x),Y)}^2\bigr]\bigr)^{1/2} \zeta_n(\alpha_n) + 4L_2(z^*(x))\zeta^2_n(\alpha_n).
\]
Consequently, when Assumption~\ref{ass:lipschitzy2} holds}
\[
\norm*{\hv^{DRO}_n(X) - v^*(X)}_{L^q} = \norm{\bigl(\mathbb{E}_{Y}\bigl[\norm{\nabla c(z^*(\cdot),Y)}^2\bigr]\bigr)^{1/2}}_{L^q} \: O_p(\zeta_n(\alpha_n)). \quad \qedhere
\]
\end{proof}

{Appendix~EC.2 of~\cite{kannan2020data} presents conditions under which some of the new assumptions of Theorem~\ref{thm:wassconvratelq} hold.}
We now identify conditions under which the ER-DRO estimators possess a finite sample solution guarantee.
We first refine Assumption~\ref{ass:reglargedevwass} to another more convenient, stronger form in Assumption~\ref{ass:reglargedevwass2}.

\begin{assumption}
\label{ass:reglargedevwass2}
The regression estimate~$\hf_n$ possesses the following large deviation properties: for any constant $\kappa > 0$, there exist positive constants $K_{p,f}(\kappa,x)$, $\bar{K}_{p,f}(\kappa)$, $\beta_{p,f}(\kappa,x)$, and $\bar{\beta}_{p,f}(\kappa)$ satisfying
\begin{align*}
&\mathbb{P}\big\lbrace \norm{f^*(x) - \hf_n(x)}^p > \kappa^p \big\rbrace \leq K_{p,f}(\kappa,x) \exp\left(-n\beta_{p,f}(\kappa,x)\right), \: \text{for a.e. } x \in \X, \\
&\mathbb{P}\bigg\lbrace \dfrac{1}{n} \displaystyle\sum_{i=1}^{n} \norm{f^*(x^i) - \hf_n(x^i)}^p > \kappa^p \bigg\rbrace \leq \bar{K}_{p,f}(\kappa) \exp\left(-n\bar{\beta}_{p,f}(\kappa)\right).
\end{align*}
\end{assumption}

Appendix~EC.3 of~\cite{kannan2020data} verifies Assumption~\ref{ass:reglargedevwass2} for some popular regression setups for $p = 2$; see the discussion after Assumption~\ref{ass:reglargedevwass} for $p \neq 2$.
{The following result will prove useful in deriving our finite sample solution guarantee.}

\begin{lemma}
\label{lem:wassfinitesampsoln}
{Suppose Assumptions~\ref{ass:lighttailerr}, \ref{ass:reglargedevwass}, \ref{ass:growthrate}, and~\ref{ass:reglargedevwass2} hold,}
 {the samples $\{\varepsilon^{i}\}_{i=1}^{n}$ are i.i.d., and either Assumption~\ref{ass:lipschitzy} or Assumption~\ref{ass:lipschitzy2} holds.
Then, for a.e.\ $x \in \X$ and any $\kappa > 0$, there exist positive constants $\tilde{\Gamma}(\kappa,x)$ and $\tilde{\gamma}(\kappa,x)$ such that the solution of the ER-DRO problem~\eqref{eqn:dro} with risk level $\alpha = \tilde{\Gamma}(\kappa,x) \exp(-n\tilde{\gamma}(\kappa,x))$, radius~$\zeta_n(\alpha,x)$ specified by~\eqref{eqn:wassradius}, and ambiguity set~$\hP_n(x;\zeta_n(\alpha,x))$ satisfies}
\begin{align}
\label{eqn:largedevoutofsampcost}
&\prob{g(\hz^{DRO}_n(x);x) > v^*(x) + \kappa} \leq 2\tilde{\Gamma}(\kappa,x) \exp(-n\tilde{\gamma}(\kappa,x)).
\end{align}
\end{lemma}
\begin{proof}
{See Appendix~\ref{app:wassfinitesampsoln}.}
\end{proof}

\begin{theorem}[Finite sample solution guarantee]
\label{thm:wassfinitesampsoln}
Suppose Assumptions~\ref{ass:lighttailerr}, \ref{ass:reglargedevwass}, \ref{ass:growthrate}, and~\ref{ass:reglargedevwass2} hold, the samples $\{\varepsilon^{i}\}_{i=1}^{n}$ are i.i.d., and either Assumption~\ref{ass:lipschitzy} or Assumption~\ref{ass:lipschitzy2} holds.
{Then, for a.e.\ $x \in \X$ and any $\eta > 0$, there exist positive constants $\Gamma(\eta,x)$ and $\gamma(\eta,x)$ such that the solution of the ER-DRO problem~\eqref{eqn:dro} with risk level $\alpha = \Gamma(\eta,x) \exp(-n\gamma(\eta,x))$, radius~$\zeta_n(\alpha,x)$ determined using equation~\eqref{eqn:wassradius}, and ambiguity set~$\hP_n(x;\zeta_n(\alpha,x))$ satisfies}
\[
\prob{\textup{dist}(\hz^{DRO}_n(x),S^*(x)) \geq \eta} \leq 2\Gamma(\eta,x) \exp(-n\gamma(\eta,x)).
\]
\end{theorem}
\begin{proof}
{From Lemma~\ref{lem:wassfinitesampsoln}, we have for any $\kappa > 0$, there exist $\tilde{\Gamma}(\kappa,x) > 0$ and $\tilde{\gamma}(\kappa,x) > 0$ such that inequality~\eqref{eqn:largedevoutofsampcost} holds with~$\alpha = \tilde{\Gamma}(\kappa,x) \exp(-n\tilde{\gamma}(\kappa,x))$.
We now argue that inequality~\eqref{eqn:largedevoutofsampcost} implies the stated result.}

Suppose we have $\text{dist}(\hz^{DRO}_n(x),S^*(x)) \geq \eta$ for some $\eta > 0$, $x \in \X$, and sample path.
Since $g(\cdot;x)$ is lower semicontinuous on the compact set~$\Z$ for a.e.\ $x \in \X$,~\cite[Lemma~3]{kannan2020data} implies $g(\hz^{DRO}_n(x);x) > v^*(x) + \kappa(\eta,x)$ for some constant $\kappa(\eta,x) > 0$ on that path (except for paths of measure zero).
We now bound the probability of this event.
The above arguments imply for a.e.\ $x \in \X$
\begin{align*}
\prob{\text{dist}(\hz^{DRO}_n(x),S^*(x)) \geq \eta} &\leq \prob{g(\hz^{DRO}_n(x);x) > v^*(x) + \kappa(\eta,x)} \\
&\leq 2\tilde{\Gamma}(\kappa(\eta,x),x) \exp(-n\tilde{\gamma}(\kappa(\eta,x),x)).
\end{align*}
{Therefore, the desired result holds with constants $\Gamma(\eta,x) = \tilde{\Gamma}(\kappa(\eta,x),x)$ and $\gamma(\eta,x) = \tilde{\gamma}(\kappa(\eta,x),x)$.}
\end{proof}

Theorem~\ref{thm:wassfinitesampsoln} is similar to the finite sample guarantee in~\cite[Theorem~3]{kannan2020data} for solutions to the ER-SAA problem.
However, unlike~\cite[Theorem~3]{kannan2020data}, the dependence of the convergence rate on the parameter $\eta$ in Theorem~\ref{thm:wassfinitesampsoln} suffers from the curse of dimensionality even if we use parametric regression methods to estimate~$f^*$ (cf.\ Remark~\ref{rem:wassdrorate}).
{Inequality~\eqref{eqn:largedevoutofsampcost} shows that the out-of-sample cost of the Wasserstein ER-DRO estimators possesses a finite sample guarantee similar to the guarantee in the solution space.
This convergence rate estimate also suffers from the curse of dimensionality with respect to the parameter~$\kappa$ (we do not know if faster rates of convergence can be derived).}

\section{Sample robust optimization-based ambiguity sets and ambiguity sets with the same support as \texorpdfstring{$\hat{P}^{ER}_n(x)$}{an estimated empirical distribution}}
\label{sec:samprobustopt}

In this section we present a unified analysis of using two forms of ambiguity sets within problem~\eqref{eqn:dro}:  sample robust optimization-based ambiguity sets and ambiguity sets with the same support as $\hat{P}^{ER}_n(x)$.
Specifically, we consider ambiguity sets of the form
\begin{align*}
\label{eqn:esdroambiguityset}
&\hP_n(x) := \bigg\{Q = \sum_{i=1}^{n} p_i\delta_{\bar{y}^i} : p \in \mathfrak{P}_n(x;\zeta_{n}(x)), \: \bar{y}^i \in \hat{\Y}^{i}_n(x;\mu_n(x)), \forall i \in [n]\bigg\},\\
&\hat{\Y}^{i}_n(x;\mu_n(x)) := \big\lbrace y \in \Y : \norm{y - \proj{\Y}{\hf_n(x) + \heps^i_{n}}} \leq \mu_{n}(x) \big\rbrace, \forall i \in [n],
\end{align*}
where $\mu_n(x)$ and $\zeta_n(x)$ are nonnegative radii and the ambiguity set~$\mathfrak{P}_n(x;\zeta_{n}(x))$ for the probabilities~$p$ satisfies~\eqref{eqn:ambiguitysetconsistency}.
This family of ambiguity sets generalizes both sample robust optimization-based ambiguity sets constructed using the $\ell_2$-norm (obtained by setting $\zeta_{n}(x) = 0$) and ambiguity sets with the same support as~$\hat{P}^{ER}_n(x)$ (obtained by setting $\mu_{n}(x) = 0$).
We establish asymptotic optimality, rates of convergence, and finite sample-type guarantees for the corresponding ER-DRO estimators~\eqref{eqn:dro}.

When~$\mu_{n}(x) = 0$ and problem~\eqref{eqn:speq} is a tractable convex program, the resulting ER-DRO problem~\eqref{eqn:dro} remains tractable and convex for many choices of the ambiguity set~$\Pf_n(x;\zeta_{n}(x))$ such as Examples~\ref{exm:cvar} and~\ref{exm:phidivergence} (see, e.g.,~\cite{ben2013robust}).
On the other hand, when~$\mu_{n}(x) > 0$ and problem~\eqref{eqn:speq} is a two-stage stochastic linear program, then the ER-DRO problem~\eqref{eqn:dro} exhibits a $\min\text{-}\max\text{-}\min$ structure whose solution is in general NP-hard.
References~\cite{bertsimas2019two,xie2020tractable} investigate approaches for approximately solving the ER-DRO problem~\eqref{eqn:dro} when the true problem~\eqref{eqn:speq} is a two-stage stochastic LP and~$\zeta_{n}(x) = 0$.

To facilitate our analysis, denote by $\hg^{ER}_{s,n}$ and $g^*_{s,n}$ the functions
\begin{align*}
\hg^{ER}_{s,n}(z;x) &:= \uset{p \in \Pf_n(x;\zeta_{n}(x))}{\sup} \sum_{i=1}^{n} p_i \sup_{y \in \hat{\Y}^{i}_n(x;\mu_n(x))} c(z,y), \\
g^*_{s,n}(z;x) &:= \uset{p \in \Pf_n(x;\zeta_{n}(x))}{\sup} \sum_{i=1}^{n} p_i c(z,f^*(x)+\varepsilon^i).
\end{align*}
Note that the function~$\hg^{ER}_{s,n}$ is equivalent to the objective function of the ER-DRO problem~\eqref{eqn:dro} with the above definition of the ambiguity set~$\hP_n(x)$.
Additionally, $g^*_{s,n}$ is equivalent to the objective function of the FI-SAA problem~\eqref{eqn:fullinfsaa} when $\zeta_n(x) = 0$ and condition~\eqref{eqn:ambiguitysetconsistency} holds.

We begin by investigating conditions under which the optimal value and set of optimal solutions to the ER-DRO problem~\eqref{eqn:dro} converge in probability to the true problem~\eqref{eqn:speq}.
We make the following assumptions in this regard.

\begin{assumption}
\label{ass:equilipschitz}
For each $z \in \Z$, the function $c(z,\cdot)$ is Lipschitz continuous on $\Y$ with Lipschitz constant $L(z)$ satisfying $\sup_{z \in \Z} L(z) < +\infty$.
\end{assumption}

\begin{assumption}
\label{ass:uniflln}
For a.e.\ $x \in \X$, the sequence of FI-SAA objectives $\left\lbrace g^*_n(\cdot;x) \right\rbrace$ converges in probability to the function $g(\cdot;x)$ uniformly on the set~$\Z$.
\end{assumption}

\begin{assumption}
\label{ass:regconsist}
The regression estimate~$\hf_n$ has the consistency properties
\begin{align*}
&\hf_n(x) \xrightarrow{p} f^*(x), \:\: \text{for a.e. } x \in \X, \quad \text{and} \quad \dfrac{1}{n} \displaystyle\sum_{i=1}^{n} \norm{f^*(x^i) - \hf_n(x^i)}^2 \xrightarrow{p} 0.
\end{align*}
\end{assumption}

Assumption~\ref{ass:equilipschitz} is a uniform Lipschitz continuity assumption that strengthens Assumption~\ref{ass:lipschitzy}.
Appendix~EC.2 of~\cite{kannan2020data} verifies that Assumption~\ref{ass:equilipschitz} 
holds for two-stage stochastic MIPs with continuous recourse.
Assumption~\ref{ass:uniflln} is a uniform weak LLN assumption, whereas Assumption~\ref{ass:regconsist} is a mild consistency assumption that holds for many popular regression setups (cf.\ Assumptions~3 and~4 of~\cite{kannan2020data}).
Assumption~\ref{ass:regconsist} is weaker than Assumption~\ref{ass:reglargedevwass}.
We require the following additional assumptions for ambiguity sets with $\zeta_n(x) > 0$.

\begin{assumption}
\label{ass:ambiguityset}
The radius $\zeta_{n}(x)$ of the ambiguity set
is chosen such that 
\[
\uset{p \in \Pf_n(x;\zeta_{n}(x))}{\sup} \sum_{i=1}^{n} \Big(p_i - \frac{1}{n}\Big)^2 = O(n^{-\rho}), \quad \text{for a.e. } x \in \X,
\]
for some constant $\rho > 1$.
\end{assumption}

\begin{assumption}
\label{ass:uniflln2}
The following weak uniform LLN holds for a.e.\ $x \in \X$:
\[
\uset{z \in \Z}{\sup}\: \Bigl\lvert\frac{1}{n}\sum_{i=1}^{n} \big(c(z,f^*(x) + \varepsilon^i)\big)^2 - \mathbb{E}\Bigl[ \big(c(z,f^*(x) + \varepsilon)\big)^2\Bigr] \Bigr\rvert \convinprob 0,
\]
with $\sup_{z \in \Z} \mathbb{E}\bigl[ \big(c(z,f^*(x) + \varepsilon)\big)^2\bigr] < +\infty$ for a.e.\ $x \in \X$.
\end{assumption}

Assumption~\ref{ass:ambiguityset} requires us to choose the radius~$\zeta_{n}(x)$ so that the ambiguity set~$\Pf_n(x;\zeta_{n}(x))$ converges to the singleton $\big\lbrace\bigl(\tfrac{1}{n},\dots,\tfrac{1}{n}\bigr)\big\rbrace$ at a fast enough rate. This is always possible since we assume equation~\eqref{eqn:ambiguitysetconsistency} holds.
We are interested in cases when Assumption~\ref{ass:ambiguityset} holds with $\rho \in (1, 2]$ (see Theorem~\ref{thm:samprobustrateofconv}).
{Lemma~13 of~\cite{duchi2016statistics} (cf.\ \cite{ben2013robust,lam2016robust,lam2019recovering}) shows that for phi-divergence ambiguity sets $\Pf_n(x;\zeta_{n}(x))$ constructed using a twice continuously differentiable and strictly convex divergence function~$\phi$ with $\phi'(1) = 0$ (these conditions are satisfied by most of the divergence functions listed in~\cite[Table~2]{ben2013robust}), we have
\[
\uset{p \in \Pf_n(x;\zeta_{n}(x))}{\sup} \sum_{i=1}^{n} \Big(p_i - \frac{1}{n}\Big)^2 = \Theta\Bigl(\frac{\zeta_{n}(x)}{n}\Bigr).
\]
Consequently, Assumption~\ref{ass:ambiguityset} holds for such phi-divergence-based ambiguity sets $\Pf_n(x;\zeta_{n}(x))$ whenever the radius~$\zeta_{n}(x) = O(n^{1-\rho})$.
This bound on~$\zeta_{n}(x)$ is \textit{sharp} in the sense that Assumption~\ref{ass:ambiguityset} does not hold if~$\zeta_{n}(x)$ grows faster than $n^{1-\rho}$ asymptotically.
Appendix~\ref{app:ambiguitysetrate} presents some other examples of ambiguity sets for which Assumption~\ref{ass:ambiguityset} holds.}

Theorem~7.48 of~\cite{shapiro2009lectures} presents conditions under which both Assumptions~\ref{ass:uniflln} and~\ref{ass:uniflln2} hold when the samples~$\{\varepsilon^i\}_{i=1}^{n}$ are i.i.d.
Note that Assumption~\ref{ass:uniflln2} can also be equivalently stated as a weak uniform LLN assumption on the sample variance of the sequence $\{c(z,f^*(x)+\varepsilon^i)\}_{i=1}^{n}$ when~$\{\varepsilon^i\}_{i=1}^{n}$ are i.i.d.\ \cite{duchi2016statistics}.

{The following result will be useful in deriving asymptotic guarantees for the ER-DRO formulations studied in this section.}

\begin{lemma}
\label{lem:unifconvsamprobust}
{Suppose Assumption~\ref{ass:equilipschitz} holds.
We have for each $x \in \X$}
\begin{align}
\label{eqn:unifconvsamprobustres}
&{\uset{z \in \Z}{\sup}\: \abs*{\hg^{ER}_{s,n}(z;x) - g(z;x)}} \nonumber\\
{\leq}& \: {\uset{z \in \Z}{\sup} \: L(z) \Bigl(\mu_{n}(x) + \Bigl(\dfrac{1}{n}\displaystyle\sum_{i=1}^{n} \bigl( \norm{\teps^{i}_{n}(x)}\bigr)^2\Bigr)^{\frac{1}{2}} \Bigr) \uset{p \in \Pf_n(x;\zeta_{n}(x))}{\sup} \Bigl( 1 + n \sum_{i=1}^{n} \Bigl(p_i - \frac{1}{n}\Bigr)^2\Bigr)^{\frac{1}{2}}} \nonumber\\
&\: {+ \uset{p \in \Pf_n(x;\zeta_{n}(x))}{\sup} \Bigl(n \sum_{i=1}^{n} \Bigl(p_i - \frac{1}{n}\Bigr)^2 \Bigr)^{\frac{1}{2}} \uset{z \in \Z}{\sup} \Bigl(\frac{1}{n}\sum_{i=1}^{n} \big(c(z,f^*(x) + \varepsilon^i)\big)^2\Bigr)^{\frac{1}{2}}} \nonumber\\
&\: {+ \uset{z \in \Z}{\sup} \: \abs*{g^*_n(z;x) - g(z;x)}.}
\end{align}
\end{lemma}
\begin{proof}
{See Appendix~\ref{app:unifconvsamprobust}.}
\end{proof}

Our first result identifies conditions under which the sequence of objective functions $\{\hg^{ER}_{s,n}(\cdot;x)\}$ of the ER-DRO problem~\eqref{eqn:dro} converges uniformly to the objective function~$g(\cdot;x)$ of the true problem~\eqref{eqn:speq} on~$\Z$.
Theorem~9 of~\cite{duchi2016statistics} presents an analogous result for a class of phi-divergence-based ambiguity sets in the absence of covariate information.

\begin{proposition}
\label{prop:unifconvsamprobust}
Suppose Assumptions~\ref{ass:equilipschitz} to~\ref{ass:uniflln2} hold and the radius $\mu_{n}(x)$ satisfies $\lim_{n \to \infty} \mu_{n}(x) = 0$ for a.e.\ $x \in \X$.
Then, for a.e.~$x \in \X$, the sequence of objectives $\{\hg^{ER}_{s,n}(\cdot;x)\}$ of the ER-DRO problem~\eqref{eqn:dro} converges in probability to the objective {$g(\cdot;x)$} of the true problem~\eqref{eqn:speq} uniformly on the set~$\Z$.
\end{proposition}
\begin{proof}
We wish to show that
\[
\uset{z \in \Z}{\sup} \: \abs*{\hg^{ER}_{s,n}(z;x) - g(z;x)} \xrightarrow{p} 0, \quad \text{for a.e. } x \in \X.
\]
We bound this term from above using Lemma~\ref{lem:unifconvsamprobust}.

The third term on the r.h.s.\ of~\eqref{eqn:unifconvsamprobustres} vanishes in the limit in probability for a.e.\ $x \in \X$ under Assumption~\ref{ass:uniflln}.
We show that the first two terms also converge to zero in probability; the result then follows by $o_p(1) + o_p(1) = o_p(1)$.

Consider the first term on the r.h.s.\ of~\eqref{eqn:unifconvsamprobustres}.
We have for a.e.\ $x \in \X$
\begin{align*}
& \:\uset{z \in \Z}{\sup} \: L(z) \Bigl(\mu_{n}(x) + \Bigl(\dfrac{1}{n}\displaystyle\sum_{i=1}^{n} \bigl( \norm{\teps^{i}_{n}(x)}\bigr)^2\Bigr)^{\frac{1}{2}} \Bigr) \uset{p \in \Pf_n(x;\zeta_{n}(x))}{\sup} \Bigl( 1 + n \sum_{i=1}^{n} \Bigl(p_i - \frac{1}{n}\Bigr)^2\Bigr)^{\frac{1}{2}} \nonumber\\
&= \: O(1) o_p(1) O(1) = o_p(1),
\end{align*}%
on account of Assumptions~\ref{ass:equilipschitz},~\ref{ass:regconsist} and~\ref{ass:ambiguityset}, $\lim_{n \to \infty} \mu_{n}(x) = 0$, and~\cite[Lemma~1]{kannan2020data}.

Next, consider the second term on the r.h.s.\ of~\eqref{eqn:unifconvsamprobustres}.
We have for a.e.\ $x \in \X$
{
\begin{align*}
& \: \uset{p \in \Pf_n(x;\zeta_{n}(x))}{\sup} \Bigl(n \sum_{i=1}^{n} \Bigl(p_i - \frac{1}{n}\Bigr)^2 \Bigr)^{\frac{1}{2}} \uset{z \in \Z}{\sup} \Bigl(\frac{1}{n}\sum_{i=1}^{n} \bigl(c(z,f^*(x) + \varepsilon^i)\bigl)^2\Bigr)^{\frac{1}{2}} \nonumber\\
&= \: o(1) O_p(1) = o_p(1),
\end{align*}
}%
on account of Assumptions~\ref{ass:ambiguityset} and~\ref{ass:uniflln2}.
\end{proof}

It can be seen from the proof that Assumptions~\ref{ass:ambiguityset} and~\ref{ass:uniflln2} are not required for sample robust optimization-based DRO, i.e., when the radius~$\zeta_n(x) \equiv 0$.

\begin{remark}
Assumption~\ref{ass:equilipschitz} can be weakened to a local Lipschitz continuity assumption under stronger assumptions on the regression setup.
In particular, when $\zeta_{n}(x) \equiv 0$, the conclusion of Proposition~\ref{prop:unifconvsamprobust} holds if we replace Assumption~\ref{ass:equilipschitz} with~\cite[Assumption~2]{kannan2020data}.
When~$\zeta_{n}(x) \neq 0$, we need to replace Assumption~\ref{ass:equilipschitz} with strengthened versions of Assumption~\ref{ass:regconsist} and~\cite[Assumption~2]{kannan2020data} involving fourth degree terms.
\end{remark}

Proposition~\ref{prop:unifconvsamprobust} provides the foundation for showing that the ER-DRO estimators are asymptotically optimal.  
We omit the proof of Theorem~\ref{thm:samprobustconv} since it is identical to the proof of~\cite[Theorem~1]{kannan2020data} in light of Proposition~\ref{prop:unifconvsamprobust}.

\begin{theorem}[Consistency and asymptotic optimality]
\label{thm:samprobustconv}
Suppose the assumptions of Proposition~\ref{prop:unifconvsamprobust} hold.
Then, for a.e.\ $x \in \X$
\[
\hv^{DRO}_n(x) \xrightarrow{p} v^*(x), \:\: \dev{\hS^{DRO}_n(x)}{S^*(x)} \xrightarrow{p} 0, \:\: \uset{z \in \hS^{DRO}_n(x)}{\sup} g(z;x) \xrightarrow{p} v^*(x).
\]
\end{theorem}

Next, we investigate the rate of convergence of the optimal value of the ER-DRO problem~\eqref{eqn:dro} to that of the true problem~\eqref{eqn:speq}.
To enable this, we require the following rate of convergence assumptions on the FI-SAA problem~\eqref{eqn:speq} and the regression estimate~$\hf_n$ (cf.\ Assumptions~5 and~6 of~\cite{kannan2020data}).

\begin{assumption}
\label{ass:functionalclt}
The function~$c$ in problem~\eqref{eqn:speq} and the data~$\D_n$ satisfy the following functional central limit theorem for the FI-SAA objective:
\[
\sqrt{n} \left( g^*_n(\cdot;x) - g(\cdot;x) \right) \xrightarrow{d} V(\cdot;x), \quad \text{for a.e. } x \in \X,
\]
where $g^*_n(\cdot;x)$, $g(\cdot;x)$, and $V(\cdot;x)$ are (random) elements of {$C(\Z)$}.
\end{assumption}

\begin{assumption}
\label{ass:regconvrate}
There is a constant\footnote{The constant $r$ is independent of~$n$, but could depend on the covariate dimension~$d_x$.} $0 < r \leq 1$ such that the regression estimate~$\hf_n$ satisfies the following convergence rate criteria for a.e.\ $x \in \X$:
\[
\norm{f^*(x) - \hf_n(x)}^2 = O_p(n^{-r}), \:\: \dfrac{1}{n} \displaystyle\sum_{i=1}^{n} \norm{f^*(x^i) - \hf_n(x^i)}^2 = O_p(n^{-r}).
\]
\end{assumption}

Assumption~\ref{ass:regconvrate} strengthens Assumption~\ref{ass:regconsist}. 
It typically holds with $r = 1$ for parametric regression methods such as OLS and Lasso regression under mild assumptions.
On the other hand, nonparametric regression methods such as kernel regression and random forests usually satisfy Assumption~\ref{ass:regconvrate} only with $r = O(1)/d_x$ due to the curse of dimensionality.

Our next result establishes a convergence rate for the ER-DRO problem~\eqref{eqn:dro}.
{The choice $\rho = 1+r$ in Assumption~\ref{ass:ambiguityset} ensures that the resulting ER-DRO estimators enjoy the same rate of convergence as the ER-SAA estimators in~\cite{kannan2020data}.}

\begin{theorem}[Rate of convergence]
\label{thm:samprobustrateofconv}
Suppose Assumptions~\ref{ass:equilipschitz},~\ref{ass:uniflln2},~\ref{ass:functionalclt}, and \ref{ass:regconvrate} hold.
In addition, suppose Assumption~\ref{ass:ambiguityset} holds with $\rho = 1+r$ and the radius~$\mu_{n}(x)$ satisfies $\mu_{n}(x) = O(n^{-r/2})$ for a.e.\ $x \in \X$, where the constant~$r$ is defined in Assumption~\ref{ass:regconvrate}.
Then, for a.e.\ $x \in \X$, the solution of the ER-DRO problem~\eqref{eqn:dro} satisfies 
\[
\abs*{\hv^{DRO}_n(x) - v^*(x)} = O_p(n^{-r/2}), \:\:\abs*{g(\hz^{DRO}_n(x);x) - v^*(x)} = O_p(n^{-r/2}).
\]
\end{theorem}
\begin{proof}
We bound
$\sup_{z \in \Z} \: \abs*{\hg^{ER}_{s,n}(z;x) - g(z;x)}$ from above using Lemma~\ref{lem:unifconvsamprobust}.

Assumptions~\ref{ass:equilipschitz}, \ref{ass:ambiguityset}, and~\ref{ass:regconvrate} and $\mu_{n} = O(n^{-r/2})$
imply that the first term on the r.h.s.\ of inequality~\eqref{eqn:unifconvsamprobustres} satisfies for a.e.\ $x \in \X$
\[
\sup_{z \in \Z} \abs*{\hg^{ER}_{s,n}(z;x) - g^*_{s,n}(z;x)} = O_p(n^{-r/2}).
\]
Assumptions~\ref{ass:ambiguityset} and~\ref{ass:uniflln2}
imply that the second term on the r.h.s.\ of inequality~\eqref{eqn:unifconvsamprobustres} satisfies for a.e.\ $x \in \X$
\[
\sup_{z \in \Z} \abs*{g^*_{s,n}(z;x) - g^*_n(z;x)} = O_p(n^{-r/2}).
\]
Finally, Assumption~\ref{ass:functionalclt} implies $\sqrt{n} \sup_{z \in \Z} \abs*{g^*_n(z;x) - g(z;x)} = O_p(1)$ for a.e.\ $x \in \X$, which in turn implies $\sup_{z \in \Z} \abs*{g^*_n(z;x) - g(z;x)} = O_p(n^{-1/2})$.
Putting the above three inequalities together into inequality~\eqref{eqn:unifconvsamprobustres}, we obtain
\[
\uset{z \in \Z}{\sup} \abs*{\hg^{ER}_{s,n}(z;x) - g(z;x)} = O_p(n^{-r/2}), \quad \text{for a.e. } x \in \X.
\]
This implies that for a.e.\ $x \in \X$ and any $\alpha > 0$, there exists $M_{\alpha} > 0$ such that
\[
\prob{\uset{z \in \Z}{\sup} \abs*{\hg^{ER}_{s,n}(z;x) - g(z;x)} > M_{\alpha}n^{-r/2}} < \alpha.
\]
Consequently, we have for a.e.\ $x \in \X$
\begin{align*}
\prob{\hv^{DRO}_n(x) > v^*(x) + M_{\alpha}n^{-\frac{r}{2}}} &\leq \prob{\hg^{ER}_{s,n}(z^*(x);x) > v^*(x) + M_{\alpha}n^{-\frac{r}{2}}} \\
&\leq \prob{\abs*{\hg^{ER}_{s,n}(z^*(x);x) - v^*(x)} > M_{\alpha}n^{-\frac{r}{2}}},\\
\prob{v^*(x) > \hv^{DRO}_n(x) + M_{\alpha}n^{-\frac{r}{2}}} &\leq \prob{g(\hz^{DRO}_n(x);x) > \hv^{DRO}_n(x) + M_{\alpha}n^{-\frac{r}{2}}} \\
&\leq \prob{\abs*{\hv^{DRO}_n(x) - g(\hz^{DRO}_n(x);x)} > M_{\alpha}n^{-\frac{r}{2}}}.
\end{align*}
Therefore, both $\abs{\hv^{DRO}_n(x) - v^*(x)}$, $\abs*{g(\hz^{DRO}_n(x);x) - v^*(x)}$ are $O_p(n^{-r/2})$.
\end{proof}

{Our next result analyzes the rate of convergence of the ER-DRO objective with respect to the $L^q$-norm.
We require the following refined assumptions.}

\begin{assumption}
\label{ass:ambiguitysetunif}
The radius $\zeta_{n}(x)$ of the ambiguity set is chosen such that 
\[
\sup_{x \in \X} \: \uset{p \in \Pf_n(x;\zeta_{n}(x))}{\sup} \sum_{i=1}^{n} \Big(p_i - \frac{1}{n}\Big)^2 = O(n^{-\rho})
\]
for some constant $\rho > 1$.
\end{assumption}

\begin{assumption}
\label{ass:regconvrateunif}
{There is a constant $0 < r \leq 1$ such that the regression estimate~$\hf_n$ satisfies the following convergence rate criteria:}
\[
{
\norm{f^*(X) - \hf_n(X)}_{L^q} = O_p(n^{-r/2}), \:\: \dfrac{1}{n} \displaystyle\sum_{i=1}^{n} \norm{f^*(x^i) - \hf_n(x^i)}^2 = O_p(n^{-r}).
}
\]
\end{assumption}

\begin{assumption}
\label{ass:functionalclt2}
{The function~$c$ in problem~\eqref{eqn:speq} and the data~$\D_n$ satisfy:}
\begin{align*}
{\bigg\lVert \uset{z \in \Z}{\sup} \biggl(\frac{1}{n}\sum_{i=1}^{n} \big(c(z,f^*(X) + \varepsilon^i)\big)^2\biggr)^{\frac{1}{2}}\bigg\rVert_{L^q}} &{= O_p(1),} \\
{\Big\lVert\sup_{z \in \Z} \abs*{g^*_{n}(z;X) - g(z;X)} \Big\rVert_{L^q}} &{= O_p(n^{-1/2}).}
\end{align*}
\end{assumption}

{Asssumption~\ref{ass:ambiguitysetunif} requires Assumption~\ref{ass:ambiguityset} to hold uniformly over the covariates $x \in \X$.
It reduces to Assumption~\ref{ass:ambiguityset} when the ambiguity set $\Pf_n(x;\zeta_{n}(x))$ is chosen to be independent of~$x \in \X$.
Assumption~\ref{ass:regconvrateunif} requires the estimation error to converge to zero on average over the covariates.
Appendix~EC.3.\ of~\cite{kannan2020data} identifies conditions under which parametric regression methods satisfy Assumption~\ref{ass:regconvrateunif} under LLN and moment assumptions on the covariate distribution. 
Assumption~\ref{ass:functionalclt2} holds, for example, when the corresponding uniform LLNs (with respect to the decisions $z \in \Z$) hold uniformly over the covariates (cf.\ Theorem~7.48 of~\cite{shapiro2009lectures}).}

\begin{theorem}[Mean convergence rate]
\label{thm:samprobustrateofconvlq}
{Suppose Assumptions~\ref{ass:equilipschitz},~\ref{ass:regconvrateunif}, and~\ref{ass:functionalclt2} hold.
Let $q \in [1,+\infty]$. Suppose Assumption~\ref{ass:ambiguitysetunif} holds with $\rho = 1+r$, with constant $r$ defined in Assumption~\ref{ass:regconvrateunif}, and the radius~$\mu_{n}(x)$ satisfies $\norm{\mu_{n}(X)}_{L^q} = O(n^{-r/2})$.
Then, the solution of the ER-DRO problem~\eqref{eqn:dro} satisfies}
\begin{align*}
{\norm*{\hv^{DRO}_n(X) - v^*(X)}_{L^q}} &{= O_p(n^{-r/2}),} \\ 
{\norm*{g(\hz^{DRO}_n(X);X) - v^*(X)}_{L^q}} &{= O_p(n^{-r/2}).}
\end{align*}
\end{theorem}
\begin{proof}
{See Appendix~\ref{app:samprobustrateofconvlq}.}
\end{proof}

Finally, we make the following assumption to establish a finite sample certificate-type guarantee for sample robust optimization-based ER-DRO, i.e., when the radius $\zeta_{n}(x) \equiv 0$.
To achieve this, we utilize a connection between sample robust optimization-based ambiguity sets and ambiguity sets defined using the $\infty$-Wasserstein distance.
In particular, Theorem~5 of~\cite{bertsimas2018data} implies that the sample robust optimization-based ER-DRO problem is equivalent to the $\infty$-Wasserstein distance-based ER-DRO problem~\eqref{eqn:dro} with ambiguity set $\hP_n(x) := \big\{Q \in \mathcal{P}(\Y): d_{W,\infty}(Q,\hat{P}^{ER}_n(x)) \leq \mu_n(x)\big\}$.

\begin{assumption}
\label{ass:contsdist}
For a.e.\ $x \in \X$, the conditional distribution~$P_{Y \mid X = x}$ has a density $\Lambda_Y(\cdot;x) : \bar{\Y} \to [0,+\infty)$, where $\bar{\Y} \subset \Y$ is an open, connected and bounded set with a Lipschitz boundary.
Furthermore, for each $y \in \bar{\Y}$ and a.e.\ $x \in \X$, the density satisfies $1/\lambda(x) \leq \Lambda_Y(y;x) \leq \lambda(x)$, for some $\lambda(x) \geq 1$.
\end{assumption}

Trillos and Slep{\v{c}}ev~\cite{trillos2015rate} consider cases when Assumption~\ref{ass:contsdist} holds.
This assumption yields the following concentration of measure result for the true empirical distribution $P^*_n(x)$.
Note that Lemma~\ref{lem:wassinfmeasureconc} also applies to settings with non-i.i.d.\ data~$\D_n$ such as time series data.

\begin{lemma}[Theorem~1.1 of~\cite{trillos2015rate}]
\label{lem:wassinfmeasureconc}
Suppose Assumption~\ref{ass:contsdist} holds and the samples $\{\varepsilon^i\}_{i=1}^{n}$ are i.i.d.
Then, for any constant $\beta > 2$ and a.e.\ $x \in \X$
\[
\mathbb{P}\bigg\{d_{W,\infty}( P^*_n(x), P_{Y \mid X=x} ) \geq O(1) \frac{\log(n)}{n^{1/d_y}} \bigg\} \leq O(n^{-\beta/2}),
\]
where the $O(1)$ term depends only on~$\beta$, $\bar{\Y}$, and $\lambda(x)$ in Assumption~\ref{ass:contsdist}.
\end{lemma}

The next result is the analogue of Lemma~\ref{lem:wass_dist_bound} for the $\infty$-Wasserstein distance.

\begin{lemma}
\label{lem:wassinf_dist_bound}
For each $x \in \X$
\[
d_{W,\infty}( \hat{P}^{ER}_n(x), P_{Y \mid X=x} ) \leq 2\sup_{x \in \X} \norm{f^*(x) - \hf_n(x)} + d_{W,\infty}( P^*_n(x), P_{Y \mid X=x} ).
\]
\end{lemma}
\begin{proof}
The triangle inequality for the $\infty$-Wasserstein distance yields 
\[
d_{W,\infty}( \hat{P}^{ER}_n(x), P_{Y \mid X=x} ) \leq d_{W,\infty}( \hat{P}^{ER}_n(x), P^*_n(x) ) + d_{W,\infty}( P^*_n(x), P_{Y \mid X=x} ).
\]
The result then follows from~\eqref{eqn:projlipschitz} and the definition of $d_{W,\infty}$, which yield
\begin{align*}
d_{W,\infty}( \hat{P}^{ER}_n(x), P^*_n(x) ) &\leq \uset{i \in [n]}{\sup} \norm{\proj{\Y}{\hf_n(x) + \heps^i_{n}} - (f^*(x) + \varepsilon^i)} \\
&\leq \uset{i \in [n]}{\sup} \norm{(\hf_n(x) + \heps^i_{n}) - (f^*(x) + \varepsilon^i)} \\
&\leq 2\sup_{x \in \X} \norm{f^*(x) - \hf_n(x)}. \qedhere
\end{align*}
\end{proof}

For a given realization $x \in \X$ and risk level $\alpha \in (0,1)$, we hereafter use
\begin{align}
\label{eqn:wassinfradius}
\zeta_{n}(\alpha,x) := 0, \quad \mu_{n}(\alpha,x) := \kappa^{(1)}_{\infty,n}(\alpha) +  \kappa^{(2)}_{\infty,n}(x)
\end{align}
as the radii for the sample robust optimization-based ambiguity set, where 
\[
\kappa^{(1)}_{\infty,n}(\alpha) := 2\kappa_{n}(\alpha), \quad \kappa^{(2)}_{\infty,n}(x) := O(1) n^{-\theta/d_y},
\]
the constant $\kappa_{n}$ is defined in Assumption~\ref{ass:reglargedevunif} and the constant {$0 < \theta < 1$ may be chosen arbitrarily close to one.
The term $\kappa^{(2)}_{\infty,n}(x)$ above is chosen so that it is greater than the $O(1)\log(n)/n^{1/d_y}$ term in Lemma~\ref{lem:wassinfmeasureconc} for $\beta = 4$ and $n$ large enough.}
Similar to the specification of the Wasserstein DRO radius in~\eqref{eqn:wassradius}, the sample robust optimization radius~$\mu_n$ equals the sum of two contributions---the first accounts for the error in estimating~$f^*$, and the second corresponds to the radius used in the absence of covariate information~\cite{bertsimas2019two}.
While the above choice of~$\mu_{n}$ helps us derive our theoretical guarantees, it involves unknown constants and is typically conservative in practice (cf.\ Remark~\ref{rem:wassdrorate}).
We investigate practical approaches for choosing the radius $\mu_{n}$ in Section~\ref{subsec:wass_radius}.

\begin{theorem}[Finite sample certificate-type guarantee]
\label{thm:wassinffinitesampcert}
Suppose Assumptions~\ref{ass:reglargedevunif} and~\ref{ass:contsdist} hold, the samples~$\{\varepsilon^i\}_{i=1}^{n}$ are i.i.d., there exists a sequence of risk levels $\{\alpha_n\}_{n \in \mathbb{N}} \subset (0,1)$ such that $\sum_{n} \alpha_n < +\infty$, and for a.e.\ $x \in \X$, $\lim_{n \to \infty} \mu_{n}(\alpha_n,x) = 0$ with~$\mu_{n}$ defined in equation~\eqref{eqn:wassinfradius}.
Then, for a.e.\ $x \in \X$, there exists $N(x) \in \mathbb{N}$ such that the solution of the ER-DRO problem~\eqref{eqn:dro} with radii $\zeta_{n}(\alpha_n,x)$ and $\mu_{n}(\alpha_n,x)$ specified by equation~\eqref{eqn:wassinfradius} a.s.\ satisfies
\[
g(\hz^{DRO}_n(x);x) \leq \hv^{DRO}_n(x), \quad \forall n \geq N(x).
\]
\end{theorem}
\begin{proof}
Our proof follows the outline of the proof of~\cite[Theorem~1]{bertsimas2019two}.

Lemma~\ref{lem:wassinf_dist_bound}, the probability inequality used in the proof of Lemma~\ref{lem:wassfinitesamp}, and Assumption~\ref{ass:reglargedevunif} yield for a.e.\ $x \in \X$
\begin{align*}
&\mathbb{P}\big\{d_{W,\infty}( \hat{P}^{ER}_n(x), P_{Y \mid X=x} ) > \mu_{n}(\alpha_n,x)\big\} \\
\leq& \: \alpha_n + \mathbb{P}\big\{d_{W,\infty}( P^*_n(x), P_{Y \mid X=x}) > \kappa^{(2)}_{\infty,n}(x) \big\}.
\end{align*}
Consider $\beta = 4$ in Lemma~\ref{lem:wassinfmeasureconc}.
Because $\kappa^{(2)}_{\infty,n}(x) \geq O(1) \log(n)/n^{1/d_y}$ for $n$ large enough, we have from Lemma~\ref{lem:wassinfmeasureconc} that for a.e.\ $x \in \X$ and $n$ large enough
\[
\mathbb{P}\big\{d_{W,\infty}( \hat{P}^{ER}_n(x), P_{Y \mid X=x} ) > \mu_{n}(\alpha_n,x)\big\} \leq \alpha_n + O(n^{-2}).
\]
Therefore, we have $\sum_{n=1}^{\infty} \mathbb{P}\big\{d_{W,\infty}( \hat{P}^{ER}_n(x), P_{Y \mid X=x} ) > \mu_{n}(\alpha_n,x)\big\} < +\infty$.
The Borel-Cantelli lemma then implies that for a.e.\ $x \in \X$, there a.s.\ exists $N(x) \in \mathbb{N}$ such that for $n \geq N(x)$, $d_{W,\infty}( \hat{P}^{ER}_n(x), P_{Y \mid X=x} ) \leq \mu_{n}(\alpha_n,x)$.

Recall that our sample robust optimization-based ER-DRO problem is equivalent to the $\infty$-Wasserstein distance-based ER-DRO problem with ambiguity set $\hP_n(x) := \big\{Q \in \mathcal{P}(\Y): d_{W,\infty}(Q,\hat{P}^{ER}_n(x)) \leq \mu_n(\alpha_n,x)\big\}$~\cite[Theorem~5]{bertsimas2018data}.
The stated result then follows by the definition of the $\infty$-Wasserstein distance-based ER-DRO problem~\eqref{eqn:dro}.
\end{proof}

Hereafter, we revert to the shortened notation~$\zeta_n(x)$ and also use it to denote the radius of sample robust optimization ambiguity sets for simplicity.

\begin{algorithm}[h]
{\small 
\caption{Specifying a covariate-independent radius~$\zeta_n$ using a naive SAA-based DRO problem}
\label{alg:naivesaaradius}
{
\begin{algorithmic}[1]
\State \textbf{Input}: data~$\D_n$, set of candidate radii $\Delta$, and number of folds $K$.

\State Partition $[n]$ into $K$ subsets $S_1, \dots, S_K$ of (roughly) equal size at random.

\For{$k = 1, \dots, K$}

\For{$\zeta \in \Delta$}

\State Solve the following DRO problem to get a solution~$\hz^{DRO}_{-k}(\zeta)$:
\begin{align*}
&\uset{z \in \Z}{\min} \: \uset{Q \in \hP_{-k}}{\sup} \: \expectation{Y \sim Q}{c(z,Y)},
\end{align*}
\Statex \hspace*{0.4in} where the ambigiuity set $\hP_{-k}$ with radius~$\zeta$ is centered at the empirical

\Statex \hspace*{0.4in} distribution $\tilde{P}_{-k} := \dfrac{1}{n-\abs{S_k}} \displaystyle\sum_{i \in [n] \backslash S_k} \delta_{y^i}$.

\EndFor

\EndFor

\State \textbf{Output}: Radius $\zeta_n \in \uset{\zeta \in \Delta}{\argmin} \dfrac{1}{K} \displaystyle\sum_{k \in [K]} \dfrac{1}{\abs{S_k}}\sum_{i\in S_k} c(\hz^{DRO}_{-k}(\zeta), y^i)$ of the ambiguity set $\hP_n(x)$ for the ER-DRO problem~\eqref{eqn:dro}.

\end{algorithmic}
}
}
\end{algorithm}

\begin{algorithm}
{\small
\caption{Specifying a covariate-independent radius~$\zeta_n$ using the ER-DRO problem}
\label{alg:ersaasameradius}
{
\begin{algorithmic}[1]
\State \textbf{Input}: data~$\D_n$, set of candidate radii $\Delta$, number of folds $K$, and number of covariate realizations sampled during each fold~$T \leq \lfloor \frac{n}{K} \rfloor$.

\State Partition $[n]$ into subsets $S_1, \dots, S_K$ of (roughly) equal size at random. Let $\D_{-k} := \D_n \backslash \{(y^i,x^i)\}_{i \in S_k}$.

\For{$k = 1, \dots, K$}

\State Pick without replacement a random subset~$\bar{\X}$ of $\{x^i\}_{i \in S_k}$ of size $T$.

\For{$\bar{x} \in \bar{\X}$}

\For{$\zeta \in \Delta$}

\State Fit a regression model $\hf_{-k}$ using the data~$\D_{-k}$ and compute its in-sample

\Statex \hspace*{0.6in} residuals $\{\heps^i_{-k}\}_{i \not\in S_k} := \{y^i - \hf_{-k}(x^i)\}_{i \not\in S_k}$.

\State Solve the ER-DRO problem below at covariate $\bar{x}$ to get solution $\hz^{DRO}_{-k}(\bar{x},\zeta)$
\begin{align*}
&\uset{z \in \Z}{\min} \: \uset{Q \in \hP_{-k}(\bar{x})}{\sup} \: \expectation{Y \sim Q}{c(z,Y)},
\end{align*}%
\Statex \hspace*{0.6in} where the ambigiuity set $\hP_{-k}(\bar{x})$ with radius~$\zeta$ is centered at the

\Statex \hspace*{0.6in} estimated empirical distribution $\hat{P}^{ER}_{-k}(\bar{x}) := \dfrac{1}{n-\abs{S_k}} \displaystyle\sum_{i \not\in S_k} \delta_{\hf_{-k}(\bar{x})+\heps^i_{-k}}$. 

\EndFor

\EndFor

\EndFor

\State \textbf{Output}: Radius $\zeta_n \in \uset{\zeta \in \Delta}{\argmin} \dfrac{1}{T} \displaystyle\sum_{\bar{x} \in \bar{\X}} \dfrac{1}{K}\displaystyle\sum_{k \in [K]}\dfrac{1}{\abs{S_k}} \displaystyle\sum_{i\in S_k} c(\hz^{DRO}_{-k}(\bar{x},\zeta), y^i)$ for the ambiguity set $\hP_n(x)$ for the ER-DRO problem~\eqref{eqn:dro}.

\end{algorithmic}
}
}
\end{algorithm}

\begin{algorithm}
{\small 
\caption{Specifying a covariate-dependent radius~$\zeta_n(x)$ using the ER-DRO problem}
\label{alg:ersaadiffradius}
{
\begin{algorithmic}[1]
\State \textbf{Input}: data~$\D_n$, set of candidate radii $\Delta$, number of folds $K$, and new covariate realization~$x \in \X$.

\State Partition $[n]$ into subsets $S_1, \dots, S_K$ of (roughly) equal size at random. Let $\D_{-k} := \D_n \backslash \{(y^i,x^i)\}_{i \in S_k}$.

\For{$k = 1, \dots, K$}

\For{$\zeta \in \Delta$}

\State Fit a regression model $\hf_{-k}$ using the data~$\D_{-k}$ and compute its in-sample

\Statex \hspace*{0.4in} residuals $\{\heps^i_{-k}\}_{i \not\in S_k} := \{y^i - \hf_{-k}(x^i)\}_{i \not\in S_k}$.

\State Solve the ER-DRO problem below at covariate $x$ to obtain  solution $\hz^{DRO}_{-k}(x,\zeta)$
\begin{align*}
&\uset{z \in \Z}{\min} \: \uset{Q \in \hP_{-k}(x)}{\sup} \: \expectation{Y \sim Q}{c(z,Y)},
\end{align*}
\Statex \hspace*{0.4in} where the ambigiuity set $\hP_{-k}(x)$ with radius~$\zeta$ is centered at the estimated

\Statex \hspace*{0.4in} empirical distribution $\hat{P}^{ER}_{-k}(x) := \dfrac{1}{n-\abs{S_k}} \displaystyle\sum_{i \not\in S_k} \delta_{\hf_{-k}(x)+\heps^i_{-k}}$.

\State Fit a regression model $\hf_{k}$ using the data $\{(y^i,x^i)\}_{i \in S_k}$ and compute its

\Statex \hspace*{0.4in} in-sample residuals $\{\heps^i_{k}\}_{i \in S_k} := \{y^i - \hf_{k}(x^i)\}_{i \in S_{k}}$.

\EndFor

\EndFor

\State \mbox{\textbf{Output}: Radius $\zeta_n(x) \in \uset{\zeta \in \Delta}{\argmin} \dfrac{1}{K}\displaystyle\sum_{k \in [K]}\dfrac{1}{\abs{S_k}} \displaystyle\sum_{i\in S_k} c(\hz^{DRO}_{-k}(x,\zeta), \hf_{k}(x) + \heps^i_{k})$} for the ambiguity set $\hP_n(x)$ for the ER-DRO problem~\eqref{eqn:dro}.

\end{algorithmic}
}
}
\end{algorithm}

\section{Specifying the radius of the ambiguity set} 
\label{subsec:wass_radius}

Determining the optimal radius~$\zeta_n(x)$ of the ambiguity sets in Section~\ref{sec:drsaa} using the theory in Sections~\ref{sec:wassdro} and~\ref{sec:samprobustopt} is hard for two reasons: (i) the theory usually involves unknown constants, and (ii) even if these constants are known or estimated, this specification of~$\zeta_n(x)$ is typically conservative in practice (see Remark~\ref{rem:wassdrorate}).
Therefore, we propose data-driven approaches that use cross-validation (CV) to specify~$\zeta_n(x)$ for the ER-DRO problem~\eqref{eqn:dro} with the goal of minimizing the out-of-sample cost $g(\hz^{DRO}_n(x);x)$ of the resulting ER-DRO solution~$\hz^{DRO}_n(x)$.
Once we choose $\zeta_n(x)$, we re-solve the ER-DRO problem~\eqref{eqn:dro} with the ambiguity set of radius~$\zeta_n(x)$ centered at the empirical distribution $\hat{P}^{ER}_{n}(x)$ to determine the optimal value $\hv^{DRO}_n(x)$ and a solution $\hz^{DRO}_n(x)$.

We outline two approaches, Algorithms~\ref{alg:naivesaaradius} and~\ref{alg:ersaasameradius}, 
for choosing the radius $\zeta_n(x)$ independently of the covariate realization~$x \in \X$.
Algorithm~\ref{alg:naivesaaradius} ignores covariate information altogether, whereas Algorithm~\ref{alg:ersaasameradius} uses all of the data~$\D_n$, including covariates, but does not use the new covariate realization~$x \in \X$ for specifying the radius.
Algorithm~\ref{alg:ersaadiffradius} presents an alternative that also uses the realization~$x \in \X$ to choose~$\zeta_n(x)$.
Algorithms~\ref{alg:naivesaaradius} and~\ref{alg:ersaasameradius} are less data and computation intensive and can be readily used in applications where the DRO problem~\eqref{eqn:dro} is repeatedly solved for different covariate realizations.
Allowing~$\zeta_n(x)$ to depend on the realization~$x \in \X$, on the other hand, could yield estimators with better out-of-sample performance, which might justify the added computational cost of Algorithm~\ref{alg:ersaadiffradius}.

Algorithm~\ref{alg:naivesaaradius} chooses a covariate-independent radius~$\zeta_n$ for the ambiguity set~$\hP_n(x)$ using $K$-fold CV on a DRO extension of a naive SAA problem that does not use covariate information (cf.\ \cite[Section~7.2.2]{esfahani2018data}).
This algorithm does not require estimation of the regression function~$f^*$.
The radius~$\zeta_n$ determined using Algorithm~\ref{alg:naivesaaradius} necessarily converges to zero as the sample size~$n$ increases.
This may result in suboptimal estimators~$\hz^{DRO}_n(x)$ when the prediction model is misspecified ({cf.\ Remark~\ref{rem:modelclass} in Section~\ref{sec:prelim}}), in which case it may be beneficial to use a positive value of~$\zeta_n$ even for large values of~$n$ (cf.\ Figure~\ref{fig:comp_ols_wass_radii} in Appendix~\ref{app:computres}).
Algorithm~\ref{alg:ersaasameradius} determines a covariate-independent radius~$\zeta_n$ using $K$-fold CV on ER-DRO problems.
Note that the objective in line~12 of Algorithm~\ref{alg:ersaasameradius} for choosing the radius~$\zeta_n$ is similar to the objective in line~8 of Algorithm~\ref{alg:naivesaaradius}.

Algorithm~\ref{alg:ersaadiffradius} determines a covariate-dependent radius~$\zeta_n(x)$ using $K$-fold CV on the ER-DRO problem~\eqref{eqn:dro}.
For each fold, this algorithm estimates the regression function~$f^*$ twice: once using the data omitted in the fold for setting up the ER-DRO problem~\eqref{eqn:dro}, and once using the data in the fold for estimating the out-of-sample costs of the constructed DRO solutions.
{The motivation for estimating the function~$f^*$ a second time is to approximate the following radius selection problem that uses $f^*$ only to construct $\abs{S_k}$ i.i.d.\ samples from the conditional distribution $P_{Y \mid X = x}$ for evaluating the quality of the $K$-fold CV-based ER-DRO solutions $\hz^{DRO}_{-k}(x,\zeta)$:}
\[
{
\zeta^*_n(x) \in \uset{\zeta \in \Delta}{\argmin} \dfrac{1}{K}\displaystyle\sum_{k \in [K]}\dfrac{1}{\abs{S_k}} \displaystyle\sum_{i\in S_k} c(\hz^{DRO}_{-k}(x,\zeta), f^*(x) + \varepsilon^i_{k}).
}
\]
Clearly, there is a trade-off between the number of data samples used to construct each estimate of~$f^*$.
Because we are particularly interested in the limited data regime, we propose to use a sparse estimation technique (such as the Lasso) for the second estimation step (i.e., for line~7 of Algorithm~\ref{alg:ersaadiffradius}).

\section{Computational experiments}
\label{sec:computexp}

We consider instances of the following mean-risk portfolio optimization model adapted from~\cite{esfahani2018data}:
\[
\uset{z \in \Z}{\min} \: \expect{-\tr{Y}z} + \rho \: \text{CVaR}_{\beta}(-\tr{Y}z) ,
\]
where $\Z := \big\{z \in \R^{d_z}_+ : \sum_j z_j = 1\big\}$, $\rho$ and $\beta$ are given parameters, and
\[
\text{CVaR}_{\beta}(-\tr{Y}z) := \uset{\tau \in \R}{\min} \: \expect{\tau + \frac{1}{1-\beta} \max\{0, -\tr{Y}z -  \tau\}}.
\]
{We can rewrite this model as the following single-stage stochastic program:
\[
\uset{z \in \Z, \tau \in \R}{\min} \expect{-\tr{Y}z + \rho\tau + \frac{\rho}{1-\beta} \max\{0, -\tr{Y}z -  \tau\}}.
\]
The variable $\tau$ can be bounded under mild conditions on the distribution of~$Y$ (see~\citep[Theorem~10]{rockafellar2002conditional}).}
For each $j \in [d_z]$, the decision variable $z_j$ denotes the fraction of capital invested in asset $j$ and the random variable $Y_j$ denotes the net return of asset~$j$.
The parameters $\rho \geq 0$ and $\beta \in (0,1)$ specify the decision-maker's risk aversion level, with $\text{CVaR}_{\beta}$ (roughly) averaging over the $100(1-\beta)\%$ worst return outcomes under the distribution of $Y$.
Following~\cite{esfahani2018data}, we use $\beta = 0.8$, $\rho = 10$, and $d_y = d_z = 10$.

Similar to~\cite{kannan2020data}, we assume that the returns $Y$ satisfy the relationship
\[
Y_j = \nu^*_j + \sum_{l \in \mathcal{L}^*} {\mu^*_{\theta,jl}} (X_l)^{\theta} + \bar{\varepsilon}_j + \omega, \quad \forall j \in [d_y],
\]
where $X_l$, $l \in \mathcal{L}$, are covariates, $\theta \in \{0.5,1,2\}$ is a fixed parameter that determines the model class, $\bar{\varepsilon}_j \sim {\mathcal{N}\left(0,0.025j\right)}$ and $\omega \sim \mathcal{N}(0,0.02)$ are additive errors {whose variances are chosen to match the case study in~\citep[Section~7.2]{esfahani2018data}}, $\nu^*$ and {$\mu^*_{\theta}$} are model parameters, and $\mathcal{L}^* \subseteq \mathcal{L}$ contains the indices of the covariates with predictive power ($\mathcal{L}^*$ does not depend on the index~$j \in [d_y]$). 
{Note that $\abs{\mathcal{L}} = d_x$}.
We draw covariate samples $\{x^i\}_{i=1}^{n}$ from a multivariate \textit{folded-normal/half-normal} distribution with the underlying normal distribution having zero mean and covariance matrix equal to a random correlation matrix generated using the \textit{vine method} of~\cite{lewandowski2009generating}.
Throughout, we assume that $\abs{\mathcal{L}^*} = 3$, i.e., the returns truly depend only on three covariates.
We simulate i.i.d.\ data $\D_n$ with {
\begin{alignat*}{2}
&\nu^*_j = 0.005j, \qquad &&\mu^*_{\theta,j2} = \bigl(0.0075j \bigr) s_{\theta} \xi_{j2}, \\
&\mu^*_{\theta,j3} = \bigl(0.005j \bigr) s_{\theta} \xi_{j3}, \qquad &&\mu^*_{\theta,j1} = 0.025j s_{\theta} - \mu^*_{\theta,j2} - \mu^*_{\theta,j3}
\end{alignat*}
}%
for each $j \in [d_y]$,
where $\xi_{j2}$ and $\xi_{j3}$ are i.i.d.\ samples from the uniform distribution {$U(0.8,1.2)$ and the scaling factor $s_{\theta}$ is (approximately) $1.25$, $1.22$, and $1$ when the exponent $\theta$ is equal to $1$, $0.5$, and $2$, respectively.
The above coefficients are chosen such that $\mathbb{E}_{X,\bar{\varepsilon},\omega}[Y_j \mid X] = 0.03j$, $\forall j \in [d_y]$, which mirrors the setup in~\citep[Section~7.2]{esfahani2018data} (the scaling factor $s_{\theta}$ offsets the differences in the term $\mathbb{E}_X[(X_l)^{\theta}]$ for $\theta \in \{1,0.5,2\}$).
Once the coefficients $\nu^*$ and $\mu^*_{\theta}$ are generated, they are considered fixed for different replications of the data $\D_n$.
}

Given joint data $\D_n$ on the random returns and random covariates, we estimate the coefficients of the linear model
\[
Y_j = \nu_j + \sum_{l \in \mathcal{L}} \mu_{jl} X_l + \eta_j, \quad \forall j \in [d_y],
\]
where $\eta_j$ are zero-mean errors, using OLS, Lasso, {or Ridge}  regression and use this model within our residuals-based formulations.
We use this linear model even when the degree $\theta \neq 1$, in which case it is misspecified.
Note that OLS, {Lasso, and Ridge} regression estimate $d_x + 1$ parameters for each $j \in [d_y]$.

We compare the ER-SAA formulation~\eqref{eqn:app} (denoted by \texttt{E}) 
with ER-DRO formulations that use the $1$-Wasserstein-based ambiguity set defined using the $\ell_1$-norm (denoted by \texttt{W}), the sample robust optimization-based ambiguity set constructed using the $\ell_1$-norm (denoted by \texttt{S}), and the ambiguity set with the same support as $\hat{P}^{ER}_n(x)$ defined using the Hellinger distance (denoted by \texttt{H}, see Example~\ref{exm:phidivergence} in Section~\ref{sec:drsaa}).
Different from the setup in Section~\ref{sec:drsaa}, we use the $\ell_1$-norm to define the $1$-Wasserstein and sample robust optimization-based ambiguity sets so that the resulting ER-DRO problems can be expressed as LPs~\cite{esfahani2018data}.
Formulation~\texttt{H} can be expressed as a conic quadratic program~\cite{ben2013robust}.

We vary the dimension $d_x$ of the covariates, the sample size~$n$, and the degree~$\theta$ in our computational experiments.
We use Algorithms~\ref{alg:naivesaaradius},~\ref{alg:ersaasameradius}, and~\ref{alg:ersaadiffradius} to specify the radii~$\zeta_n(x)$ of the above ambiguity sets for the ER-DRO problem~\eqref{eqn:dro} with $K = 5$ folds in all three cases and $T = \min\{50, \lfloor \frac{n}{5} \rfloor\}$ in Algorithm~\ref{alg:ersaasameradius}.
We use Lasso regression in line 7 of Algorithm~\ref{alg:ersaadiffradius} with $5$-fold CV.
For all ER-DRO formulations, following~\cite{esfahani2018data}, we choose the radius~$\zeta_n(x)$ from the set of $28$ candidate points $\Set{b \times 10^{e}}{b \in \{0,1,\dots,9\}, \: e \in \{-1,-2,-3\}}$ instead of $\mathbb{R}_+$.

Solutions obtained from the different approaches are compared by estimating a normalized version of the upper bound of a $99\%$ confidence interval (UCB) on their {out-of-sample} optimality gaps using the multiple replication procedure (MRP)~\cite{mak1999monte} {(see Algorithm~\ref{alg:99ucbs} in Appendix~\ref{app:computres} for details)}.
We use {$20{,}000$} i.i.d.\ samples from the conditional distribution of $Y$ given $X = x$ to compute these UCBs.
Because the data-driven solutions depend on the realization of~$\D_n$, we perform $50$ data replications per test instance, sample $20$ different covariate realizations~$x \in \X$, and report our results in the form of box plots of these $50 \times 20 = 1000$ UCBs. The boxes denote the $25^{\text{th}}$, $50^{\text{th}}$, and $75^{\text{th}}$ percentiles of the $99\%$ UCBs, and the whiskers denote the {$5^{\text{th}}$ and $95^{\text{th}}$ percentiles} of the $99\%$ UCBs over the $1000$ instances.

Source code and data {are} available at \url{https://github.com/rohitkannan/ER-DRO}.
Our codes are written in Julia~0.6.4~\citep{bezanson2017julia}, use Gurobi 8.1.0 to solve LPs and conic quadratic programs through the JuMP~0.18.5 interface~\citep{dunning2017jump}, and use \texttt{glmnet}~0.3.0~\citep{friedman2010regularization} for Lasso {and Ridge} regression. 
All computational tests were conducted through the UW-Madison Center for High Throughput Computing (CHTC) software \texttt{HTCondor} (\url{http://chtc.cs.wisc.edu/}).
\\

\begin{figure}[t!]
    \centering
    \begin{subfigure}[t]{0.33\textwidth}
        \centering
        \includegraphics[width=\textwidth]{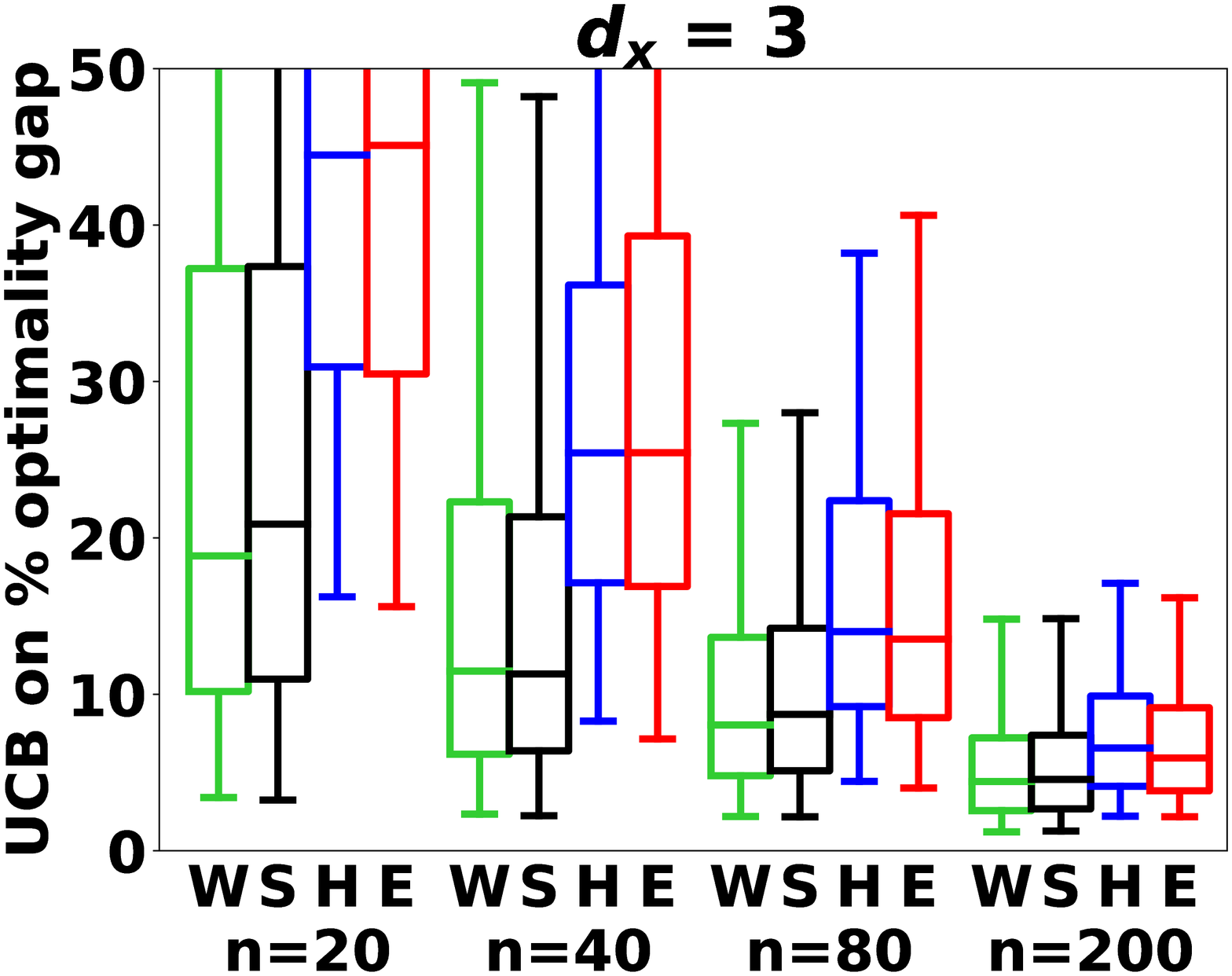}
    \end{subfigure}%
    ~ 
    \begin{subfigure}[t]{0.33\textwidth}
        \centering
        \includegraphics[width=\textwidth]{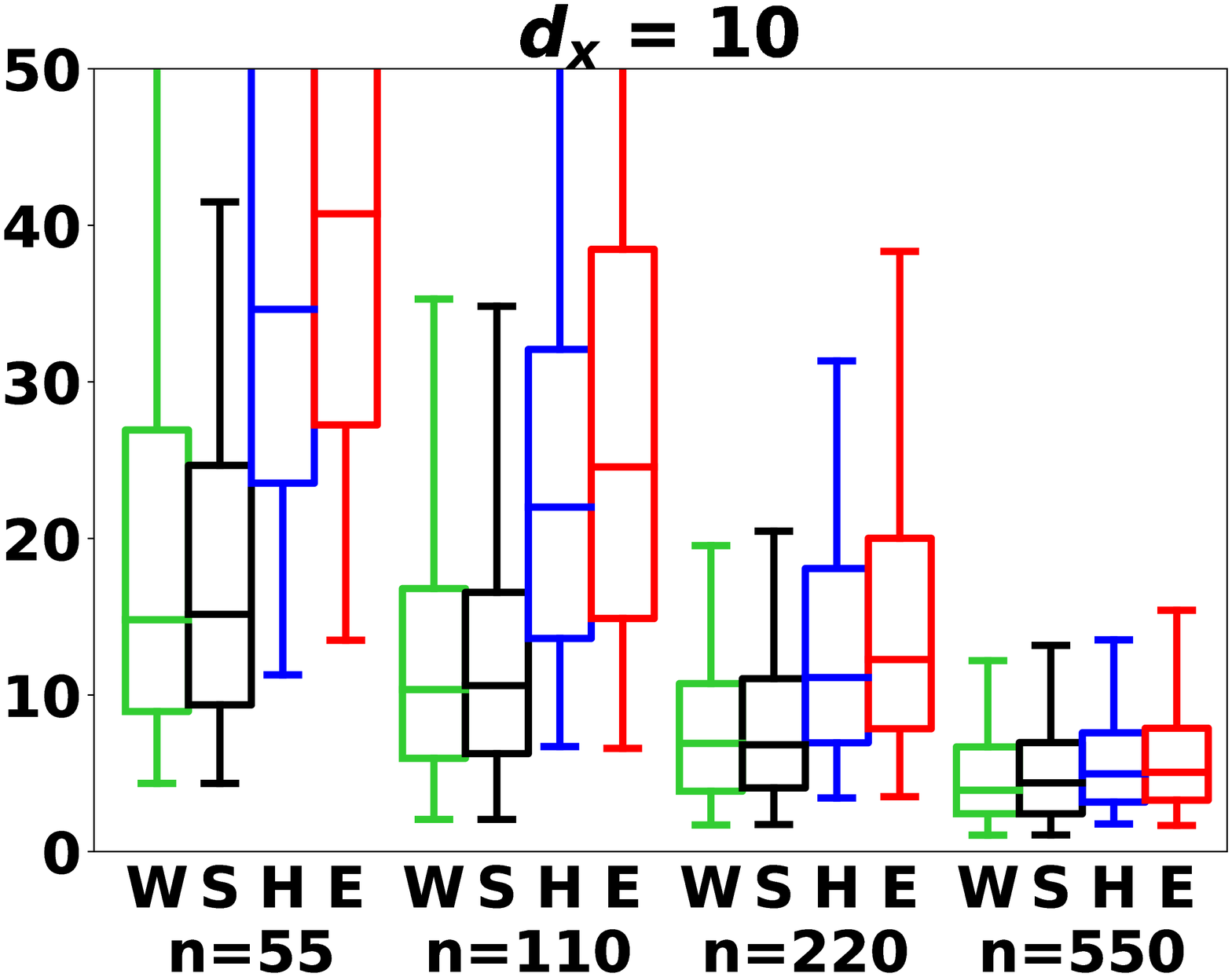}
    \end{subfigure}%
    ~ 
    \begin{subfigure}[t]{0.33\textwidth}
        \centering
        \includegraphics[width=\textwidth]{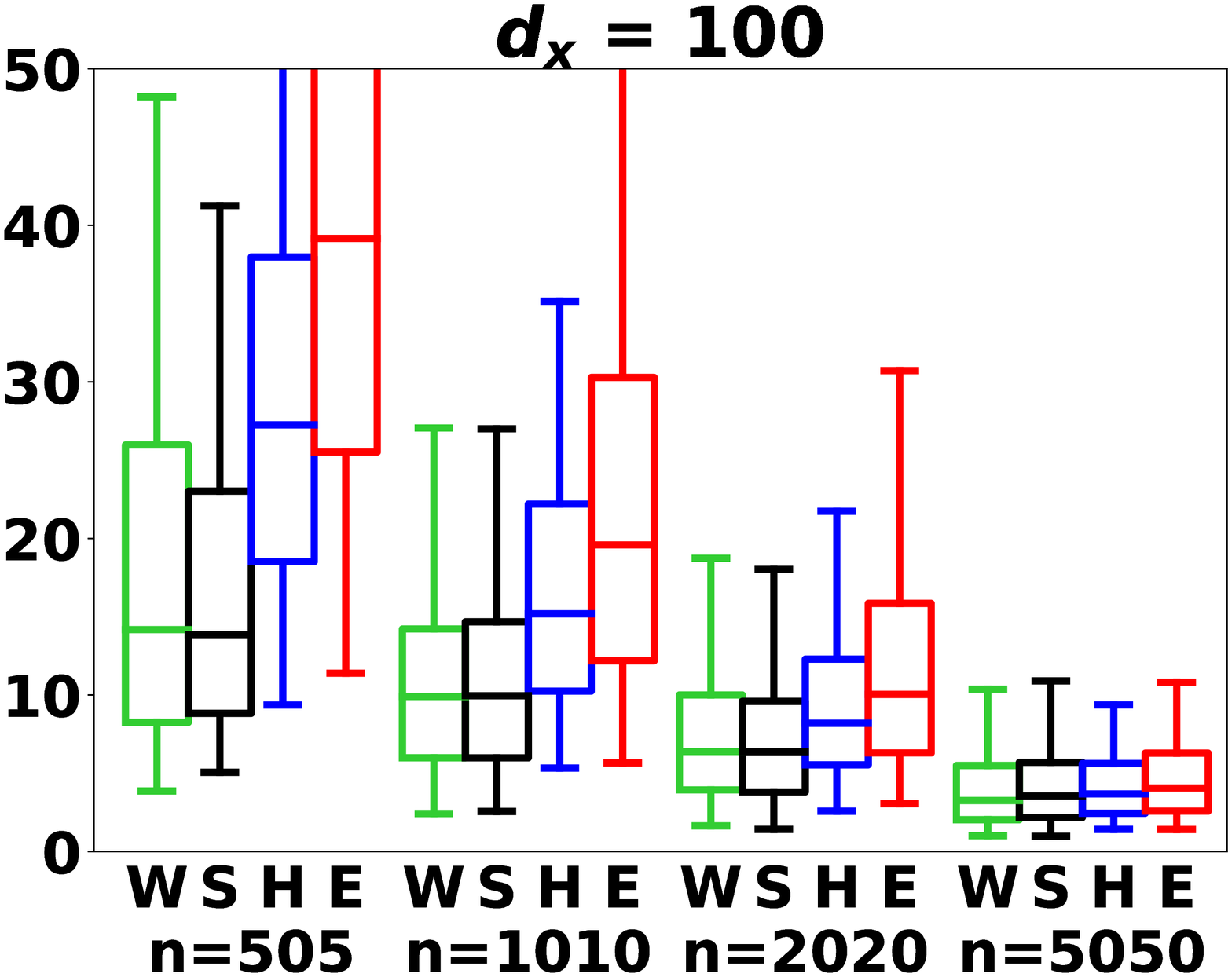}
    \end{subfigure}\\
    \begin{subfigure}[t]{0.33\textwidth}
        \centering
        \includegraphics[width=\textwidth]{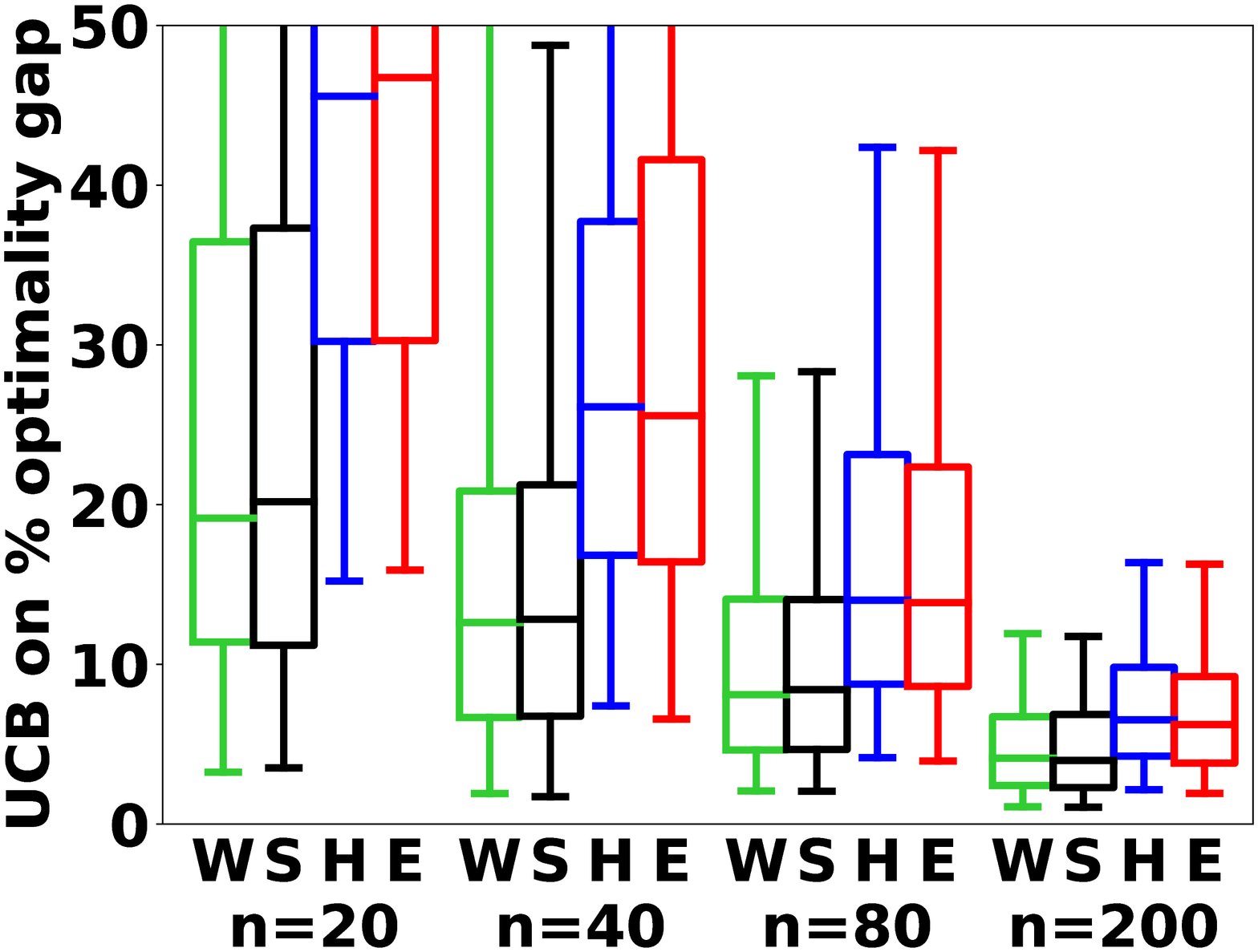}
    \end{subfigure}%
    ~ 
    \begin{subfigure}[t]{0.33\textwidth}
        \centering
        \includegraphics[width=\textwidth]{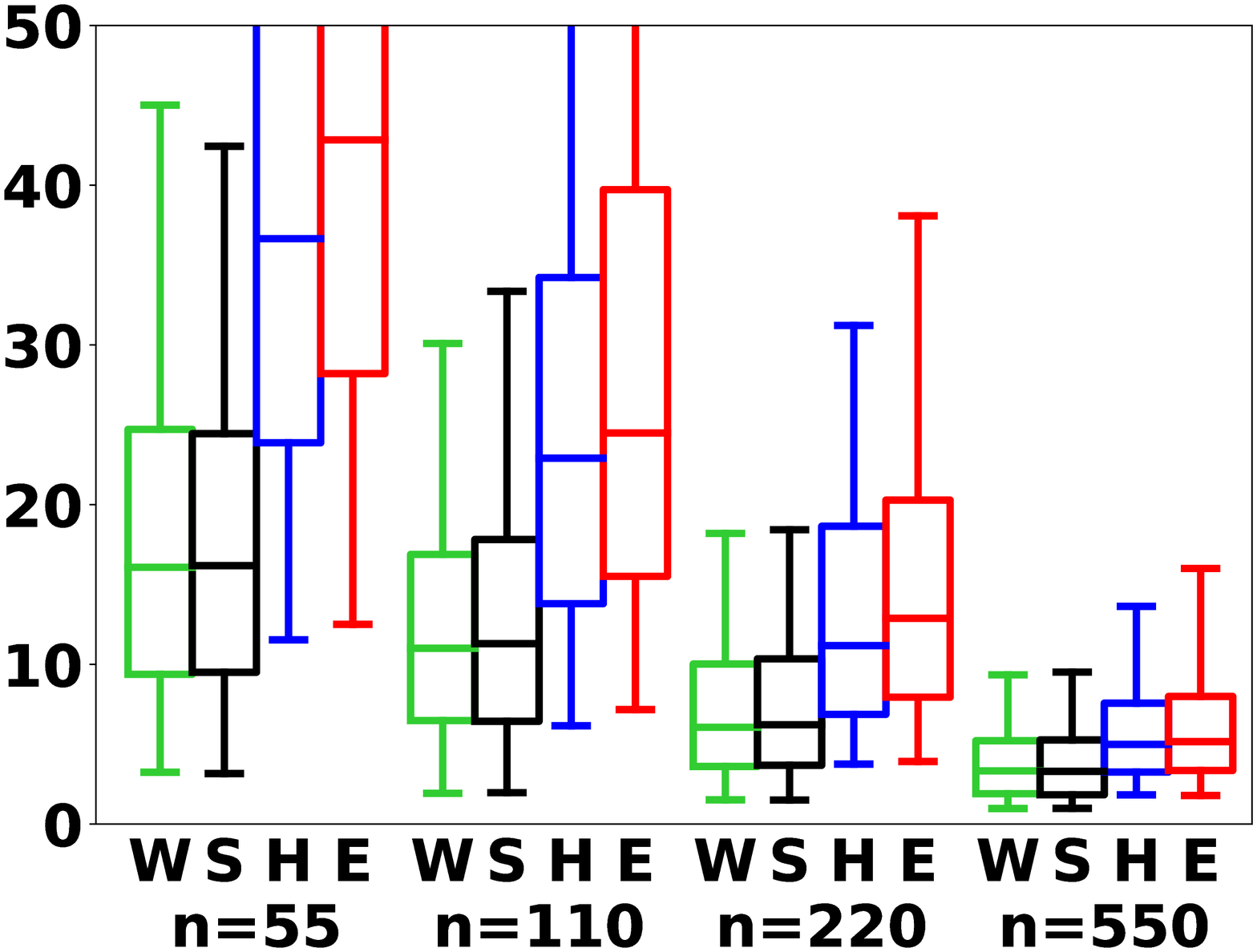}
    \end{subfigure}%
    ~ 
    \begin{subfigure}[t]{0.33\textwidth}
        \centering
        \includegraphics[width=\textwidth]{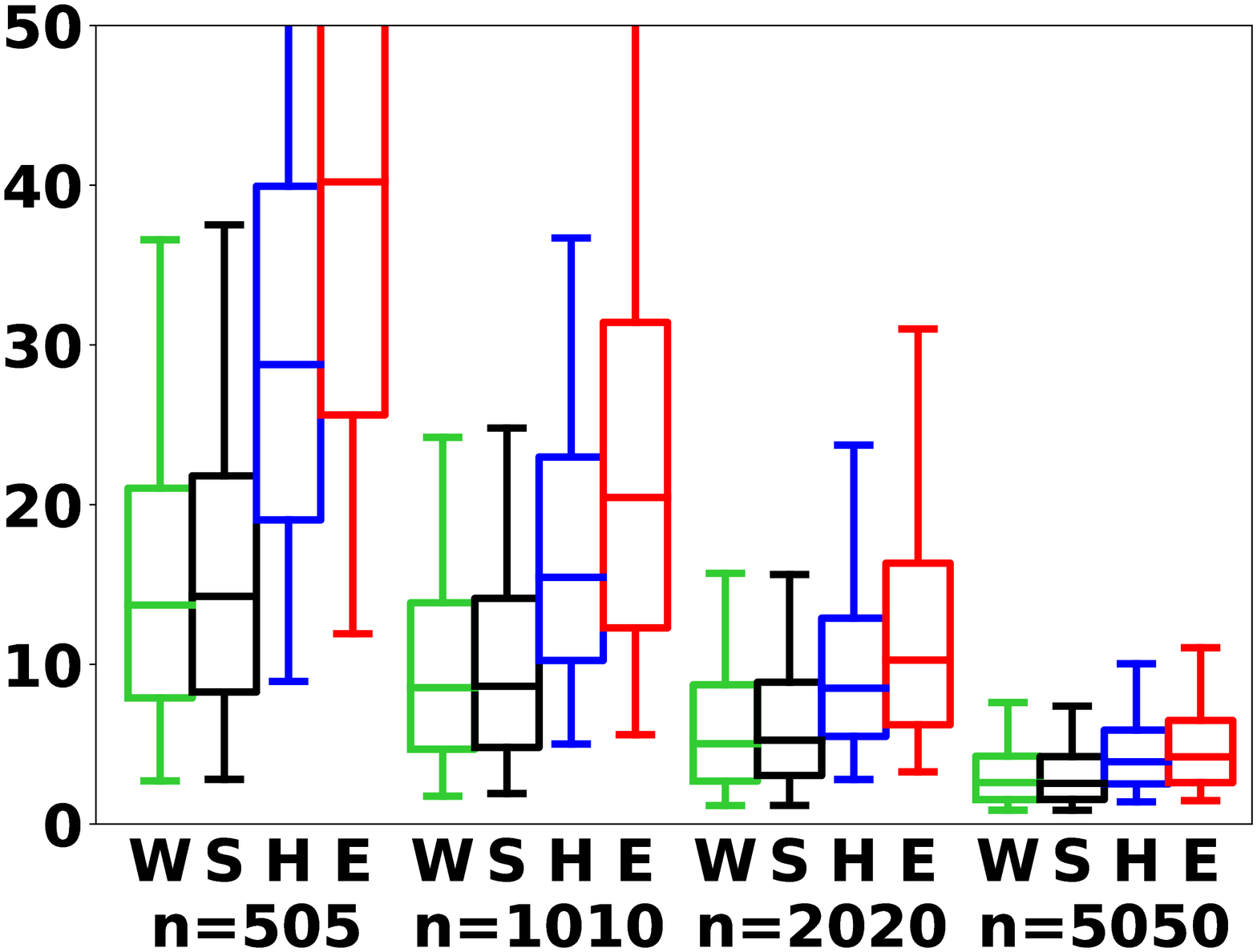}
    \end{subfigure}\\
    \begin{subfigure}[t]{0.33\textwidth}
        \centering
        \includegraphics[width=\textwidth]{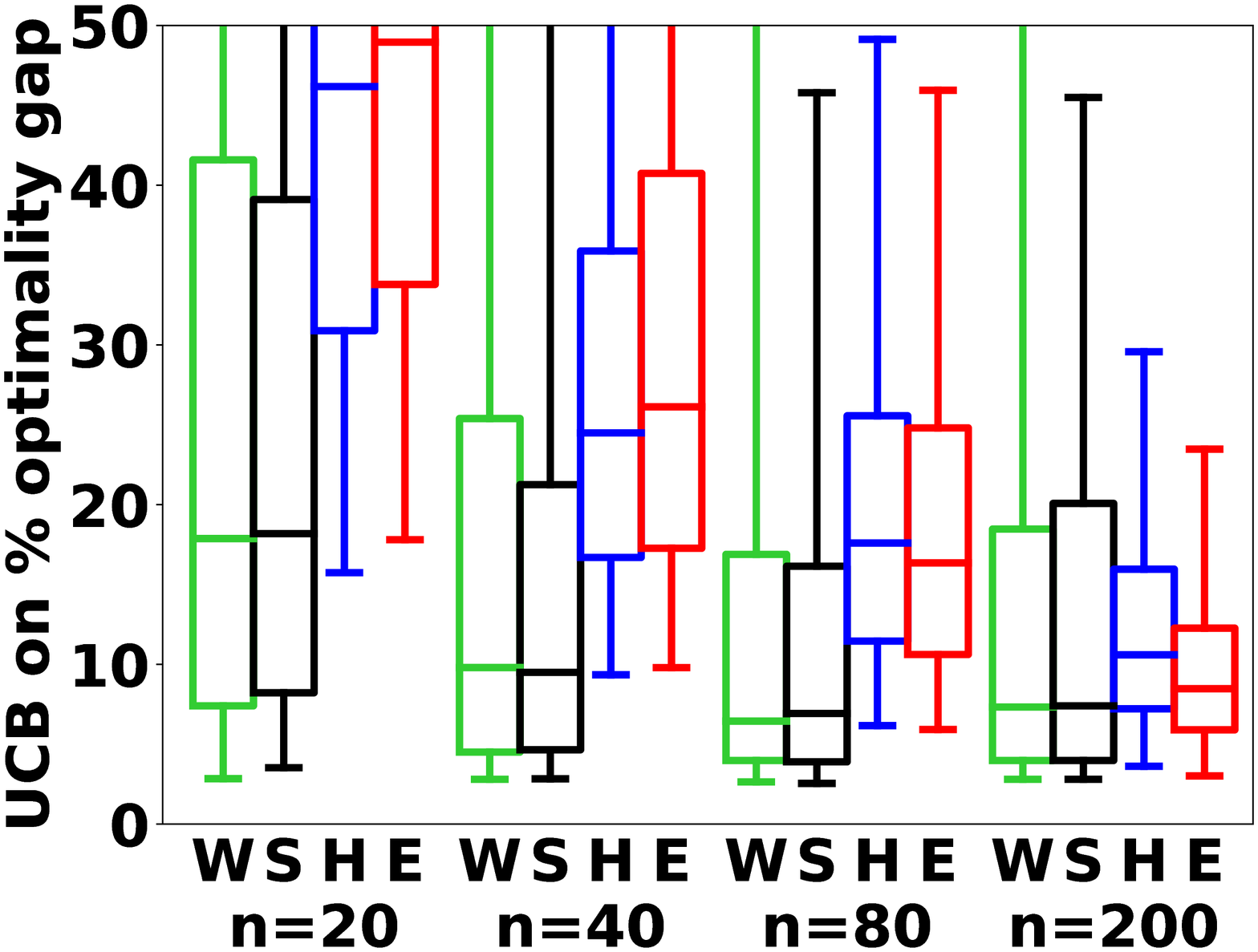}
    \end{subfigure}%
    ~ 
    \begin{subfigure}[t]{0.33\textwidth}
        \centering
        \includegraphics[width=\textwidth]{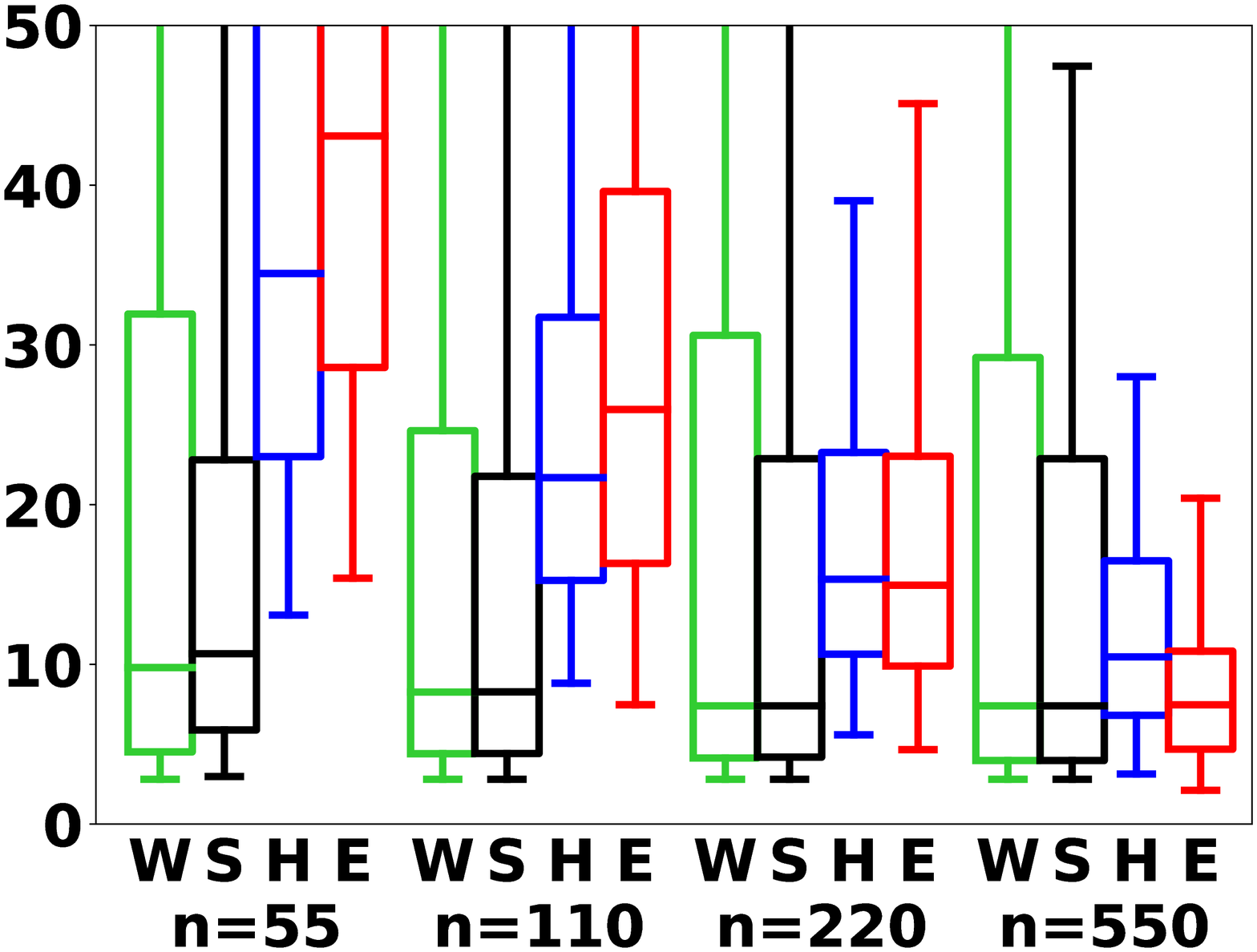}
    \end{subfigure}%
    ~ 
    \begin{subfigure}[t]{0.33\textwidth}
        \centering
        \includegraphics[width=\textwidth]{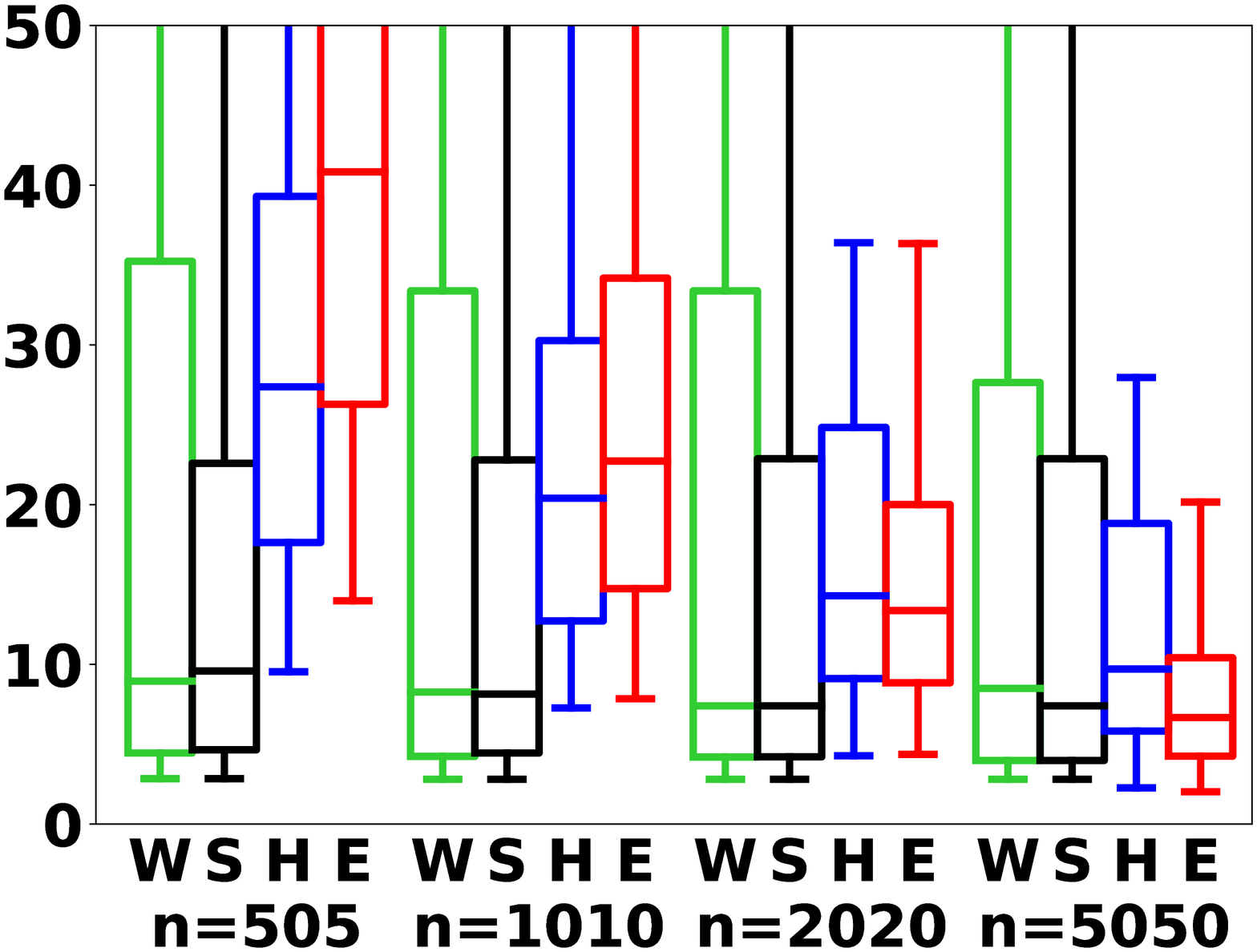}
    \end{subfigure}
    \caption{{\textbf{(Comparison of the different ER-DRO formulations)} Comparison of the \tE+OLS approach (\tE) with the covariate-independent tuning of the \tW+OLS radius (\tW), the \tS+OLS radius (\tS), and the \tH+OLS radius (\tH), all tuned using Algorithm~\ref{alg:ersaasameradius}. Top row: $\theta = 1$. Middle row: $\theta = 0.5$. Bottom row: $\theta = 2$. Left column: $d_x = 3$. Middle column: $d_x = 10$. Right column: $d_x = 100$.}}
    \label{fig:comp_ols_dros_indep}
\end{figure}

\noindent
\textbf{Comparison of the different ER-DRO formulations.}
Figure~\ref{fig:comp_ols_dros_indep} compares the performance of the \tE+OLS formulation with the \tW+OLS, \tS+OLS, and \tH+OLS formulations when the radius~$\zeta_n$ of the ambiguity sets of all three ER-DRO formulations are specified using Algorithm~\ref{alg:ersaasameradius}.
{We vary the model degree $\theta$, the covariate dimension among $d_x \in \{3,10,100\}$, and the sample size among $n \in \{5(d_x + 1),10(d_x + 1),20(d_x + 1),50(d_x + 1)\}$ in these experiments.
Note that OLS regression estimates $d_x + 1$ parameters for each $j \in \mathcal{J}$ even though the true model only contains $\abs{\mathcal{L}^*} + 1 = 4$ nonzero parameters for each $j$.}
The performance of the \tS+OLS formulation is similar to that of the \tW+OLS formulation {with the \tS+OLS formulation performing slightly better when $\theta = 2$}.
The \tH+OLS formulation does not significantly improve over the \tE+OLS formulation {for smaller covariate dimensions but provides an intermediate level of improvement relative to the \tW+OLS and \tS+OLS approaches for larger covariate dimensions}.
Recall that the Wasserstein (\tW) and sample robust optimization (\tS) ambiguity sets allow distributions with support different from $\hat{P}^{ER}_n(x)$, whereas the Hellinger (\tH) ambiguity set only considers distributions with the same support as $\hat{P}^{ER}_n(x)$. 
Because the data~$\D_n$ comes from a continuous distribution and $\hat{P}^{ER}_n(x)$ may be a crude estimate of $P^*_n(x)$ for small $n$, this highlights the advantage of DRO formulations that go beyond the estimated empirical distribution $\hat{P}^{ER}_n(x)$. 
From this point on, we do not include any more results for the $\tS$ formulations because they are similar to those of the $\tW$ formulations.
We also do not consider the $\tH$ formulations further.

{The optimality gaps of the \tE+OLS and the ER-DRO+OLS estimators are not guaranteed to converge to zero whenever $\theta \neq 1$ due to model misspecification; however, the \tW+OLS and \tS+OLS formulations are able to effectively mitigate the impact of model misspecification, especially for $\theta = 0.5$, in this case study.
Interestingly, choosing the radius~$\zeta_n$ of the ambiguity sets of all three ER-DRO formulations using Algorithm~\ref{alg:ersaasameradius} performs worse than the \tE+OLS formulation when $\theta = 2$ and $n$ is large.
This indicates that Algorithm~\ref{alg:ersaasameradius} may not yield a good choice of the radius~$\zeta_n$ for large sample sizes when model misspecification is significant.
From Figure~\ref{fig:comp_ols_wass_indep}, we see that Algorithm~\ref{alg:ersaadiffradius} provides a better-performing alternative to Algorithm~\ref{alg:ersaadiffradius} 
in this regime.}
{Note that relatively large optimality gaps for DRO at small sample sizes $n$ is to be expected for this case study because of the risk-averse nature of the portfolio objective.}

\begin{figure}[t!]
    \centering
    \begin{subfigure}[t]{0.33\textwidth}
        \centering
        \includegraphics[width=\textwidth]{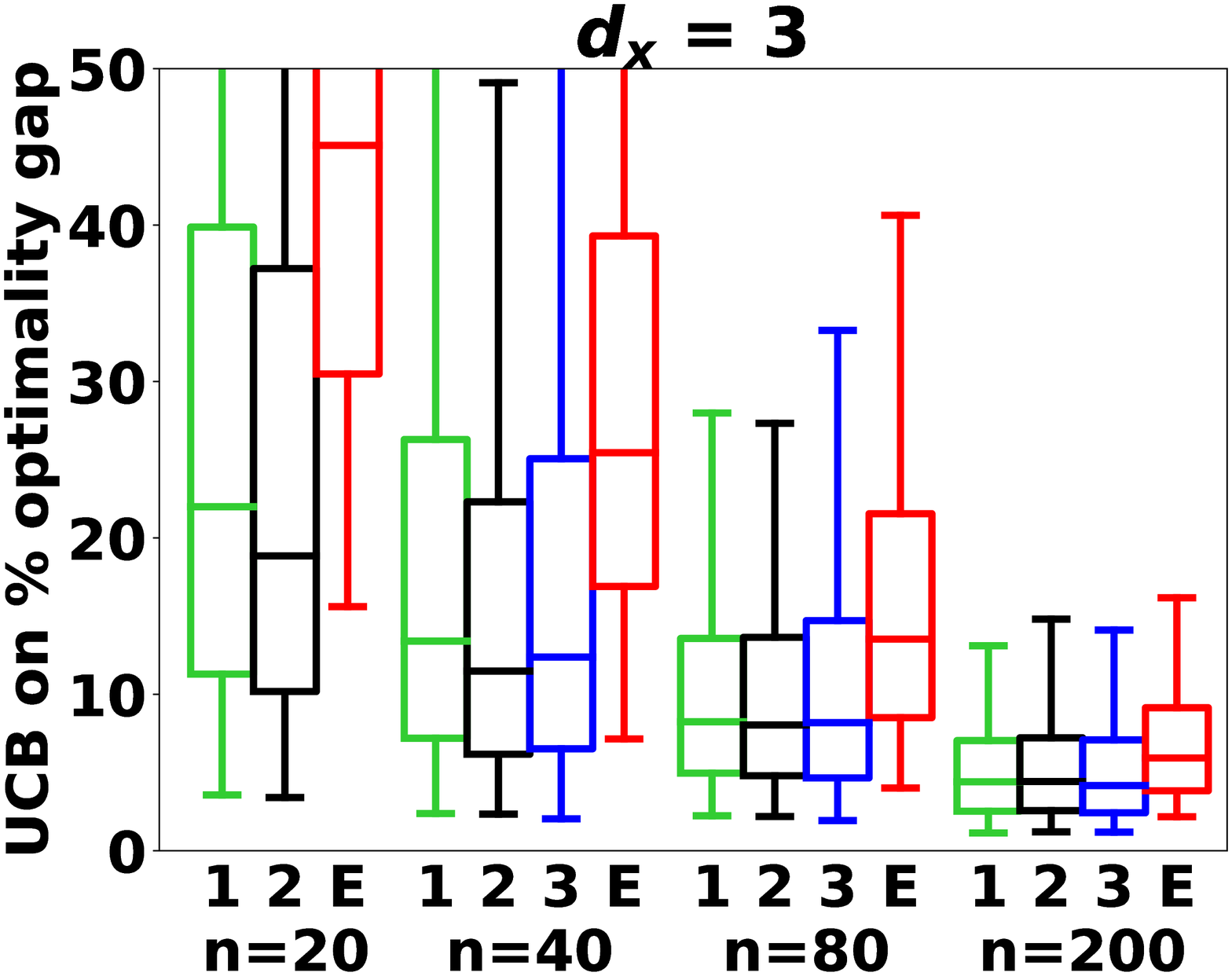}
    \end{subfigure}%
    ~ 
    \begin{subfigure}[t]{0.33\textwidth}
        \centering
        \includegraphics[width=\textwidth]{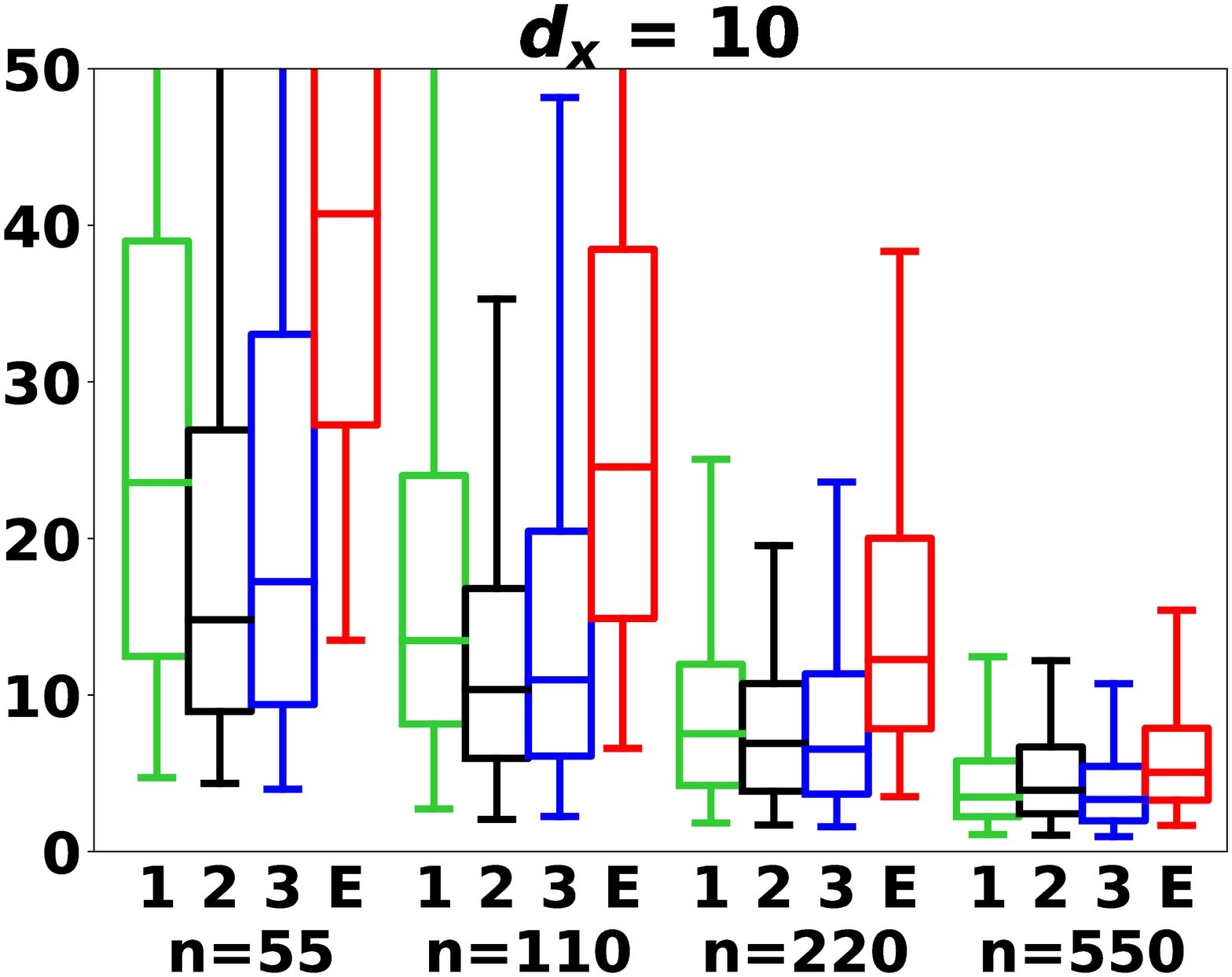}
    \end{subfigure}%
    ~ 
    \begin{subfigure}[t]{0.33\textwidth}
        \centering
        \includegraphics[width=\textwidth]{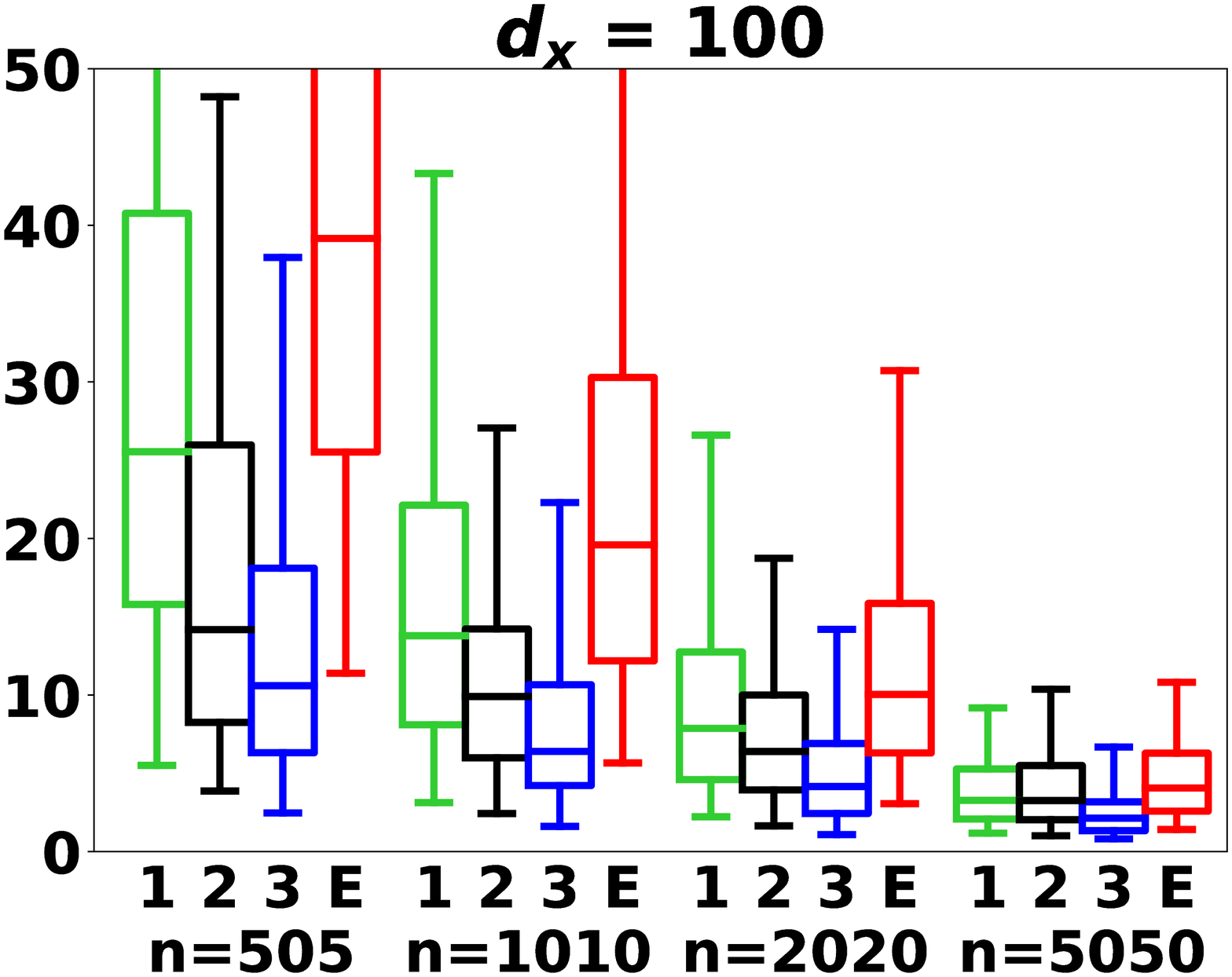}
    \end{subfigure}\\
    \begin{subfigure}[t]{0.33\textwidth}
        \centering
        \includegraphics[width=\textwidth]{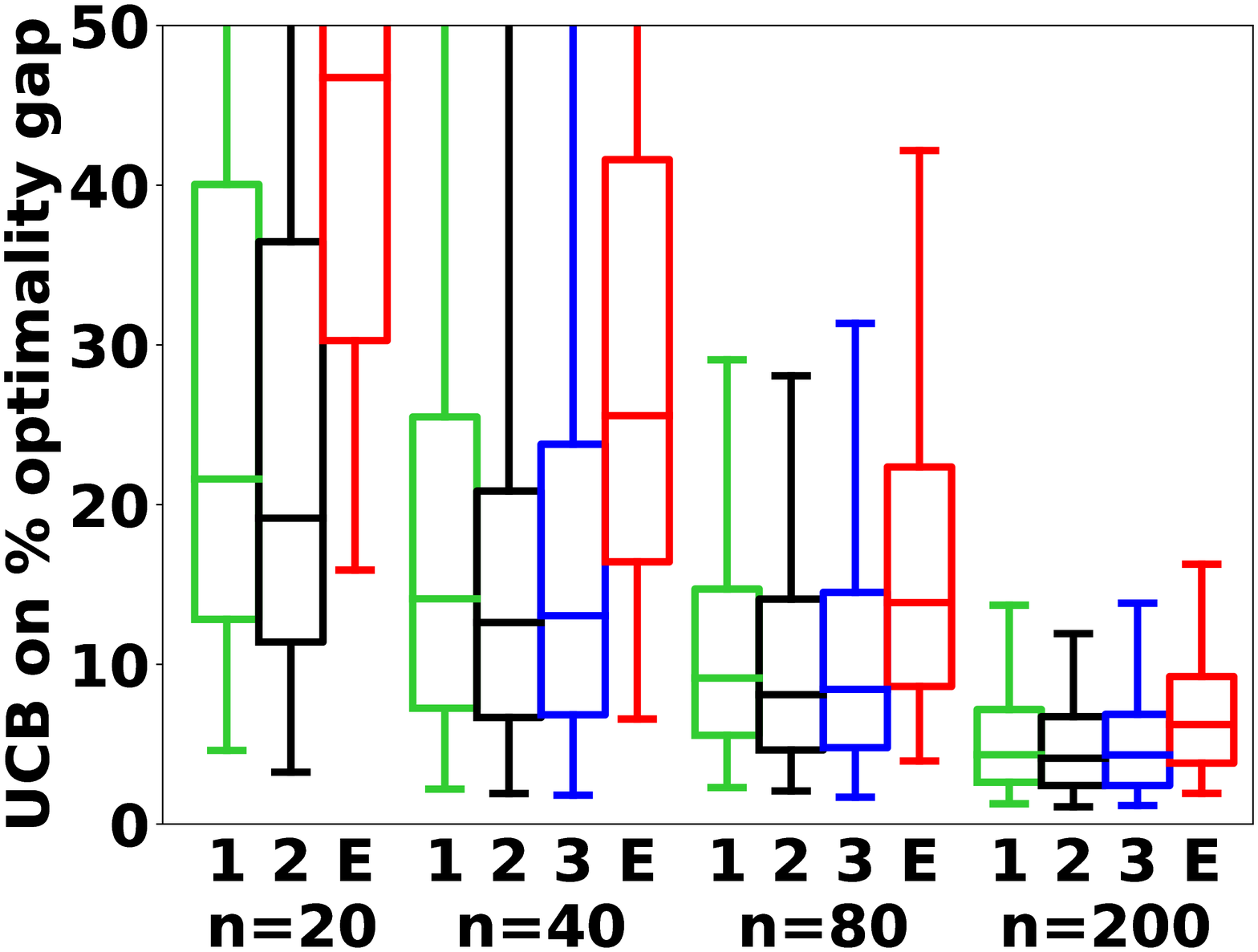}
    \end{subfigure}%
    ~ 
    \begin{subfigure}[t]{0.33\textwidth}
        \centering
        \includegraphics[width=\textwidth]{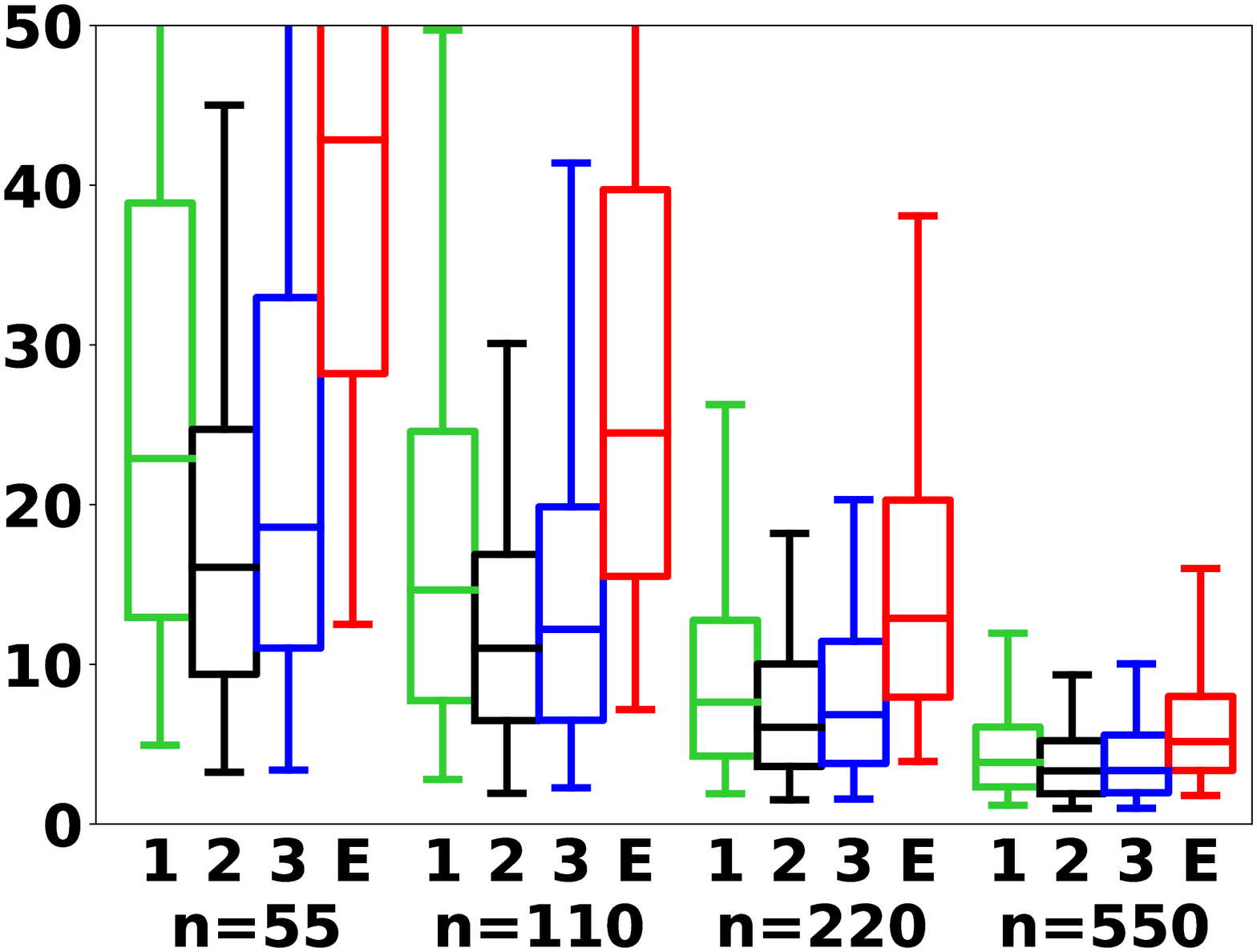}
    \end{subfigure}%
    ~ 
    \begin{subfigure}[t]{0.33\textwidth}
        \centering
        \includegraphics[width=\textwidth]{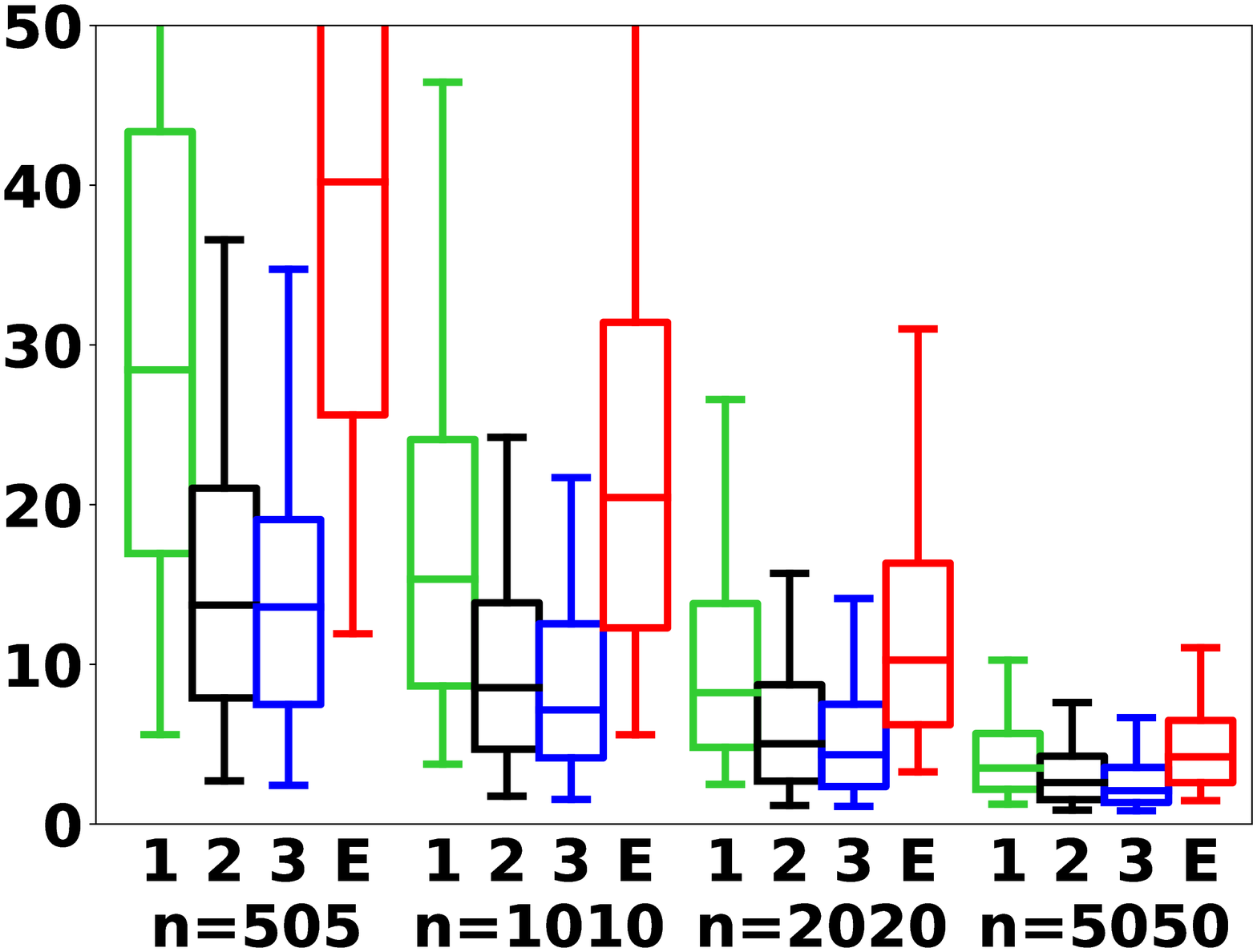}
    \end{subfigure}\\
    \begin{subfigure}[t]{0.33\textwidth}
        \centering
        \includegraphics[width=\textwidth]{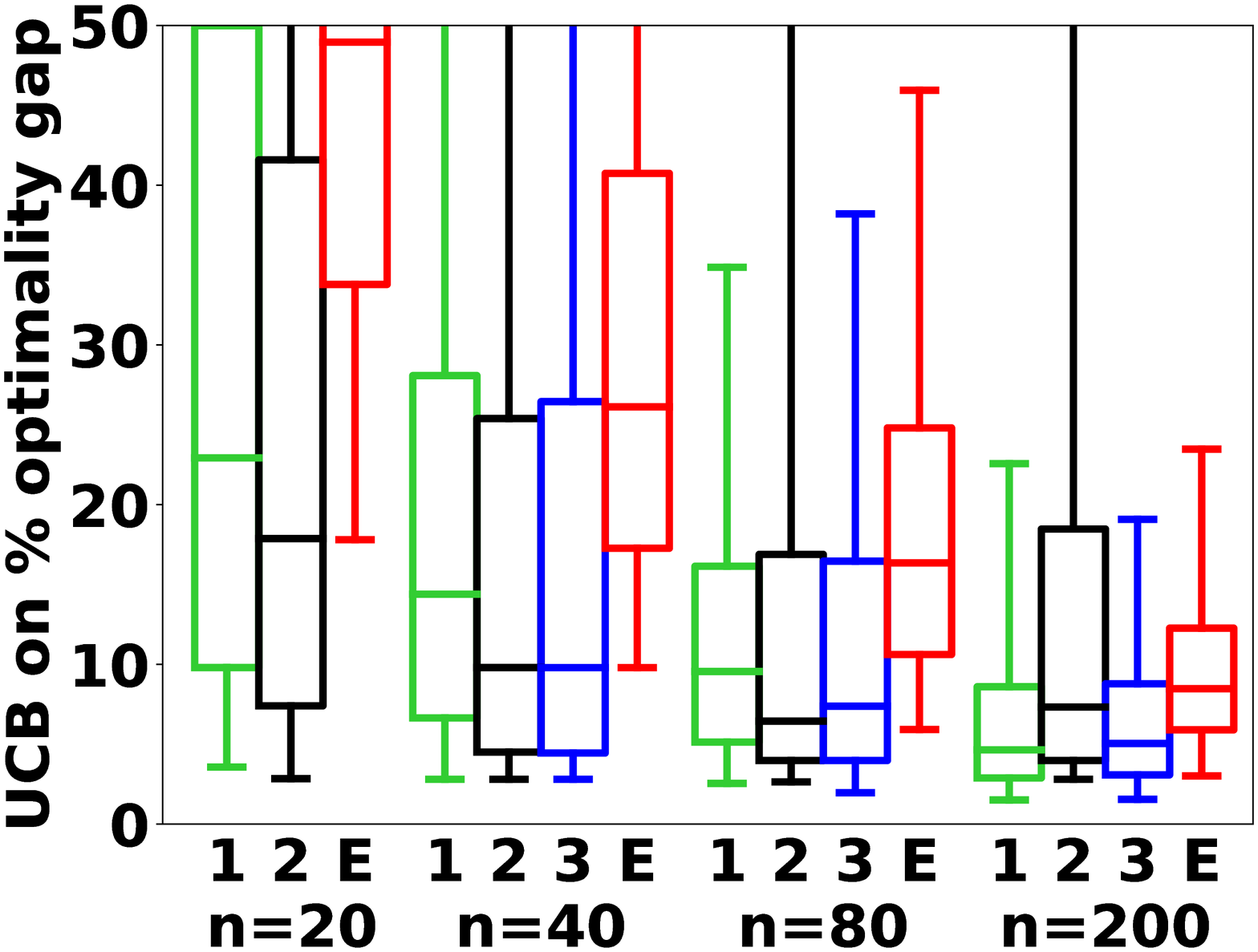}
    \end{subfigure}%
    ~ 
    \begin{subfigure}[t]{0.33\textwidth}
        \centering
        \includegraphics[width=\textwidth]{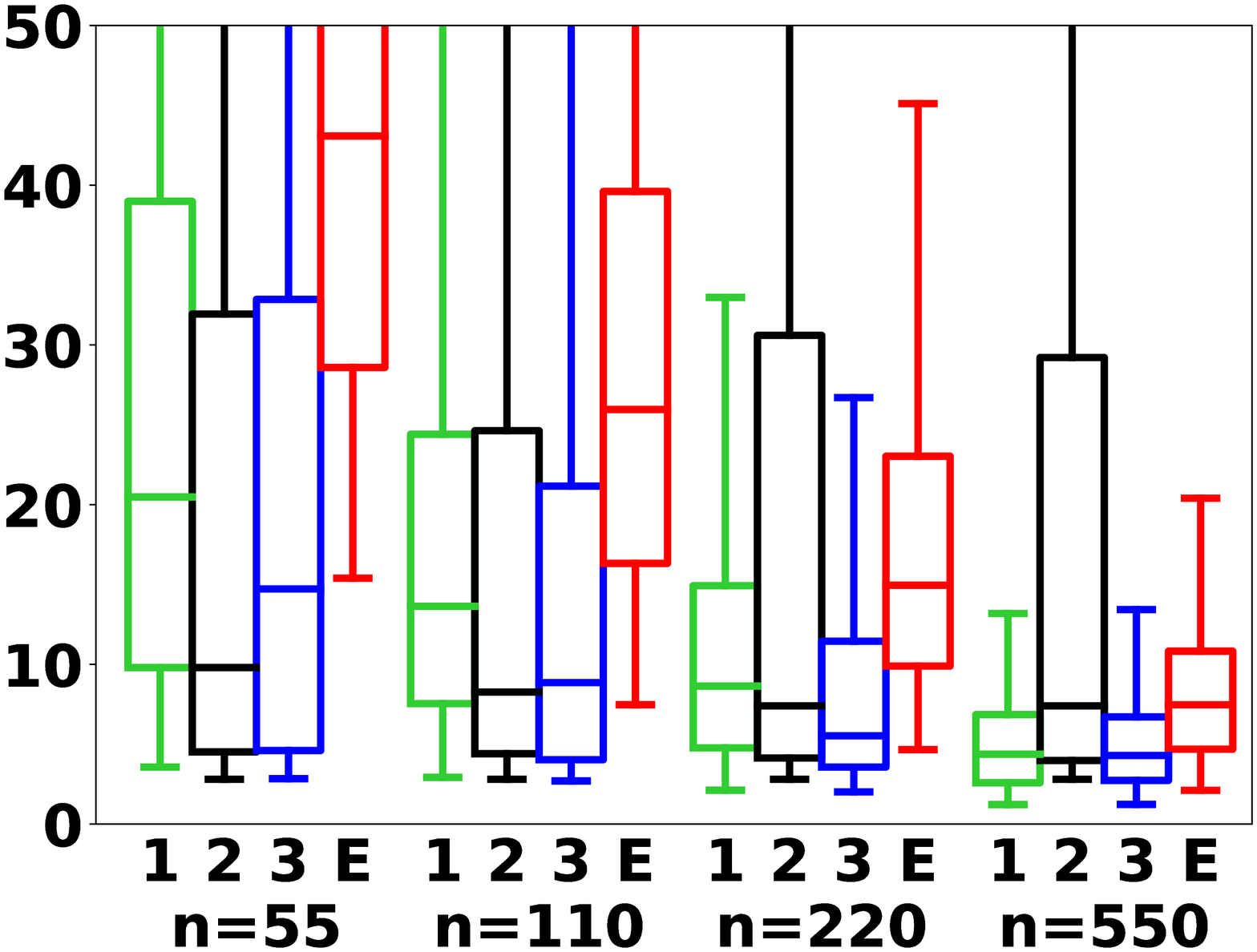}
    \end{subfigure}%
    ~ 
    \begin{subfigure}[t]{0.33\textwidth}
        \centering
        \includegraphics[width=\textwidth]{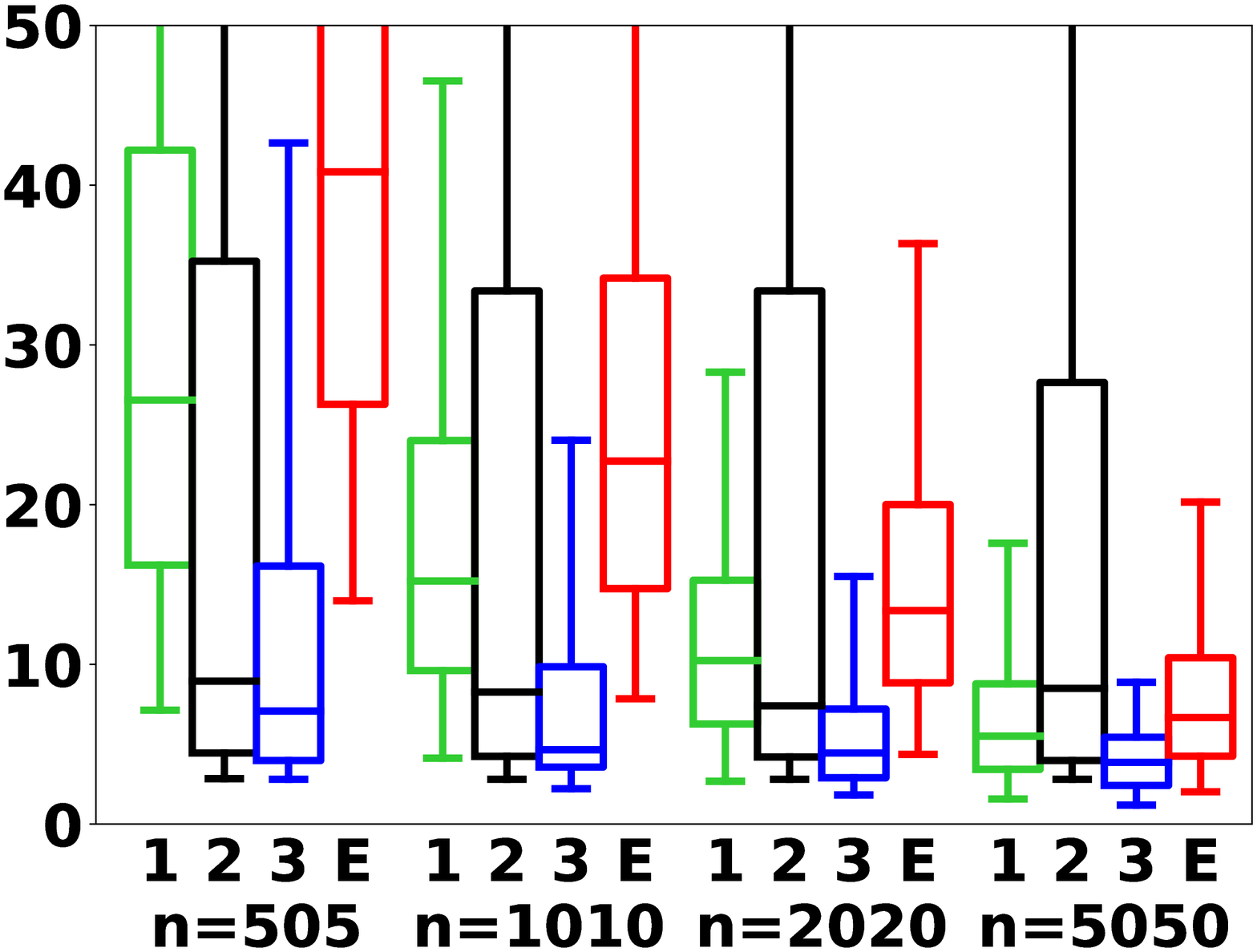}
    \end{subfigure}
    \caption{{\textbf{(Wasserstein-DRO with OLS regression)} Comparison of the \tE+OLS approach (\tE) with the tuning of the \tW+OLS radius using Algorithms~\ref{alg:naivesaaradius} (\tone),~\ref{alg:ersaasameradius} (\ttwo), and~\ref{alg:ersaadiffradius} (\tthree). Top row: $\theta = 1$. Middle row: $\theta = 0.5$. Bottom row: $\theta = 2$. Left column: $d_x = 3$. Middle column: $d_x = 10$. Right column: $d_x = 100$.}}
    \label{fig:comp_ols_wass_indep}
\end{figure}

\vspace*{0.1in}
\noindent
\textbf{Evaluation of the different radius selection strategies\footnote{We do not include results for Algorithm~\ref{alg:ersaadiffradius} with $n = 20$ because it requires at least $30$ samples (line~7 of Algorithm~\ref{alg:ersaadiffradius} needs at least $6$ points for Lasso regression with $5$-fold CV).}.}
Figure~\ref{fig:comp_ols_wass_indep} compares the performance of the \tE+OLS formulation with the \tW+OLS formulation {over the same range of parameter values as in Figure~\ref{fig:comp_ols_dros_indep}}.
The radius~$\zeta_n(x)$ of the ambiguity set for the \tW+OLS formulation is determined using Algorithms~\ref{alg:naivesaaradius} and~\ref{alg:ersaasameradius} {that pick covariate-independent radii and Algorithm~\ref{alg:ersaadiffradius} that picks a covariate-dependent radius}.
{Note that all three strategies for choosing the radius~$\zeta_n(x)$ are based on the same underlying Wasserstein ER-DRO formulation with OLS regression.}
In this small sample regime, the \tW+OLS formulations perform better than the \tE+OLS formulation across {almost} all cases---{the only exception again is when Algorithm~\ref{alg:ersaasameradius} is used for $\theta = 2$ and $n$ large.}
The radius specified by Algorithm~\ref{alg:ersaasameradius} performs better than the radius specified using Algorithm~\ref{alg:naivesaaradius} {in all other cases}, with the difference being most significant for larger covariate dimensions {and smaller sample sizes}.
{The radius specified by Algorithm~\ref{alg:ersaasameradius} performs better than the radius specified using Algorithm~\ref{alg:ersaadiffradius} for smaller sample sizes and covariate dimensions, and the converse holds for larger covariate dimensions and sample sizes}.
{When $\theta \neq 1$, the \tE+OLS and \tW+OLS approaches are not guaranteed to yield consistent estimators because the regression model is misspecified; however, Figure~\ref{fig:comp_ols_wass_indep} shows that the \tW+OLS formulation with the radius~$\zeta_n$ specified by Algorithm~\ref{alg:ersaasameradius} is able to effectively mitigate the impact of model misspecification for $\theta = 0.5$ in this case study.
Similarly, \tW+OLS with Algorithm~\ref{alg:ersaadiffradius} is able to effectively mitigate the impact of model misspecification for both $\theta = 0.5$ and $\theta = 2$.}
Finally, as expected, the benefits of the ER-DRO formulations diminish with increasing sample size.
\\

\begin{figure}[t!]
    \centering
    \begin{subfigure}[t]{0.33\textwidth}
        \centering
        \includegraphics[width=\textwidth]{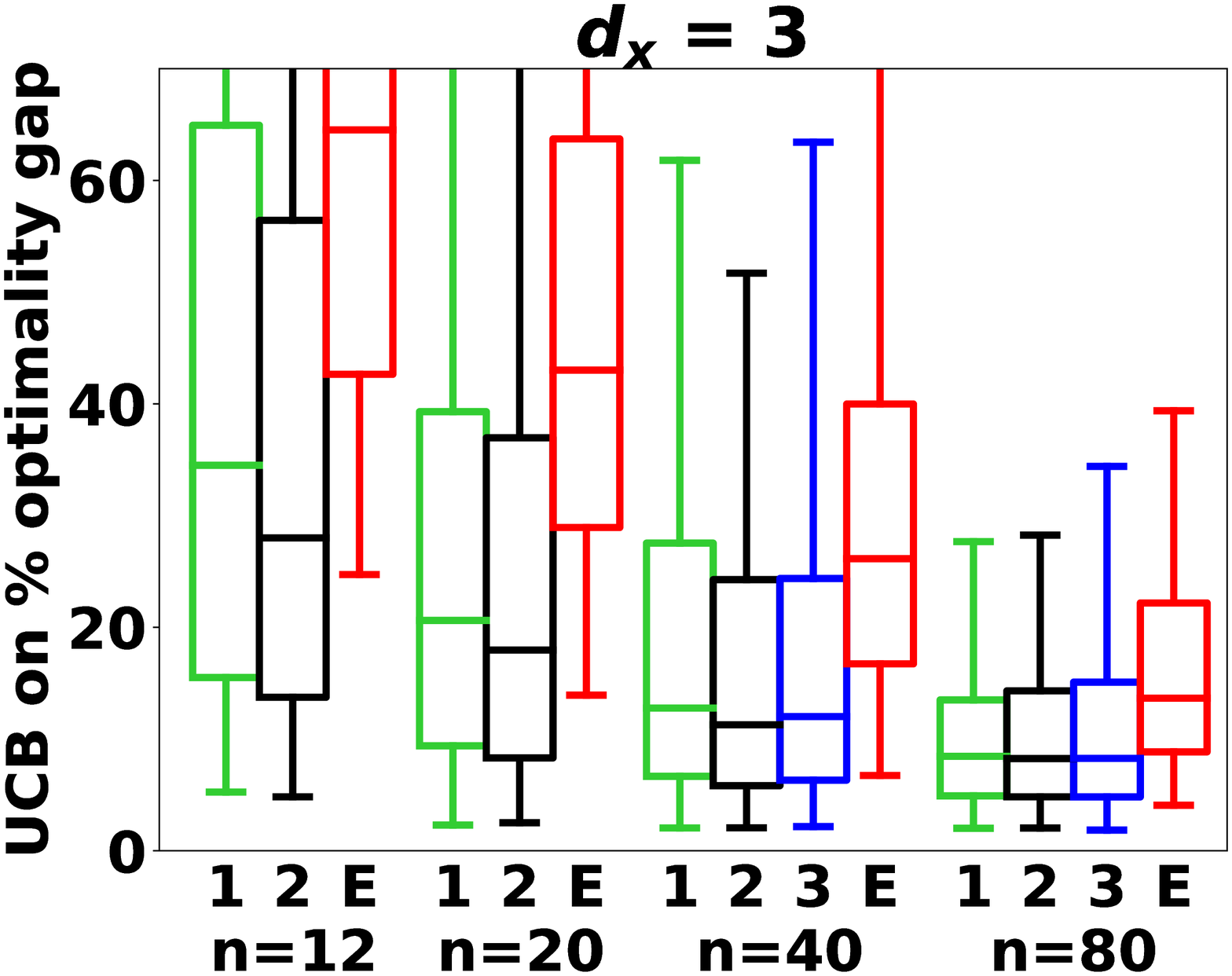}
    \end{subfigure}%
    ~ 
    \begin{subfigure}[t]{0.33\textwidth}
        \centering
        \includegraphics[width=\textwidth]{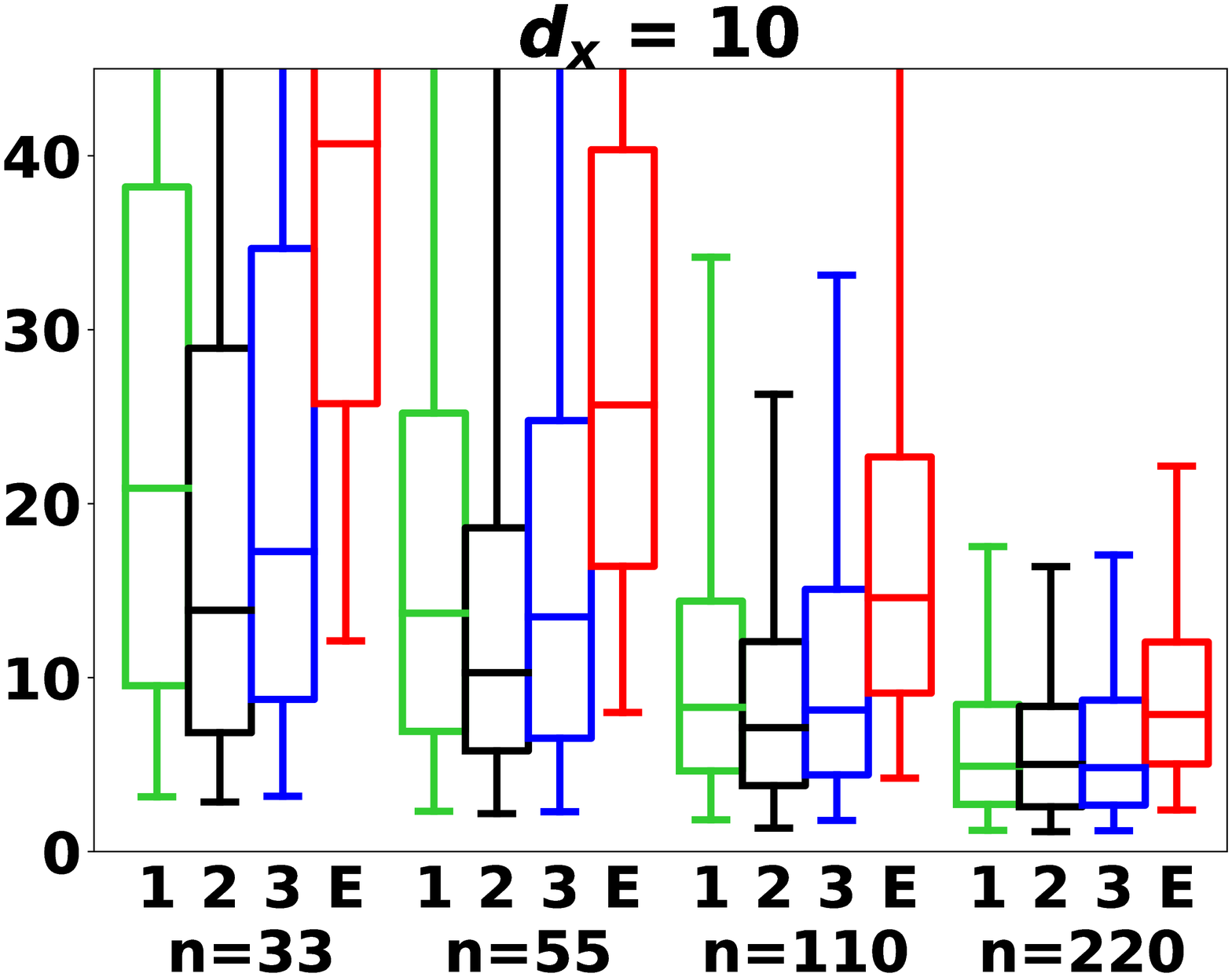}
    \end{subfigure}%
    ~ 
    \begin{subfigure}[t]{0.33\textwidth}
        \centering
        \includegraphics[width=\textwidth]{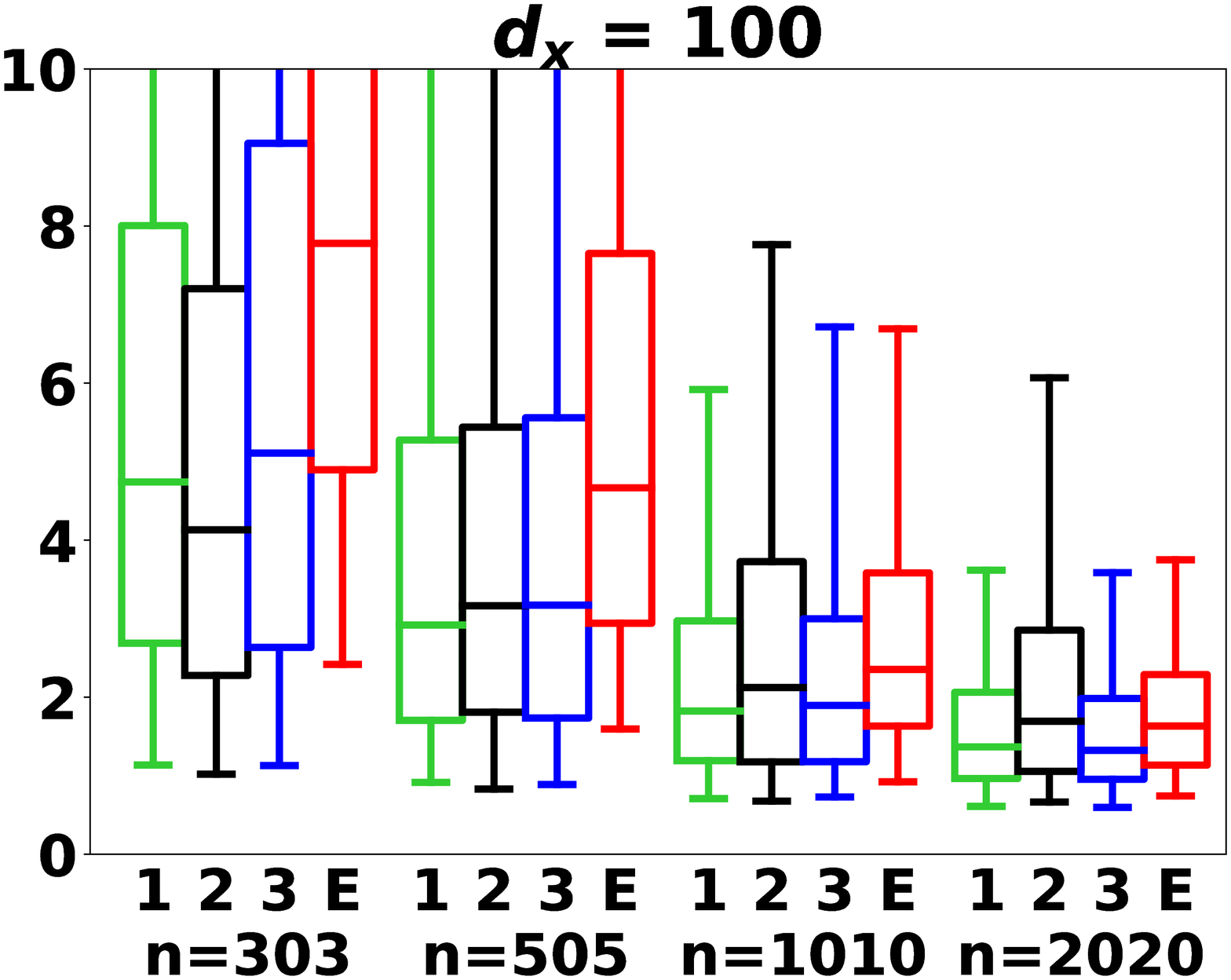}
    \end{subfigure}\\
    \begin{subfigure}[t]{0.33\textwidth}
        \centering
        \includegraphics[width=\textwidth]{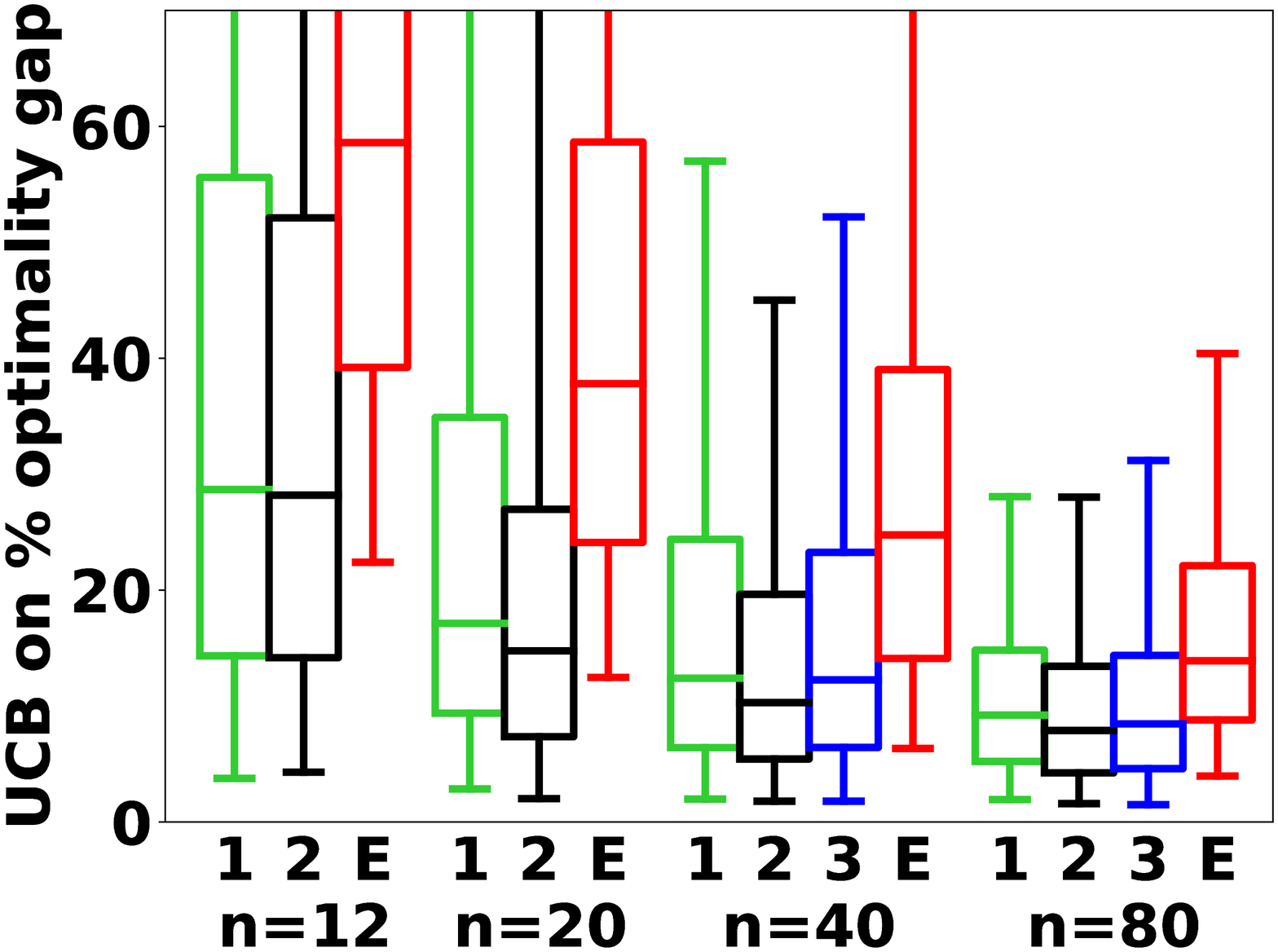}
    \end{subfigure}%
    ~ 
    \begin{subfigure}[t]{0.33\textwidth}
        \centering
        \includegraphics[width=\textwidth]{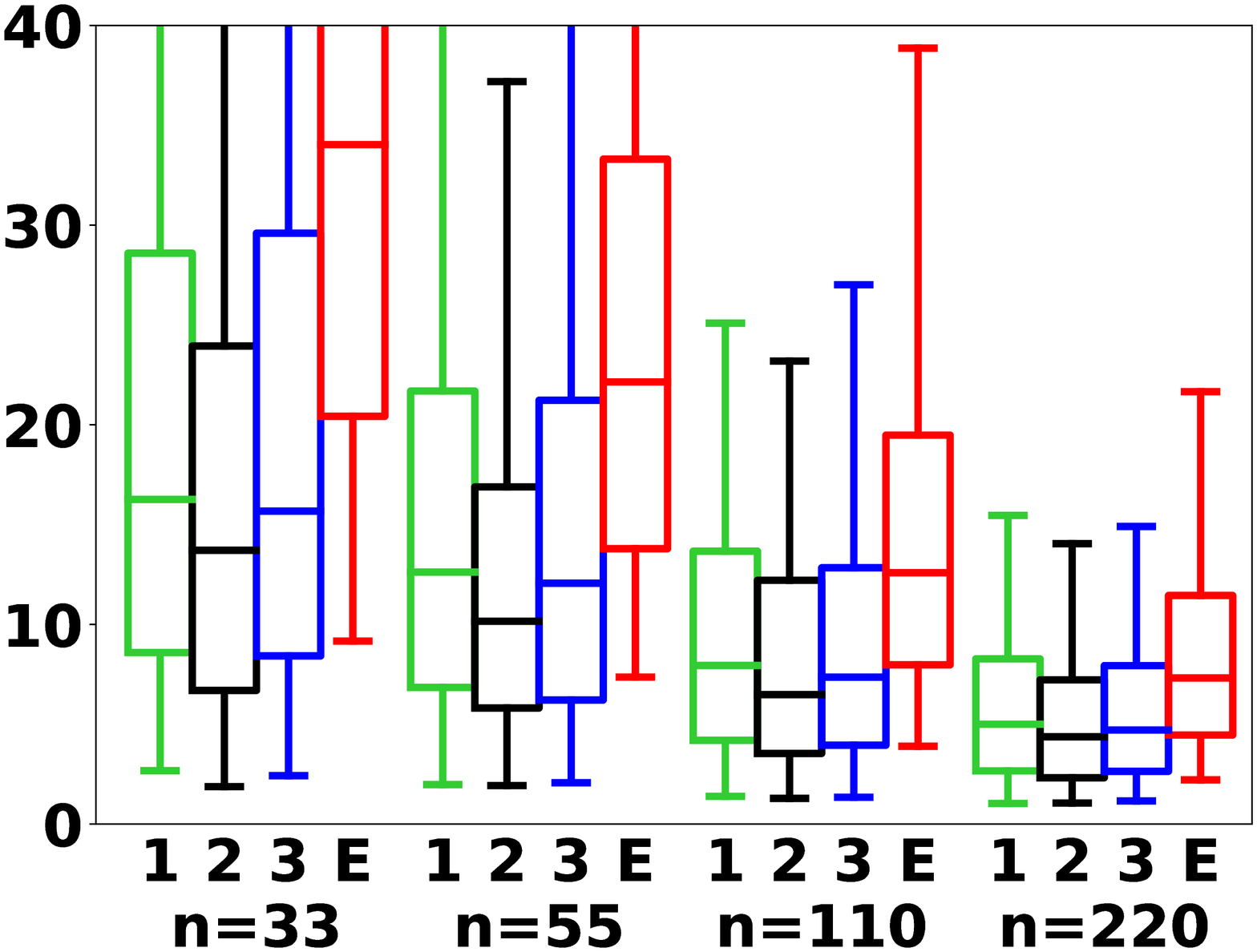}
    \end{subfigure}%
    ~ 
    \begin{subfigure}[t]{0.33\textwidth}
        \centering
        \includegraphics[width=\textwidth]{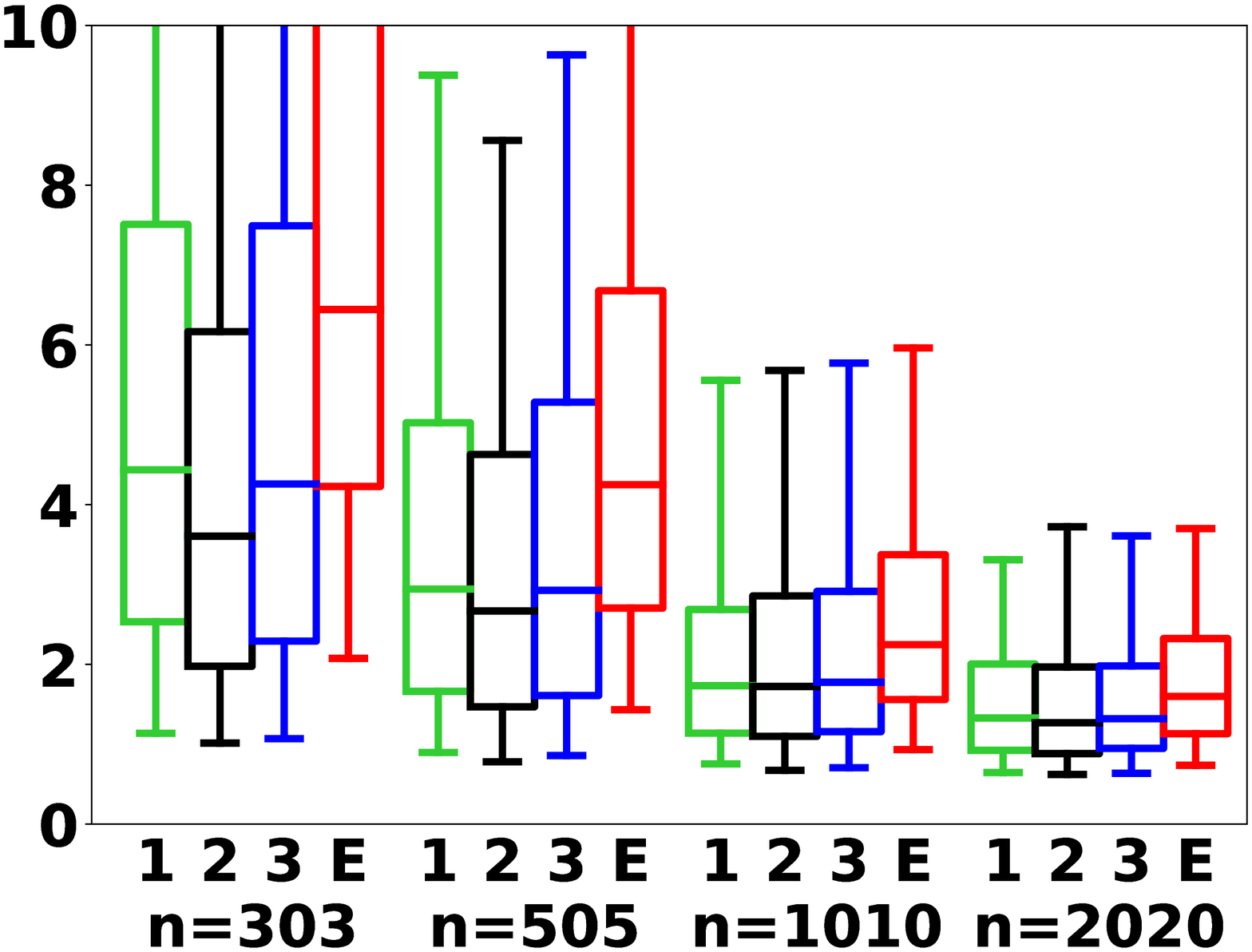}
    \end{subfigure}\\
    \begin{subfigure}[t]{0.33\textwidth}
        \centering
        \includegraphics[width=\textwidth]{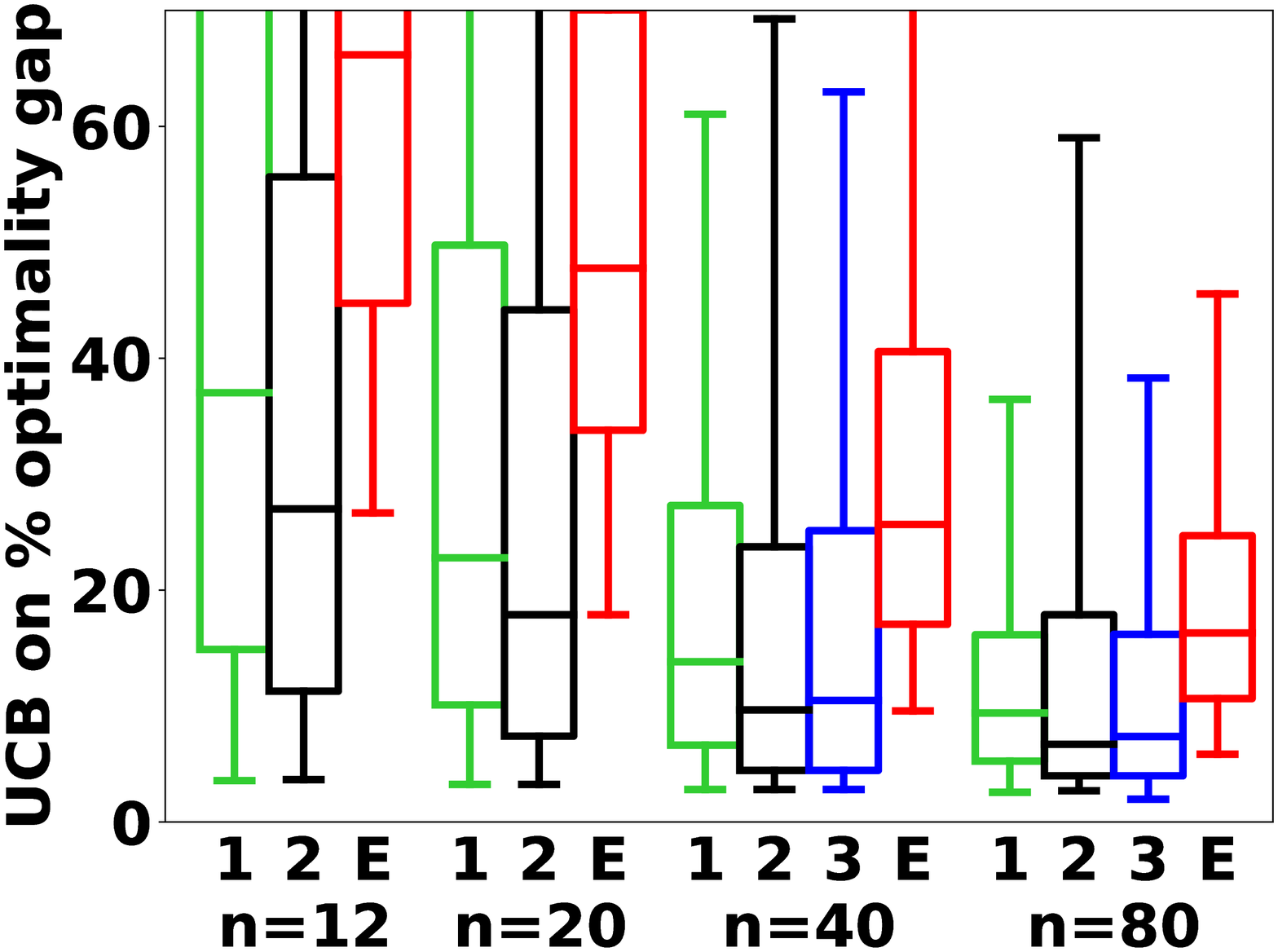}
    \end{subfigure}%
    ~ 
    \begin{subfigure}[t]{0.33\textwidth}
        \centering
        \includegraphics[width=\textwidth]{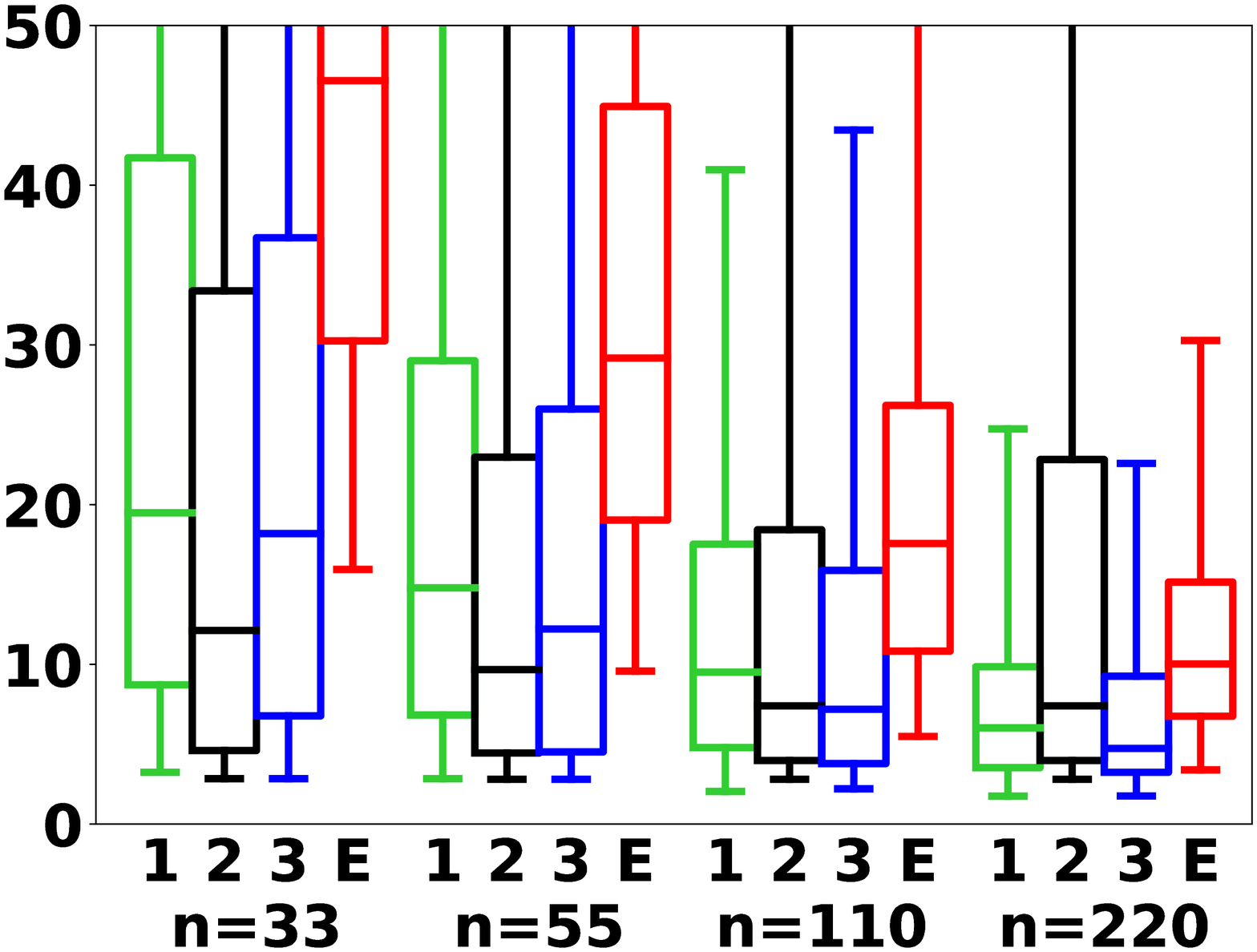}
    \end{subfigure}%
    ~ 
    \begin{subfigure}[t]{0.33\textwidth}
        \centering
        \includegraphics[width=\textwidth]{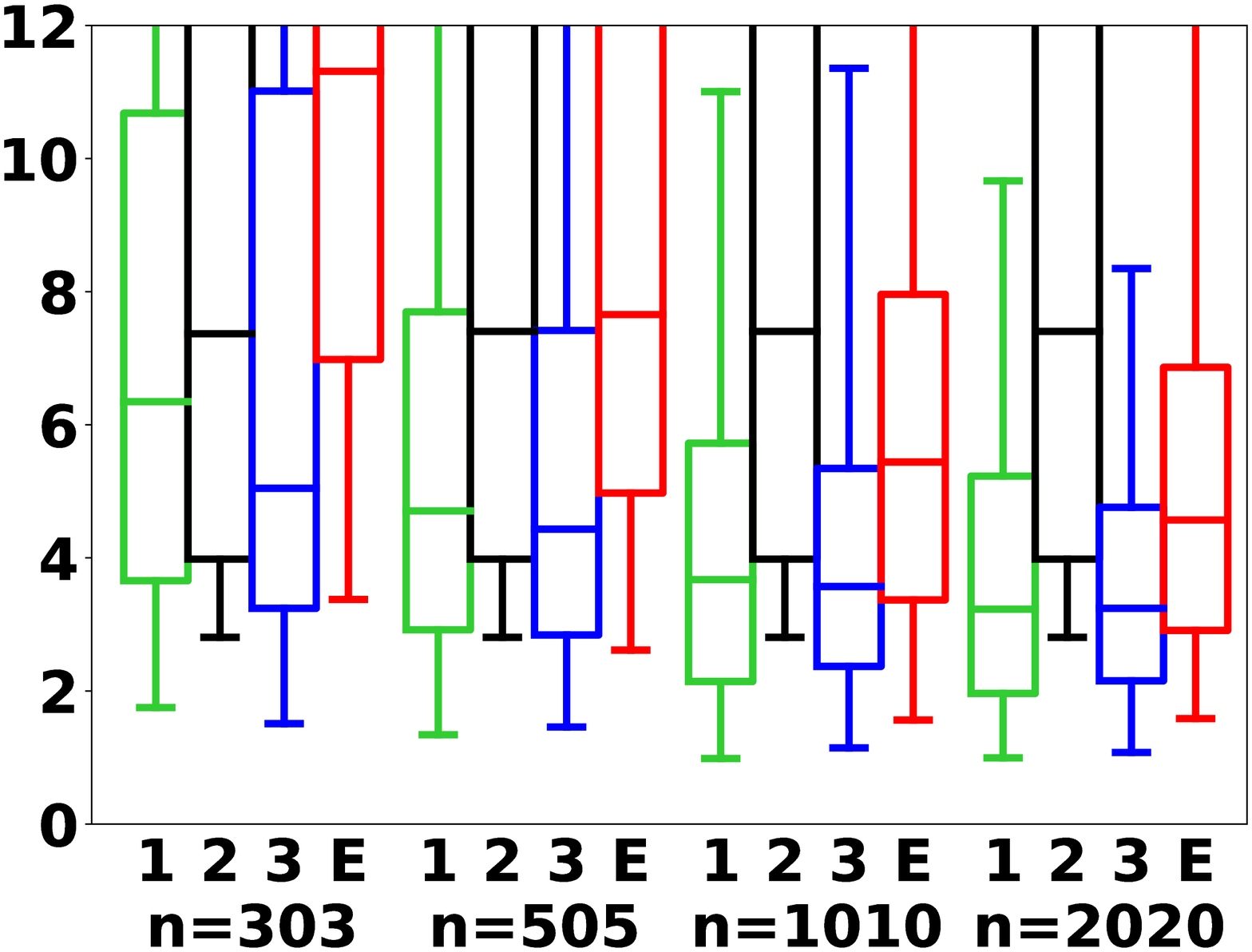}
    \end{subfigure}
    \caption{{\textbf{(Wasserstein-DRO with the Lasso)} Comparison of the \tE+Lasso approach (\texttt{E}) with the tuning of the \tW+Lasso radius using Algorithms~\ref{alg:naivesaaradius} (\tone),~\ref{alg:ersaasameradius} (\ttwo), and~\ref{alg:ersaadiffradius} (\tthree). Top row: $\theta = 1$. Middle row: $\theta = 0.5$. Bottom row: $\theta = 2$. Left column: $d_x = 3$. Middle column: $d_x = 10$. Right column: $d_x = 100$.}}
    \label{fig:comp_lasso_wass_dep}
\end{figure}

\noindent
\textbf{Impact of the prediction step.}
{We now highlight the modularity of our ER-DRO framework by exploring the potential benefits of regularization-based methods for estimating $f^*$.}
Figure~\ref{fig:comp_lasso_wass_dep} compares the performance of the \tE+Lasso approach with the \tW+Lasso approach, {whereas Figure~\ref{fig:comp_ridge_wass} in Appendix~\ref{app:computres} compares the performance of the \tE+Ridge approach with the \tW+Ridge approach.
The radius~$\zeta_n$ of the ambiguity set for these \tW\ formulations is determined using Algorithms~\ref{alg:naivesaaradius},~\ref{alg:ersaasameradius}, and~\ref{alg:ersaadiffradius}.}
We consider {$d_x \in \{3, 10, 100\}$}, vary the model degree $\theta$, and vary the sample size among {$n \in \{3(d_x + 1),5(d_x + 1),10(d_x + 1),20(d_x + 1)\}$} in these experiments\footnote{Once again, we only report results for Algorithm~\ref{alg:ersaadiffradius} when $n \geq 30$.}.
We consider smaller sample sizes because the Lasso {and Ridge regression} are most effective in this regime.
These experiments also illustrate the modularity of our residuals-based formulations.
{The \tW+Lasso and \tW+Ridge formulations outperform the \tE+Lasso and \tE+Ridge formulations, respectively, when the sample size $n$ is small relative to the covariate dimension $d_x$ (except when Algorithm~\ref{alg:ersaasameradius} is used for $\theta = 2$ and $n$ large)}.
Note that the $y$-axis limits are different {across the subplots in Figures~\ref{fig:comp_lasso_wass_dep} and~\ref{fig:comp_ridge_wass}}.
{Regularization as in the case of Lasso and Ridge regression reduces the variance of the regression estimate~$\hat{f}_n$ at the expense of an increase in its bias. 
Since the ambiguity sets of our residuals-based ER-DRO formulations do not \textit{explicitly} address the uncertainty in the coefficients of the regression estimate~$\hat{f}_n$, trading the variance of the coefficient estimates for some bias can result in ER-DRO estimators with better out-of-sample performance.}
\\

\begin{figure}
    \centering
    \begin{subfigure}[t]{0.33\textwidth}
        \centering
        \includegraphics[width=\textwidth]{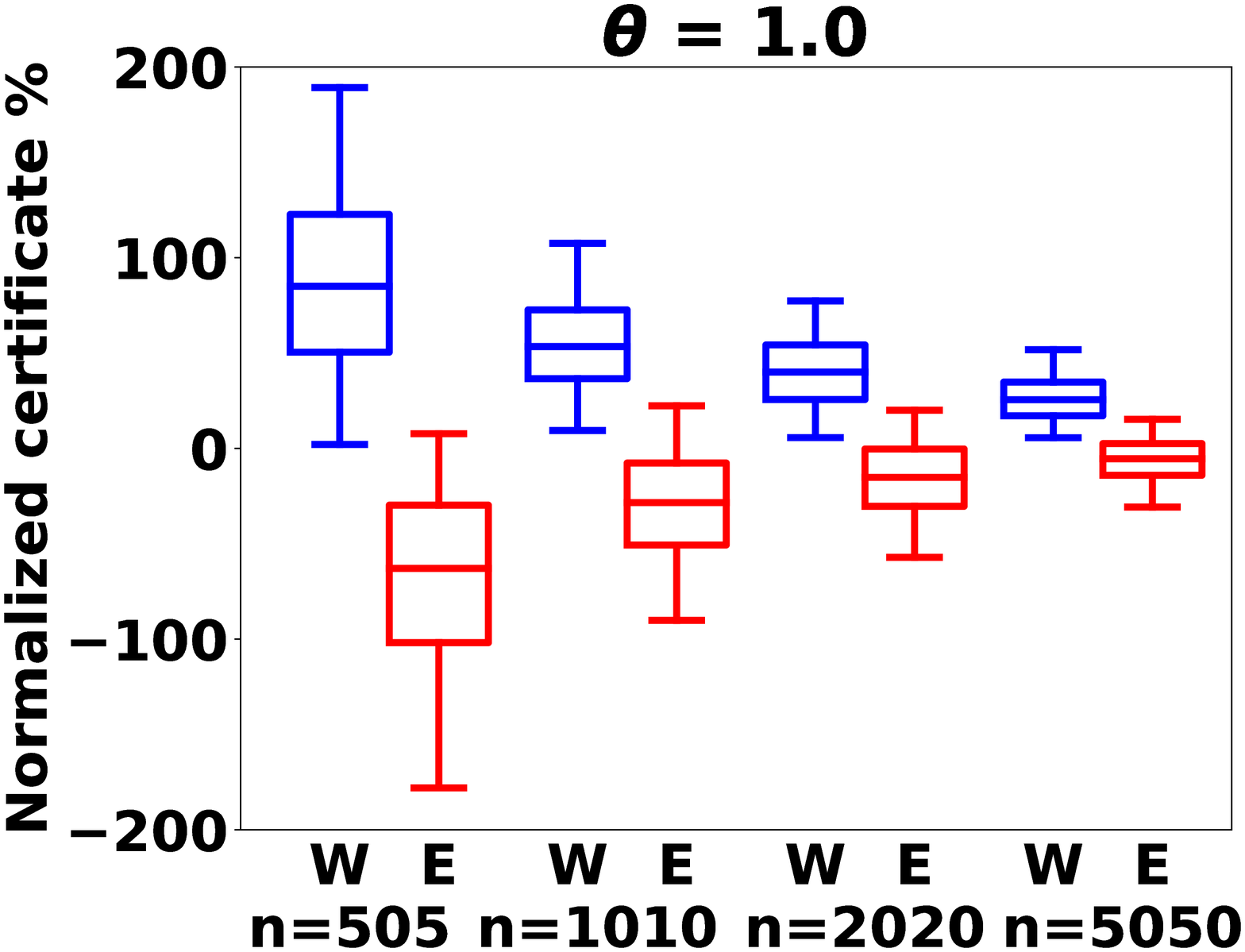}
    \end{subfigure}%
    ~ 
    \begin{subfigure}[t]{0.33\textwidth}
        \centering
        \includegraphics[width=\textwidth]{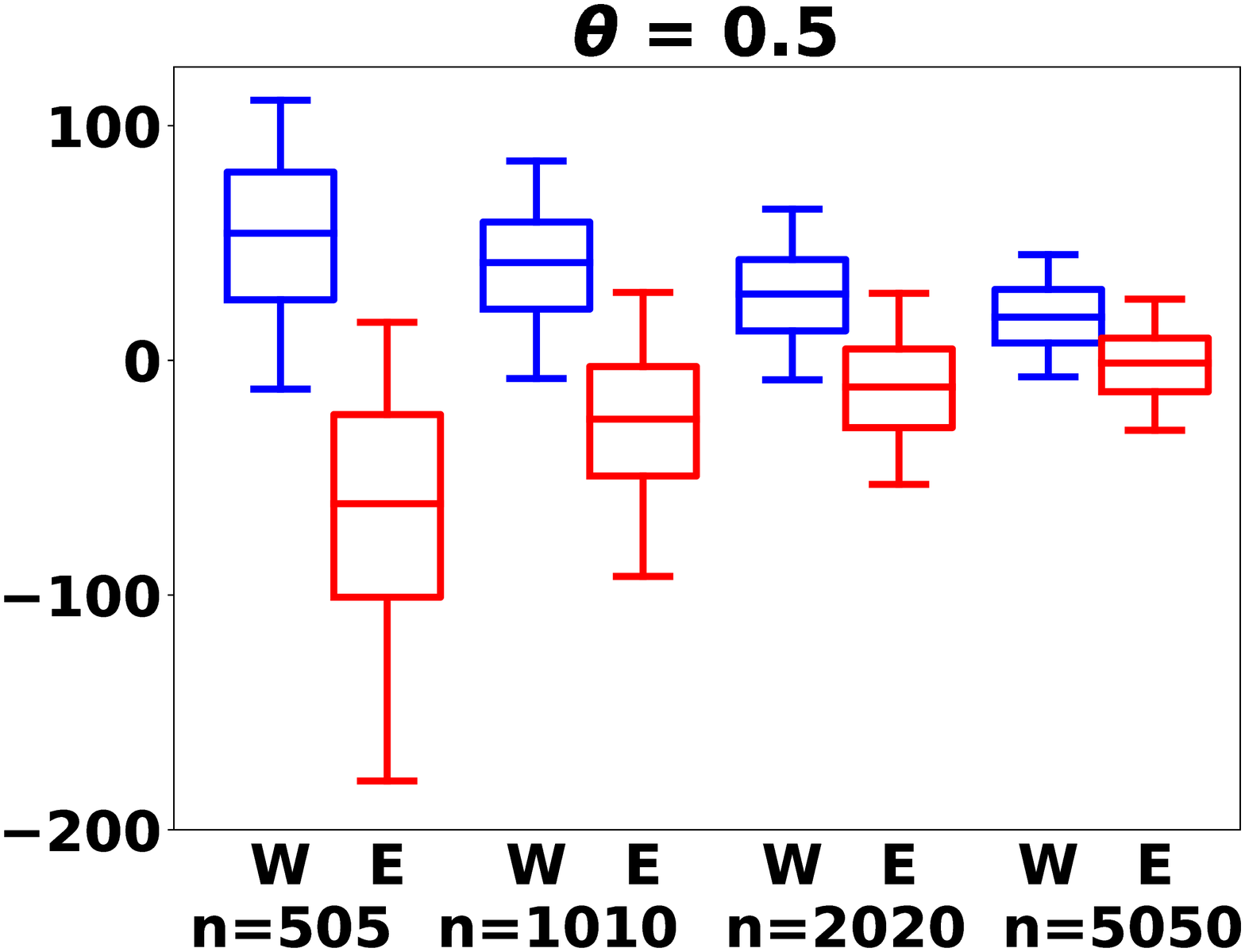}
    \end{subfigure}%
    ~ 
    \begin{subfigure}[t]{0.33\textwidth}
        \centering
        \includegraphics[width=\textwidth]{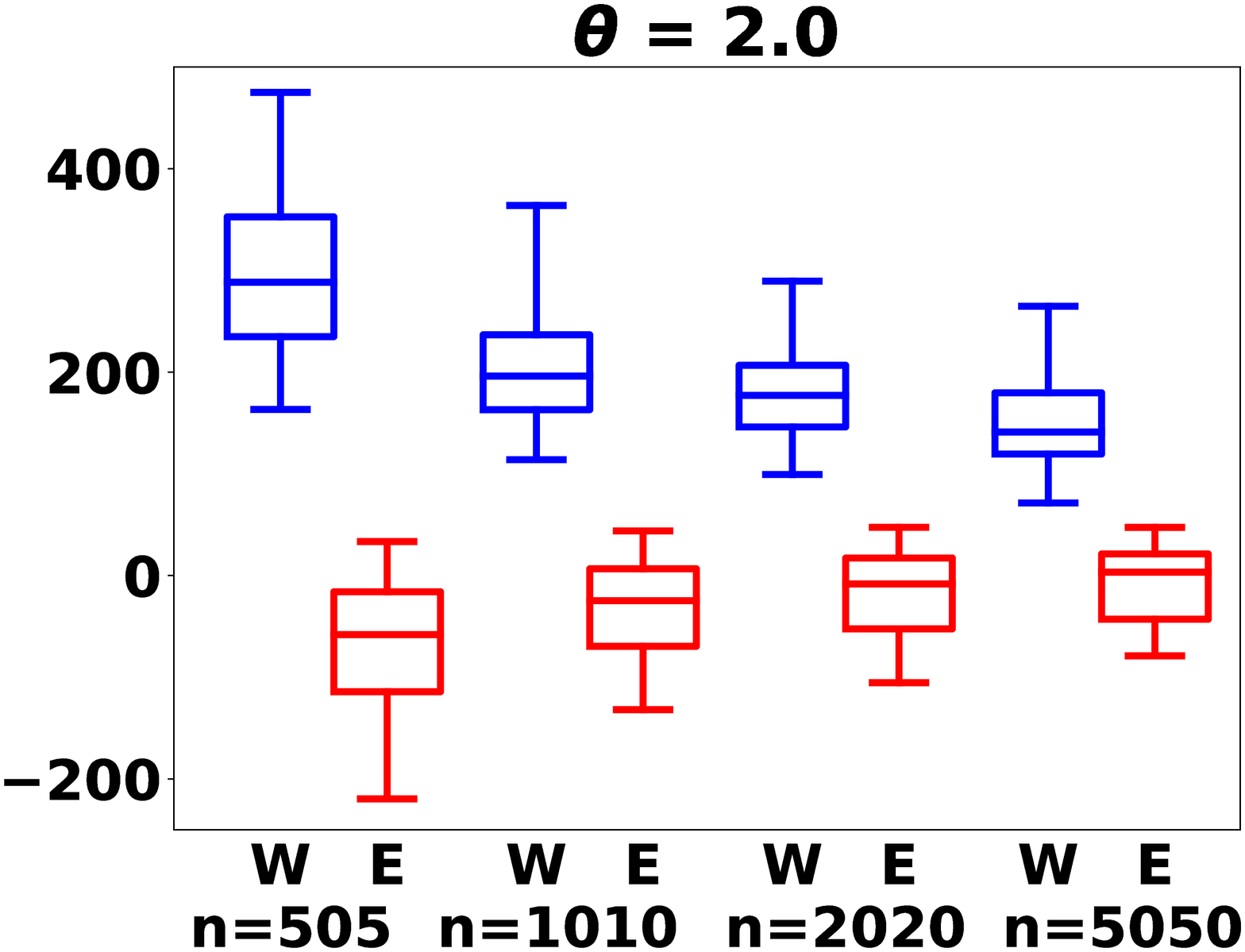}
    \end{subfigure}
    \caption{{\textbf{(Wasserstein-DRO certificate)} Comparison of the \tE+OLS approach (\texttt{E}) with the covariate-independent tuning of the \tW+OLS Wasserstein radius using Algorithm~\ref{alg:ersaasameradius} (\texttt{W}) for $d_x = 100$. Left: $\theta = 1$. Middle: $\theta = 0.5$. Right: $\theta = 2$.}}
    \label{fig:comp_ols_wass_cert}
\end{figure}

\noindent
\textbf{Wasserstein-DRO certificates.}
Figure~\ref{fig:comp_ols_wass_cert} compares the normalized optimal objective value $100(\hv^{ER}_n(x) - v^*(x))$ of the \tE+OLS formulation with the normalized optimal objective value $100(\hv^{DRO}_n(x) - v^*(x))$ of the \tW+OLS formulation when the radius~$\zeta_n$ is specified by Algorithm~\ref{alg:ersaasameradius}.
We consider $d_x = 100$, vary the model degree $\theta$, and vary the {sample size among $n \in \{5(d_x + 1),10(d_x + 1),20(d_x + 1),50(d_x + 1)\}$ in these experiments}.
We omit the results for smaller covariate dimensions for brevity.
{Note that the $y$-axis limits are different across the subplots}.
First, we see that the ER-SAA solutions are optimistically biased and the bias reduces with increasing sample size (cf.\ \cite{bertsimas2018data,esfahani2018data,mak1999monte}).
{Second, we see that ER-DRO solutions tend to err on the side of caution, with the pessimism of the \tW+OLS formulation shrinking to zero for $\theta = 1$ and $\theta = 0.5$ as the sample size increases.
Finally, the pessimistic bias of the \tW+OLS formulation does not rapidly shrink to zero for $\theta = 2$ because the radius~$\zeta_n$ specified using Algorithm~\ref{alg:ersaasameradius} does not shrink to zero for this case due to significant model misspecification (cf.\ Figure~\ref{fig:comp_ols_wass_radii} in Appendix~\ref{app:computres}).}

\section{Conclusion and future work}

We propose a flexible data-driven DRO framework for incorporating covariate information in stochastic optimization when we only have limited concurrent observations of random variables and covariates.
We study formulations that build a Wasserstein ambiguity set or an ambiguity set with only discrete distributions on top of a data-driven SAA formulation.
Our approach seamlessly generalizes existing DRO formulations that do not use covariate information without sacrificing tractability or favorable theoretical guarantees.
We explore {new} data-driven approaches for sizing our ambiguity sets {that do not require samples from the conditional distribution of the random variables}.
Numerical experiments illustrate that our residuals-based Wasserstein and sample robust optimization DRO formulations can {significantly} outperform the ER-SAA formulation in the limited data regime.
We conclude that the ER-DRO and ER-SAA approaches are complementary. With limited data, the ER-DRO approach can yield better solutions. On the other hand, the value of ER-DRO over ER-SAA diminishes if there is ample data available, and
the ER-SAA formulation remains tractable under milder assumptions on the true problem~\eqref{eqn:speq} compared to the Wasserstein and sample robust optimization-based ER-DRO formulations. In particular, these ER-DRO formulations generally result in NP-hard formulations for two-stage stochastic programs and hence may require approximations~\cite{kuhn2019wasserstein,rahimian2019distributionally}.

{Deriving sharper finite sample guarantees for Wasserstein ER-DRO is an interesting avenue for future work.}
Extensions of the ER-SAA and ER-DRO formulations to multi-stage stochastic programming (cf.\ \cite{bertsimas2019dynamic}), for the case when decisions affect the realizations of the random variables (cf.\ \cite{bertsimas2014predictive}), and for problems with stochastic constraints (cf.\ \cite{homem2014monte}) also merit further investigation.

\section*{Acknowledgments}
This material is based upon work supported by the U.S.\ Department of Energy, Office of Science, Office of Advanced Scientific Computing Research (ASCR) under Contract No.\ DE-AC02-06CH11357.
{R.K.\ also gratefully acknowledges the support of the U.S.\ Department of Energy through the LANL/LDRD Program and the Center for Nonlinear Studies.}
This research was performed using the computing resources of the UW-Madison CHTC in the CS Department. The CHTC is supported by UW-Madison, the Advanced Computing Initiative, the Wisconsin Alumni Research Foundation, the Wisconsin Institutes for Discovery, and the National Science Foundation.
{The authors thank the three anonymous reviewers for suggestions that helped improve the readability of this paper.}
R.K.\ {also} thanks Nam Ho-Nguyen for helpful discussions.

{
\footnotesize
\section*{References}
\begingroup
\renewcommand{\section}[2]{}%
\bibliographystyle{abbrvnat}
\bibliography{main}
\endgroup
}


\appendix

\section{Omitted Proofs}
\label{app:proofs}

\subsection{Proof of Proposition~\ref{prop:subgaussian}}
\label{app:subgaussian}

{We begin with the following useful results.}

\begin{lemma}
\label{lem:ineq}
{Let $a_1, a_2, \dots, a_{d}$ be positive constants with $a_j \geq 1$, $\forall j \in [d]$.
Then, we have}
\[
{
\Bigl( \sum_{i=1}^{d} a^2_i \Bigr)^{\theta/2} \leq \sum_{i=1}^{d} a^{1+ \frac{\theta}{2}}_i, \quad \forall \theta \in [1,2].
}
\]
\end{lemma}
\begin{proof}
{Let $F(\theta) := \bigl(\sum_{i=1}^{d} a^{1+ \frac{\theta}{2}}_i\bigr)^{2/\theta}$ and $G(\theta) := \log(F(\theta)) =  \frac{2}{\theta} \log\bigl(\sum_{i=1}^{d} a^{1+ \frac{\theta}{2}}_i \bigr)$.
The stated result holds if $F$ (or equivalently, $G$) is nonincreasing on $\theta \in [1,2]$.
We have}
\[
{
G'(\theta) = \frac{1}{\theta} \biggl[ \sum_{i=1}^{d} w_i \log(a_i) - \frac{2}{\theta} \log\Bigl(\sum_{i=1}^{d} a^{1+ \frac{\theta}{2}}_i\Bigr)  \biggr] = \frac{1}{\theta} \biggl[  \log\biggl(\prod_{i=1}^{d} (a_i)^{w_i}\biggr) - \frac{2}{\theta} \log\Bigl(\sum_{i=1}^{d} a^{1+ \frac{\theta}{2}}_i\Bigr)  \biggr],
}
\]
{where the nonnegative weight $w_i := \bigl(a^{1+ \theta/2}_i\bigr) \bigl(\sum_{j=1}^{d} a^{1+ \theta/2}_j\bigr)^{-1}$.
Note that $w_i \in (0,1)$ and $\sum_{i=1}^{d} w_i = 1$.
For $\theta \in [1,2]$, the above expression implies that $G'(\theta) \leq 0$ whenever $\prod_{i=1}^{d} (a_i)^{w_i} \leq \sum_{i=1}^{d} a^{1+ \theta/2}_i$.
We have}
\[
{
\prod_{i=1}^{d} (a_i)^{w_i} \leq \sum_{i=1}^{d} w_i a_i \leq \sum_{i=1}^{d} a_i \leq \sum_{i=1}^{d} a^{1+ \frac{\theta}{2}}_i,
}
\]
{where the first inequality follows from the weighted AM-GM inequality, the second inequality follows from the fact that $0 < w_i < 1$, $\forall i \in [d]$, and the final inequality follows from the assumption that $a_i \geq 1$, $\forall i \in [d]$.
Consequently, $G'(\theta) \leq 0$, $\forall \theta \in [1,2]$, which implies $F$ is nonincreasing on $\theta \in [1,2]$.}
\end{proof}

\begin{lemma}
\label{lem:subgineq}
{Let $W$ be a sub-Gaussian random variable with variance proxy $\sigma^2_w$.
Then}
\[
{
\expect{\exp\bigl(\abs{W}^{1 + \frac{\theta}{2}}\bigr)} < +\infty, \quad \forall \theta \in (1,2).
}
\]
\end{lemma}
\begin{proof}
{We have
\begin{align}
\label{eqn:ineq1}
\expect{\exp\bigl(\abs{W}^{1 + \frac{\theta}{2}}\bigr)} &= \mathbb{E}\biggl[\sum_{j=0}^{\infty} \frac{\abs{W}^{(1+\frac{\theta}{2})j}}{j!}\biggr]
\end{align}
Lemma~1.4 of~\cite{rigollet2015high} implies that}
\[
{
\mathbb{E}\biggl[\frac{\abs{W}^{(1+\frac{\theta}{2})j}}{j!}\biggr] \leq \frac{\bigl( \sigma^2_w e^{\frac{2}{e}} j \bigr)^{(1+\frac{\theta}{2})\frac{j}{2}}}{j!}, \quad \forall j \geq 2,
}
\]
{where $e := \exp(1)$.
Therefore, inequality~\eqref{eqn:ineq1}, the Fubini-Tonelli theorem, and the ratio test imply the stated result whenever $\lim_{j \to \infty} t_{j+1}/t_j < 1$, where 
$t_j := \bigl( \sigma^2_w e^{\frac{2}{e}} j \bigr)^{(1+\frac{\theta}{2})\frac{j}{2}}(j!)^{-1}$.
Let $C := \sigma^2_w e^{\frac{2}{e}}$. We have
\begin{align*}
\lim_{j \to \infty} \frac{t_{j+1}}{t_j} &= \lim_{j \to \infty} \frac{\bigl( C (j+1) \bigr)^{(1+\frac{\theta}{2})\frac{j+1}{2}}}{\bigl( C j \bigr)^{(1+\frac{\theta}{2})\frac{j}{2}}} \frac{j!}{(j+1)!} \\
&= \lim_{j \to \infty} \frac{(C(j+1))^{0.5(1+\frac{\theta}{2})}}{j+1} \biggl(\frac{j+1}{j}\biggr)^{(1+\frac{\theta}{2})\frac{j}{2}} \\
&= \lim_{j \to \infty} O(1) (j+1)^{0.5(\frac{\theta}{2}-1)} \: \lim_{j \to \infty} \biggl(\frac{j+1}{j}\biggr)^{(1+\frac{\theta}{2})\frac{j}{2}} \\
&= O(1)\lim_{j \to \infty} (j+1)^{0.5(\frac{\theta}{2}-1)} \: \lim_{j \to \infty} \biggl(1 + \frac{(1+\frac{\theta}{2})\frac{1}{2}}{(1+\frac{\theta}{2})\frac{j}{2}}\biggr)^{(1+\frac{\theta}{2})\frac{j}{2}} \\
&= 0 \times O(1) = 0,
\end{align*}
where we use the fact $\theta \in (1,2)$ in the last step.}
\end{proof}

{We are now ready to prove Proposition~\ref{prop:subgaussian}.
To establish $\expect{\exp(\norm{\varepsilon}^a)} < +\infty$, it suffices to show that $\expect{\exp(\norm{\varepsilon}^a) \1_{\{\norm{\varepsilon}_{\infty} \geq 1\}}} < +\infty$, where $\1_{\{\norm{\varepsilon}_{\infty} \geq 1\}} = 1$ if $\norm{\varepsilon}_{\infty} \geq 1$ and $0$ otherwise.
Lemma~\ref{lem:ineq} implies that
\[
\expect{\exp(\norm{\varepsilon}^a)} \leq \mathbb{E}\biggl[\exp\Bigl(\sum_{i=1}^{d_y} \abs{\varepsilon_i}^{1 + \frac{a}{2}}\Bigr)\biggr] = \mathbb{E}\biggl[\prod_{i=1}^{d_y}\exp\Bigl( \abs{\varepsilon_i}^{1 + \frac{a}{2}}\Bigr)\biggr].
\]
Independence of the components of $\varepsilon$ further implies
\[
\expect{\exp(\norm{\varepsilon}^a)} \leq \prod_{i=1}^{d_y} \mathbb{E}\biggl[\exp\Bigl( \abs{\varepsilon_i}^{1 + \frac{a}{2}}\Bigr)\biggr].
\]
Consequently, it suffices to show that $\mathbb{E}\Bigl[\exp\bigl( \abs{\varepsilon_i}^{1 + \frac{a}{2}}\bigr) \1_{\{\abs{\varepsilon_i} \geq 1\}}\Bigr] < +\infty$ for each $i \in [d_y]$, which follows from Lemma~\ref{lem:subgineq}.}
\qed

\subsection{Proof of Lemma~\ref{lem:wassintres}}
\label{app:wassintres}

{We require the following result (cf.\ \cite[Lemma~3.7]{esfahani2018data}).}

\begin{lemma}
\label{lem:convwassdist}
{Suppose Assumptions~\ref{ass:lighttailerr},~\ref{ass:reglargedevwass}, and~\ref{ass:wassrisklevelseq} hold and the samples $\{\varepsilon^{i}\}_{i=1}^{n}$ are i.i.d.
Let $\{Q_n(x)\}$ be a sequence of distributions with $Q_n(x) \in \hP_n(x;\zeta_n(\alpha_n,x))$.
Then
\[
\mathbb{P}\Bigl\{ d_{W,p}(P_{Y \mid X=x},Q_n(x)) \leq 2\zeta_n(\alpha_n,x) \Bigr\} \geq 1 - \alpha_n, \quad \text{for a.e. } x \in \X.
\]
Consequently, we a.s.\ have for $n$ large enough:
\[
d_{W,p}(P_{Y \mid X=x},Q_n(x))  \leq 2\zeta_n(\alpha_n,x), \quad \text{for a.e. } x \in \X.
\]
Furthermore, $\{Q_n(x)\}$ converges a.s.\ converges to $P_{Y \mid X = x}$ under the Wasserstein metric}
\[
{
\mathbb{P}\Bigl\{ \lim_{n \to \infty} d_{W,p}(P_{Y \mid X=x},Q_n(x)) = 0 \Bigr\} = 1.
}
\]
\end{lemma}
\begin{proof}
{Mirrors the proof of~\cite[Lemma~3.7]{esfahani2018data} on account of Theorem~\ref{thm:wassfinitesampcert}.}
\end{proof}

{We now prove Lemma~\ref{lem:wassintres}.
From Theorem~\ref{thm:wassfinitesampcert}, we have for a.e.\ $x \in \X$ that
\[
\mathbb{P}\big\{v^*(x) \leq g(\hz^{DRO}_n(x);x) \leq \hv^{DRO}_n(x)\big\} \geq 1 - \alpha_n, \quad \forall n \in \mathbb{N}.
\]
Since $\sum_n \alpha_n < +\infty$, the Borel-Cantelli lemma implies a.s.\ that for all $n$ large enough
\[
v^*(x) \leq g(\hz^{DRO}_n(x);x) \leq \hv^{DRO}_n(x), \quad \text{for a.e. } x \in \X.
\]
{From Lemma~\ref{lem:convwassdist},}
we a.s.\ have for any distribution $Q_n(x) \in \hP_n(x;\zeta_n(\alpha_n,x))$ and $n$ large enough that
$d_{W,p}(P_{Y \mid X=x},Q_n(x)) \leq 2\zeta_n(\alpha_n,x)$ for a.e.\ $x \in \X$.}

{Suppose Assumption~\ref{ass:lipschitzy} holds.
Using the fact that
$\hP_n(x;\zeta_n(\alpha_n,x)) \subseteq \bar{\mathcal{P}}_{1,n}(x;\zeta_n(\alpha_n,x))$ for all orders $p \in [1,+\infty)$,
we a.s.\ have for $n$ large enough and for a.e.\ $x \in \X$ that
\begin{align*}
&\hv^{DRO}_n(x) \leq \uset{Q \in \bar{\mathcal{P}}_{1,n}(x;\zeta_n(\alpha_n,x))}{\sup} \expectation{Y \sim Q}{c(z^*(x),Y)} \leq g(z^*(x);x) + 2L_1(z^*(x))\zeta_n(\alpha_n,x),
\end{align*}
where the second inequality follows from Assumption~\ref{ass:lipschitzy} and the Kantorovich-Rubinstein theorem (cf.\ \citep[Theorem~5]{kuhn2019wasserstein}).}

{Suppose instead that Assumption~\ref{ass:lipschitzy2} holds and $p \in [2,+\infty)$.
Since
$\hP_n(x;\zeta_n(\alpha_n,x)) \subseteq \bar{\mathcal{P}}_{2,n}(x;\zeta_n(\alpha_n,x))$ for all $p \in [2,+\infty)$,
we a.s.\ have for $n$ large enough and a.e.\ $x \in \X$ that}
{\begin{align*}
\hv^{DRO}_n(x) &\leq \uset{Q \in \bar{\mathcal{P}}_{2,n}(x;\zeta_n(\alpha_n,x))}{\sup} \expectation{Y \sim Q}{c(z^*(x),Y)} \\
&\leq g(z^*(x);x) + 2\bigl(\expect{\norm{\nabla c(z^*(x),Y)}^2}\bigr)^{1/2} \zeta_n(\alpha_n,x) + 4L_2(z^*(x))\zeta^2_n(\alpha_n,x),
\end{align*}}
{where the latter inequality follows from Assumption~\ref{ass:lipschitzy2} and~\cite[Lemma~2]{gao2020finite} (see also~\cite{gao2017wasserstein}).}
\qed

\subsection{Proof of Theorem~\ref{thm:wassaympconsist}}
\label{app:wassaympconsist}

From Theorem~\ref{thm:wassfinitesampcert}, we have 
\[
\mathbb{P}\big\{d_{W,p}( \hat{P}^{ER}_n(x), P_{Y \mid X=x} ) > \zeta_n(\alpha_n,x)\big\} \leq \alpha_n, \quad \text{for a.e. } x \in \X.
\]
{From Lemma~\ref{lem:convwassdist},}
we a.s.\ have $\lim_{n \to \infty} d_{W,p}(P_{Y \mid X=x},Q_n(x)) = 0$ for a.e.\ $x \in \X$ for any $Q_n(x) \in \hP_n(x;\zeta_n(\alpha_n,x))$.
Theorem~6.9 of~\cite{villani2008optimal} then a.s.\ implies that $Q_n(x)$ converges weakly to $P_{Y \mid X = x}$ in the space of distributions with finite $p$th moments for a.e.\ $x \in \X$.

Lemma~\ref{lem:wassintres} implies a.s.\ that for all $n$ large enough
\begin{align}
\label{eqn:borelcantelli}
&v^*(x) \leq g(\hz^{DRO}_n(x);x) \leq \hv^{DRO}_n(x), \quad \text{for a.e. } x \in \X.
\end{align}
Therefore, to prove $\lim_{n \to \infty} \hv^{DRO}_n(x) = v^*(x) = \lim_{n \to \infty} g(\hz^{DRO}_n(x);x)$ in probability (or a.s.) for a.e.\ $x \in \X$, it suffices to show that $\limsup_{n \to \infty} \hv^{DRO}_n(x) \leq v^*(x)$ a.s.\ for a.e.\ $x \in \X$.

Fix~$\eta > 0$.
For a.e.\ $x \in \X$, let $z^*(x) \in S^*(x)$ be an optimal solution to the true problem~\eqref{eqn:speq}, and $Q^{*}_{n}(x) \in \hP_n(x;\zeta_n(\alpha_n,x))$ be such that 
\[
\uset{Q \in \hP_n(x;\zeta_n(\alpha_n,x))}{\sup} \expectation{Y \sim Q}{c(z^*(x),Y)} \leq  \expectation{Y \sim Q^*_n(x)}{c(z^*(x),Y)} + \eta.
\]
We suppress the dependence of $Q^{*}_{n}(x)$ on $\eta$ for simplicity.
We a.s.\ have for a.e.\ $x \in \X$
\begin{align*}
\limsup_{n \to \infty} \hv^{DRO}_n(x) &\leq \limsup_{n \to \infty} \uset{Q \in \hP_n(x;\zeta_n(\alpha_n,x))}{\sup} \expectation{Y \sim Q}{c(z^*(x),Y)} \\
&\leq \limsup_{n \to \infty} \expectation{Y \sim Q^*_n(x)}{c(z^*(x),Y)} + \eta \\
&= g(z^*(x);x) + \eta = v^*(x) + \eta.
\end{align*}
The first equality above follows from the fact that $Q^*_n(x)$ converges weakly to $P_{Y \mid X = x}$ (as noted above) and by Definition~6.8 of~\cite{villani2008optimal} (which holds by virtue of Assumption~\ref{ass:growthrate}).
Since $\eta > 0$ was arbitrary, we conclude that $\limsup_{n \to \infty} \hv^{DRO}_n(x) \leq v^*(x)$ a.s.\ for a.e.\ $x \in \X$. 

Finally, we show that any accumulation point of $\{\hz^{DRO}_n(x)\}$ is almost surely an element of $S^*(x)$ for a.e.\ $x \in \X$, and argue that this implies $\text{dist}(\hz^{DRO}_n(x),S^*(x)) \xrightarrow{a.s.} 0$ for a.e.\ $x \in \X$.
From~\eqref{eqn:borelcantelli} and the above conclusion, we a.s.\ have 
\[
\liminf_{n \to \infty} g(\hz^{DRO}_n(x);x) \leq \lim_{n \to \infty} \hv^{DRO}_n(x) = v^*(x), \quad \text{for a.e. } x \in \X.
\]
Let $\bar{z}(x)$ be an accumulation point of $\hz^{DRO}_n(x)$ for a.e.\ $x \in \X$.
Assume by moving to a subsequence if necessary that $\lim_{n \to \infty} \hz^{DRO}_n(x) = \bar{z}(x)$.
We a.s.\ have for a.e.\ $x \in \X$
\[
v^*(x) \leq g(\bar{z}(x);x) \leq \expect{\liminf_{n \to \infty} c(\hz^{DRO}_n(x),f^*(x)+\varepsilon)} \leq \liminf_{n \to \infty} g(\hz^{DRO}_n(x);x) \leq v^*(x),
\]
where the second inequality follows from the lower semicontinuity of~$c(\cdot,Y)$ on~$\Z$ for each $Y \in \Y$ and the third inequality follows from Fatou's lemma by virtue of Assumption~\ref{ass:growthrate}.
Consequently, we a.s.\ have that $\bar{z}(x) \in S^*(x)$.

Suppose by contradiction that $\text{dist}(\hz^{DRO}_n(x),S^*(x))$ does not a.s.\ converge to zero for a.e.\ $x \in \X$.
Then, there exists $\bar{\X} \subseteq \X$ with $P_X(\bar{\X}) > 0$ such that for each $x \in \bar{\X}$, $\text{dist}(\hz^{DRO}_n(x),S^*(x))$ does not a.s.\ converge to zero.
Since $\Z$ is compact, any sequence of estimators $\{\hz^{DRO}_n(x)\}$ has a convergent subsequence for each $x \in \bar{\X}$.
Therefore, whenever $\text{dist}(\hz^{DRO}_n(x),S^*(x))$ does not converge to zero for some $x \in \bar{\X}$ and a realization of the data~$\D_n$, there exists an accumulation point of the sequence $\{\hz^{DRO}_n(x)\}$ that is not a solution to problem~\eqref{eqn:speq}.
This contradicts the fact that every accumulation point of $\{\hz^{DRO}_n(x)\}$ is almost surely a solution to problem~\eqref{eqn:speq} for a.e.\ $x \in \X$. \qed

\subsection{Proof of Theorem~\ref{thm:wassconvrate}}
\label{app:wassconvrate}

Lemma~\ref{lem:wassintres} implies that inequality~\eqref{eqn:borelcantelli} a.s.\ holds for all $n$ large enough.
{From Lemma~\ref{lem:convwassdist},}
we a.s.\ have for any distribution $Q_n(x) \in \hP_n(x;\zeta_n(\alpha_n,x))$ and sample size $n$ large enough that
$d_{W,p}(P_{Y \mid X=x},Q_n(x)) \leq 2\zeta_n(\alpha_n,x)$ for a.e.\ $x \in \X$.

If Assumption~\ref{ass:lipschitzy} holds, then the desired result follows from inequality~\eqref{eqn:borelcantelli} and part A of Lemma~\ref{lem:wassintres}.
On the other hand, if Assumption~\ref{ass:lipschitzy} holds and $p \geq 2$, then the desired result follows from inequality~\eqref{eqn:borelcantelli} and part B of Lemma~\ref{lem:wassintres}.
\qed

\subsection{Proof of Lemma~\ref{lem:wassfinitesampsoln}}
\label{app:wassfinitesampsoln}

{Theorem~\ref{thm:wassfinitesampcert} implies $g(\hz^{DRO}_n(x);x) \leq \hv^{DRO}_n(x)$ with probability at least $1-\alpha$ when~$\zeta_n(\alpha,x)$ is chosen according to equation~\eqref{eqn:wassradius}.
Lemma~\ref{lem:probineq} then yields
\begin{align*}
\prob{g(\hz^{DRO}_n(x);x) > v^*(x) + \kappa} =&\: \prob{g(\hz^{DRO}_n(x);x) - \hv^{DRO}_n(x) + \hv^{DRO}_n(x) > v^*(x) + \kappa} \\
\leq&\: \alpha + \prob{\hv^{DRO}_n(x) > v^*(x) + \kappa}.
\end{align*}}%
{Suppose Assumption~\ref{ass:lipschitzy} holds.
Following the proof of part A of Lemma~\ref{lem:wassintres} (see Lemma~\ref{lem:convwassdist}), we have for any $z^*(x) \in S^*(x)$
\[
\prob{\hv^{DRO}_n(x) > v^*(x) + 2L_1(z^*(x))\zeta_n(\alpha,x)} \leq \alpha.
\]
Therefore, if we choose $\alpha \in (0,1)$ so that $2L_1(z^*(x))\zeta_n(\alpha,x) \leq \kappa$, we have
\[
\prob{g(\hz^{DRO}_n(x);x) > v^*(x) + \kappa} \leq 2\alpha.
\]
Equation~\eqref{eqn:wassradius} implies that $2L_1(z^*(x))\kappa^{(2)}_{p,n}(\alpha) \leq \kappa/2$ whenever the risk level 
\[
\alpha \geq O(1) \exp\bigl(-O(1) n \bigl(\tfrac{\kappa}{4L_1(z^*(x))}\bigr)^{1/\theta}\bigr)
\]
with $\theta$ equal to $\min\{p/d_y, 1/2\}$ or $p/a$.
Assumption~\ref{ass:reglargedevwass2} implies that we can choose the constant $\kappa^{(1)}_{p,n}(\alpha,x)$ in equation~\eqref{eqn:wassradius} such that for a.e.\ $x \in \X$, $2L_1(z^*(x))\kappa^{(1)}_{p,n}(\alpha,x) \leq \kappa/2$ whenever
\begin{align*}
\alpha \geq 4\max\bigl\{&K_{p,f}\bigl(\tfrac{\kappa}{8L_1(z^*(x))},x\bigr) \exp\bigl(-n\beta_{p,f}\bigl(\tfrac{\kappa}{8L_1(z^*(x))},x\bigr)\bigr), \\
&\quad\quad\quad \bar{K}_{p,f}\bigl(\tfrac{\kappa}{8L_1(z^*(x))}\bigr) \exp\bigl(-n\bar{\beta}_{p,f}\bigl(\tfrac{\kappa}{8L_1(z^*(x))}\bigr)\bigr)\bigr\}.
\end{align*}
The above bounds imply the existence of constants $\tilde{\Gamma}(\kappa,x), \tilde{\gamma}(\kappa,x) > 0$ such that the risk level $\alpha = \tilde{\Gamma}(\kappa,x) \exp(-n\tilde{\gamma}(\kappa,x))$ satisfies $2L_1(z^*(x))\zeta_n(\alpha,x) \leq \kappa$.
Consequently,~\eqref{eqn:largedevoutofsampcost} holds.}

{Next, suppose instead that Assumption~\ref{ass:lipschitzy2} holds.
Following the proof of part~B of Lemma~\ref{lem:wassintres} (see Lemma~\ref{lem:convwassdist}), we have for any $z^*(x) \in S^*(x)$
\[
\mathbb{P}\bigl\{\hv^{DRO}_n(x) > v^*(x) + \bigl(\expect{\norm{\nabla c(z^*(x),Y)}^2}\bigr)^{1/2} \zeta_n(\alpha,x) + 4L_2(z^*(x))\zeta^2_n(\alpha,x)\bigr\} \leq \alpha.
\]
Therefore, if we pick $\alpha \in (0,1)$ so that 
\[
\bigl(\expect{\norm{\nabla c(z^*(x),Y)}^2}\bigr)^{1/2} \zeta_n(\alpha,x) + 4L_2(z^*(x))\zeta^2_n(\alpha,x) \leq \kappa,
\]
then $\prob{g(\hz^{DRO}_n(x);x) > v^*(x) + \kappa} \leq 2\alpha$.
Similar to the analysis above, positive constants $\tilde{\Gamma}(\kappa,x)$ and $\tilde{\gamma}(\kappa,x)$ and inequality~\eqref{eqn:largedevoutofsampcost} can be obtained by bounding the smallest value of~$\alpha$ using Assumption~\ref{ass:reglargedevwass2} and equation~\eqref{eqn:wassradius} so that}
\begin{align*}
&{\bigl(\expect{\norm{\nabla c(z^*(x),Y)}^2}\bigr)^{1/2} \zeta_n(\alpha,x) + 4L_2(z^*(x))\zeta^2_n(\alpha,x) \leq \kappa. \tag*{\qed}}
\end{align*}

\subsection{Proof of Lemma~\ref{lem:unifconvsamprobust}}
\label{app:unifconvsamprobust}

{By first adding and subtracting $g^*_n(z;x)$, defined in problem~\eqref{eqn:fullinfsaa}, and then doing the same with $g^*_{s,n}(z;x)$, we obtain}
{

\begin{align}
\label{eqn:unifconvsamprobust}
\uset{z \in \Z}{\sup}\: \abs*{\hg^{ER}_{s,n}(z;x) - g(z;x)} \nonumber \leq& \: \uset{z \in \Z}{\sup} \: \uset{p \in \Pf_n(x;\zeta_{n}(x))}{\sup}\displaystyle\sum_{i=1}^{n} p_i \Big\lvert \sup_{y \in \hat{\Y}^{i}_n(x;\mu_n(x))} c(z,y) - c \left( z,f^*(x) + \varepsilon^i \right)\Big\rvert \nonumber \\
&\quad + \uset{z \in \Z}{\sup} \: \abs*{g^*_{s,n}(z;x) - g^*_n(z;x)} + \uset{z \in \Z}{\sup} \: \abs*{g^*_n(z;x) - g(z;x)}.
\end{align}
}%
{Consider the first term on the r.h.s.\ of~\eqref{eqn:unifconvsamprobust}.
We have for each $x \in \X$}
{

\begin{align*}
&\uset{z \in \Z}{\sup} \: \uset{p \in \Pf_n(x;\zeta_{n}(x))}{\sup}\displaystyle\sum_{i=1}^{n} p_i \Big\lvert \sup_{y \in \hat{\Y}^{i}_n(x;\mu_n(x))} c(z,y) - c \left( z,f^*(x) + \varepsilon^i \right)\Big\rvert \nonumber\\
\leq& \: \uset{z \in \Z}{\sup} \: \uset{p \in \Pf_n(x;\zeta_{n}(x))}{\sup}\displaystyle\sum_{i=1}^{n} p_i \sup_{y \in \hat{\Y}^{i}_n(x;\mu_n(x))} L(z) \norm{y - (f^*(x) + \varepsilon^i)} \nonumber\\
\leq& \: \uset{z \in \Z}{\sup} \: \uset{p \in \Pf_n(x;\zeta_{n}(x))}{\sup}\displaystyle\sum_{i=1}^{n} p_i L(z) \left(\mu_{n}(x) + \norm{\teps^{i}_{n}(x)}\right) \nonumber\\
=& \: \uset{z \in \Z}{\sup} \: L(z) \uset{p \in \Pf_n(x;\zeta_{n}(x))}{\sup}\displaystyle\sum_{i=1}^{n} p_i \left(\mu_{n}(x) + \norm{\teps^{i}_{n}(x)}\right) \nonumber\\
\leq& \: \uset{z \in \Z}{\sup} \: L(z) \biggl(\mu_{n}(x) + \biggl(\dfrac{1}{n}\displaystyle\sum_{i=1}^{n} \bigl( \norm{\teps^{i}_{n}(x)}\bigr)^2\biggr)^{\frac{1}{2}} \biggr) \uset{p \in \Pf_n(x;\zeta_{n}(x))}{\sup} \biggl( n \sum_{i=1}^{n} p^2_i\biggr)^{\frac{1}{2}} \nonumber\\
=& \: \uset{z \in \Z}{\sup} \: L(z) \biggl(\mu_{n}(x) + \biggl(\dfrac{1}{n}\displaystyle\sum_{i=1}^{n} \bigl( \norm{\teps^{i}_{n}(x)}\bigr)^2\biggr)^{\frac{1}{2}} \biggr) \uset{p \in \Pf_n(x;\zeta_{n}(x))}{\sup} \biggl( 1 + n \sum_{i=1}^{n} \Big(p_i - \frac{1}{n}\Big)^2\biggr)^{\frac{1}{2}}, \nonumber
\end{align*}
}%
{where the first step follows from Assumption~\ref{ass:equilipschitz}, the second step follows from the definition of the set~$\hat{\Y}^i_n(x;\mu_n(x))$, the triangle inequality, and inequality~\eqref{eqn:projlipschitz}, and the fourth step follows by applying the Cauchy-Schwarz inequality twice.}

{Next, consider the second term on the r.h.s.\ of~\eqref{eqn:unifconvsamprobust}.
For each $x \in \X$}
{

\begin{align*}
&\uset{z \in \Z}{\sup} \: \lvert g^*_{s,n}(z;x) - g^*_n(z;x) \rvert \nonumber\\
=& \: \uset{z \in \Z}{\sup} \Big\lvert \uset{p \in \Pf_n(x;\zeta_{n}(x))}{\sup} \sum_{i=1}^{n} p_i c(z,f^*(x) + \varepsilon^i) - \frac{1}{n} \sum_{i=1}^{n} c(z,f^*(x) + \varepsilon^i) \Big\rvert \nonumber\\
= & \: \uset{z \in \Z}{\sup} \Big\lvert \uset{p \in \Pf_n(x;\zeta_{n}(x))}{\sup} \sum_{i=1}^{n} \Big( p_i - \frac{1}{n} \Big) c(z,f^*(x) + \varepsilon^i) \Big\rvert \nonumber\\
\leq & \: \uset{p \in \Pf_n(x;\zeta_{n}(x))}{\sup} \biggl(n \sum_{i=1}^{n} \Big(p_i - \frac{1}{n}\Big)^2 \biggr)^{\frac{1}{2}} \uset{z \in \Z}{\sup} \biggl(\frac{1}{n}\sum_{i=1}^{n} \big(c(z,f^*(x) + \varepsilon^i)\big)^2\biggr)^{\frac{1}{2}}, \nonumber
\end{align*}
}%
{where the inequality follows by Cauchy-Schwarz.}
\qed

\subsection{Proof of Theorem~\ref{thm:samprobustrateofconvlq}}
\label{app:samprobustrateofconvlq}

{Since}
{

\begin{align}
\label{eqn:samprobustlq}
\norm*{\hv^{DRO}_n(X) - v^*(X)}_{L^q} &= \big\lVert \min_{z \in \Z} \hg^{ER}_{s,n}(z;X) - \min_{z \in \Z} g(z;X)\big\rVert_{L^q} \nonumber \\
&\leq \Big\lVert \uset{z \in \Z}{\sup} \abs*{\hg^{ER}_{s,n}(z;X) - g(z;X)}\Big\rVert_{L^q},
\end{align}
}%
{we look to establish uniform rates of convergence of $\hg^{ER}_{s,n}(\cdot;X)$ to $g(\cdot;X)$ with respect to the $L^q$-norm on~$\X$.
From~\eqref{eqn:unifconvsamprobust} and the triangle inequality, we have}
{

\begin{align}
\label{eqn:unifconvsamprobust2}
\Big\lVert \uset{z \in \Z}{\sup}\: \abs*{\hg^{ER}_{s,n}(z;X) - g(z;X)}\Big\rVert_{L^q} \nonumber \leq& \: \Big\lVert \uset{z \in \Z}{\sup} \: \abs*{\hg^{ER}_{s,n}(z;x) - g^*_{s,n}(z;X)} \Big\rVert_{L^q} + \nonumber \\
&\quad \Big\lVert \uset{z \in \Z}{\sup} \: \abs*{g^*_{s,n}(z;X) - g^*_n(z;X)}\Big\rVert_{L^q} + \nonumber \\
&\quad \Big\lVert \uset{z \in \Z}{\sup} \: \abs*{g^*_n(z;X) - g(z;X)}\Big\rVert_{L^q}.
\end{align}
}%
{We bound the terms on the r.h.s.\ of~\eqref{eqn:unifconvsamprobust2} using Lemma~\ref{lem:unifconvsamprobust}.
Assumptions~\ref{ass:equilipschitz}, \ref{ass:ambiguitysetunif}, and~\ref{ass:regconvrateunif} and $\norm{\mu_{n}(X)}_{L^q} = O(n^{-r/2})$
imply the first term on the r.h.s.\ of inequality~\eqref{eqn:unifconvsamprobustres} satisfies:
\begin{align*}
&\Big\lVert \uset{z \in \Z}{\sup} \: \abs*{\hg^{ER}_{s,n}(z;X) - g^*_{s,n}(z;X)} \Big\rVert_{L^q} \\
\leq \: & \bigg\lVert \uset{z \in \Z}{\sup} \: L(z) \biggl(\mu_{n}(X) + \biggl(\dfrac{1}{n}\displaystyle\sum_{i=1}^{n} \norm{\teps^{i}_{n}(X)}^2\biggr)^{\frac{1}{2}} \biggr) \uset{p \in \Pf_n(X;\zeta_{n}(X))}{\sup} \biggl( 1 + n \sum_{i=1}^{n} \Big(p_i - \frac{1}{n}\Big)^2\biggr)^{\frac{1}{2}} \bigg\rVert_{L^q} \\
= \: & O(1) \biggl[ \norm*{\mu_n(X)}_{L^q} + \bigg\lVert \biggl(\dfrac{1}{n}\displaystyle\sum_{i=1}^{n} \norm{\teps^{i}_{n}(X)}^2\biggr)^{\frac{1}{2}} \bigg\rVert_{L^q} \biggr] \\
= \: & O(1) \biggl[ \norm*{\mu_n(X)}_{L^q} + \norm*{f^*(X) - \hf_n(X)}_{L^q} + \bigg\lVert \biggl(\dfrac{1}{n}\displaystyle\sum_{i=1}^{n} \norm{f^*(x^i) - \hf_n(x^i)}^2\biggr)^{\frac{1}{2}} \bigg\rVert_{L^q} \biggr] \\
= \: & O_p(n^{-r/2}).
\end{align*}}%
{Assumptions~\ref{ass:ambiguitysetunif} and~\ref{ass:functionalclt2}
imply that the second term on the r.h.s.\ of inequality~\eqref{eqn:unifconvsamprobustres} satisfies
\begin{align*}
&\Big\lVert\sup_{z \in \Z} \abs*{g^*_{s,n}(z;X) - g^*_n(z;X)} \Big\rVert_{L^q} \\
\leq \: & \bigg\lVert \uset{p \in \Pf_n(X;\zeta_{n}(X))}{\sup} \biggl(n \sum_{i=1}^{n} \Big(p_i - \frac{1}{n}\Big)^2 \biggr)^{\frac{1}{2}} \uset{z \in \Z}{\sup} \biggl(\frac{1}{n}\sum_{i=1}^{n} \big(c(z,f^*(X) + \varepsilon^i)\big)^2\biggr)^{\frac{1}{2}}\bigg\rVert_{L^q} \\
= \: & O_p(1) \bigg\lVert \uset{p \in \Pf_n(X;\zeta_{n}(X))}{\sup} \biggl(n \sum_{i=1}^{n} \Big(p_i - \frac{1}{n}\Big)^2 \biggr)^{\frac{1}{2}} \bigg\rVert_{L^q} = O_p(n^{-r/2}).
\end{align*}
Finally, Assumption~\ref{ass:functionalclt} implies 
\[
\Big\lVert\sup_{z \in \Z} \abs*{g^*_{n}(z;X) - g(z;X)} \Big\rVert_{L^q} = O_p(n^{-1/2}).
\]
Putting the above three inequalities together into inequality~\eqref{eqn:unifconvsamprobust2}, we obtain
\[
\Big\lVert \uset{z \in \Z}{\sup} \abs*{\hg^{ER}_{s,n}(z;X) - g(z;X)}\Big\rVert_{L^q} = O_p(n^{-r/2}).
\]
The first part of the stated result then follows from~\eqref{eqn:samprobustlq}. 
The second part of the stated result follows from~\eqref{eqn:samprobustlq} and the fact that}
\begin{align*}
{\norm*{g(\hz^{DRO}_n(X);X) - v^*(X)}_{L^q}} &{\leq \norm*{g(\hz^{DRO}_n(X);X) - \hv^{DRO}_n(X)}_{L^q} +} \\
&\qquad {\norm*{\hv^{DRO}_n(X) - v^*(X)}_{L^q}.} \tag*{\qed}
\end{align*}

\section{Ambiguity sets satisfying Assumption~\ref{ass:ambiguityset}}
\label{app:ambiguitysetrate}

{In Section~\ref{sec:samprobustopt}, we outlined conditions under which phi-divergence ambiguity sets $\Pf_n(x;\zeta_{n}(x))$ satisfy Assumption~\ref{ass:ambiguityset} for a suitable choice of the radius~$\zeta_{n}(x)$}.
Lemma~\ref{lem:ambiguityset} below determines sharp bounds on the radius~$\zeta_{n}(x)$ for some other families of ambiguity sets to satisfy Assumption~\ref{ass:ambiguityset}.
Before presenting the lemma, we introduce a third example of the ambiguity set $\Pf_n(x;\zeta_{n}(x))$ to add to Examples~\ref{exm:cvar} and~\ref{exm:phidivergence} in Section~\ref{sec:drsaa}.

\begin{example}
\label{exm:meanuppersemidev}
Mean-upper-semideviation-based ambiguity sets~\cite{shapiro2009lectures}: given order $a \in [1,+\infty)$ and radius $\zeta_n(x) \geq 0$, let $b := a/(a-1)$ and define~$\hP_n(x)$ using
\begin{align*}
\mathfrak{P}_n(x;\zeta_n(x)) := \bigg\{p \in \R^{n}_+ : &\sum_{i=1}^{n} p_i = 1 \: \text{and } \exists \: q \in \R^n_+ \text{ such that } \norm{q}_b \leq \zeta_n(x), \\
&\quad p_i = \frac{1}{n} \bigg[ 1 + q_i - \frac{1}{n} \sum_{j=1}^{n} q_j \bigg], \forall i \in [n] \bigg\}.
\end{align*}
\end{example}

\begin{lemma}
\label{lem:ambiguityset}
The following ambiguity sets satisfy Assumption~\ref{ass:ambiguityset} with constant $\rho \in (1,2]$:
\begin{enumerate}[label=(\alph*)]
\item CVaR-based ambiguity sets (see Example~\ref{exm:cvar}) with radius $\zeta_{n}(x) = O(n^{1-\rho})$, 

\item Variation distance-based ambiguity sets (see Example~\ref{exm:phidivergence}) with radius $\zeta_{n}(x) = O(n^{-\rho/2})$,

\item Mean-upper-semideviation-based ambiguity sets of order $a \in [1,+\infty)$ (see Example~\ref{exm:meanuppersemidev}) with radius $\zeta_{n}(x) = \begin{cases} O(n^{1 - \rho/2}) &\mbox{if } a \geq 2 \\
O(n^{3/2 - 1/a - \rho/2}) & \mbox{if } a < 2 \end{cases}$.
\end{enumerate}
Furthermore, these bounds are sharp in the sense described above.
\end{lemma}
\begin{proof}
\begin{enumerate}[label=(\alph*)]
\item 
Assume that $\zeta_{n}(x) < 0.5$.
We begin by noting that there exists an optimal solution to the problem $\sup_{p \in \Pf_n(x;\zeta_{n}(x))} \sum_{i=1}^{n} \big(p_i - \frac{1}{n}\big)^2$ that is an extreme point of the polytope $\Pf_n(x;\zeta_{n}(x))$.
Furthermore, every extreme point of $\Pf_n(x;\zeta_{n}(x))$ satisfies at least $n-1$ of the set of $2n$ inequalities $\Bigl\{p_i \geq 0, \: i \in [n], \: p_i \leq \frac{1}{n(1-\zeta_{n}(x))}, \: i \in [n]\Bigr\}$, with equality.
This implies that there exists an optimal solution at which at least $n-1$ of the $p_i$s either take the value zero, or take the value $\frac{1}{n(1-\zeta_{n}(x))}$.
At this solution, $n-1$ of the terms $\big(p_i - \frac{1}{n}\big)^2$ are either $\frac{1}{n^2}$ or $\frac{1}{n^2} \bigl(\frac{\zeta_{n}(x)}{1-\zeta_{n}(x)}\bigr)^2$ (with $\frac{1}{n^2}$ larger since $\zeta_{n}(x) < 0.5$ by assumption).

\hspace*{0.17in} Suppose $M \in \{0,\dots,n-1\}$ of the inequalities $p_i \geq 0$, $i \in [n]$, are satisfied with equality at such an optimal solution.
Since $\sum_{i=1}^{n} p_i = 1$ and $p_i \leq \frac{1}{n(1-\zeta_{n}(x))}$, $\forall i \in [n]$, we require $(n-M)\frac{1}{n(1-\zeta_{n}(x))} \geq 1$, which implies $M \leq n\zeta_{n}(x)$.
Consequently, $M \leq n\zeta_{n}(x) < n/2$ of the inequalities $p_i \geq 0$, $i \in [n]$, are satisfied with equality and at least $(n-1-M) \geq n(1-\zeta_{n}(x)) - 1 > n/2 - 1$ of the inequalities $p_i \leq \frac{1}{n(1-\zeta_{n}(x))}$, $i \in [n]$, are satisfied with equality.
Therefore, whenever $\zeta_n(x) < 0.5$, we have:
\begin{align*}
\sum_{i=1}^{n} \bigg(p_i - \frac{1}{n}\bigg)^2 &\leq (n\zeta_{n}(x)+1)\frac{1}{n^2} + n(1-\zeta_{n}(x))\frac{1}{n^2} \biggl(\frac{\zeta_{n}(x)}{1-\zeta_{n}(x)}\biggr)^2 \\
&= \frac{1}{n^2} + \frac{1}{n}\biggl(\frac{\zeta_{n}(x)}{1-\zeta_{n}(x)}\biggr).
\end{align*}
Because the above analysis is constructive, it can be immediately used to deduce that the bound on~$\zeta_n(x)$ is sharp.

\item The stated result follows from the fact that 
\[
\sum_{i=1}^{n} \left(p_i - \frac{1}{n}\right)^2 \leq \left(\sum_{i=1}^{n} \abs*{p_i - \frac{1}{n}}\right)^2 \leq \zeta^2_{n}(x), \quad \forall p \in \Pf_n(x;\zeta_{n}(x)), \: x \in \X.
\]
To see that the above bound is sharp, assume without loss of generality that $n \geq 2$ and $\zeta_{n}(x) \leq 1$.
Then, because
\[
\biggl(\frac{1}{n} + \frac{\zeta_{n}(x)}{2}, \underbrace{\frac{1}{n} - \frac{\zeta_{n}(x)}{2n-2},\dots,\frac{1}{n} - \frac{\zeta_{n}(x)}{2n-2}}_{n-1 \text{ terms}}\biggr) \in \Pf_n(x;\zeta_{n}(x)),
\]
we have
\[
\uset{p \in \Pf_n(x;\zeta_{n}(x))}{\sup} \sum_{i=1}^{n} \Big(p_i - \frac{1}{n}\Big)^2 \geq \frac{\zeta^2_n(x)}{4} + \frac{\zeta^2_n(x)}{4(n-1)}.
\]

\item  Let $\bar{q} := \frac{1}{n}\sum_{i=1}^{n} q_i$. 
We have:
\[
\uset{p \in \Pf_n(x;\zeta_{n}(x))}{\sup} \sum_{i=1}^{n} \Big(p_i - \frac{1}{n}\Big)^2 \leq \uset{q \in \mathfrak{Q}_n(x;\zeta_{n}(x))}{\sup} \frac{1}{n^2}\sum_{i=1}^{n} (q_i - \bar{q})^2,
\]
where $\mathfrak{Q}_n(x;\zeta_{n}(x)) := \Set{q \in \R^n_+}{\norm{q}_b \leq \zeta_{n}(x)}$.
Note that for each $q \in \mathfrak{Q}_n(x;\zeta_{n}(x))$, we have $\abs{\bar{q}} \leq n^{-1} \norm{q}_1 \leq n^{-1/b} \norm{q}_b$, which in turn implies 
\[
\norm{q - \bar{q} \mathbf{1}}_b \leq \norm{q}_b + \abs{\bar{q}} \norm{\mathbf{1}}_b = \norm{q}_b + \abs{\bar{q}} n^{1/b} \leq \norm{q}_b + \norm{q}_b \leq 2\zeta_{n}(x),
\]
where $\mathbf{1}$ is a vector of ones of appropriate dimension.
Additionally, note that
\[
\sum_{i=1}^{n} (q_i - \bar{q})^2 = \norm{q - \bar{q}\mathbf{1}}^2 \leq \begin{cases} \norm{q - \bar{q}\mathbf{1}}^2_b &\mbox{if } b \leq 2 \\
n^{1 - 2/b} \norm{q - \bar{q}\mathbf{1}}^2_b & \mbox{if } b > 2 \end{cases}.
\]
The desired result then follows from
\begin{align*}
\uset{q \in \mathfrak{Q}_n(x;\zeta_{n}(x))}{\sup} \frac{1}{n^2}\sum_{i=1}^{n} (q_i - \bar{q})^2 &\leq \uset{\Set{q}{\norm{q - \bar{q}\mathbf{1}}_b \leq 2\zeta_{n}(x)}}{\sup} \frac{1}{n^2}\norm{q - \bar{q}\mathbf{1}}^2 \\
&\leq \begin{cases} \frac{4}{n^2}\zeta^2_{n}(x) &\mbox{if } b \leq 2 \\
\frac{4}{n^{1 + 2/b}} \zeta^2_{n}(x) & \mbox{if } b > 2 \end{cases}.
\end{align*}
We now show that the above bounds are sharp.

\hspace*{0.17in} Consider first the case when $b \leq 2$ and assume without loss of generality that $\zeta_{n}(x) = O(\sqrt{n})$.
Note that $p_i = \frac{1}{n} \big[ 1 + q_i - \frac{1}{n} \sum_{j=1}^{n} q_j \big]$, $i \in [n]$, with $q_1 = \zeta_{n}(x)$ and $q_i = 0$, $\forall i \geq 2$, is an element of~$\Pf_n(x;\zeta_{n}(x))$.
Therefore
\[
\uset{p \in \Pf_n(x;\zeta_{n}(x))}{\sup} \sum_{i=1}^{n} \Big(p_i - \frac{1}{n}\Big)^2 = \Theta\biggl(\frac{\zeta^2_{n}(x)}{n^2}\biggr).
\]

Next, suppose instead that $b > 2$ and assume without loss of generality that $\zeta_{n}(x) = O(n^{1/b})$.
Note that $p_i = \frac{1}{n} \big[ 1 + q_i - \frac{1}{n} \sum_{j=1}^{n} q_j \big]$ with $q_i = \begin{cases} \bigl(\frac{2}{n}\bigr)^{1/b}\zeta_{n}(x) &\mbox{if } i \equiv 0 \;(\bmod\; 2) \\
0 & \mbox{if } i \equiv 1 \;(\bmod\; 2) \end{cases}$, $i \in [n]$, is an element of~$\Pf_n(x;\zeta_{n}(x))$.
Therefore
\[
\uset{p \in \Pf_n(x;\zeta_{1,n}(x))}{\sup} \sum_{i=1}^{n} \Big(p_i - \frac{1}{n}\Big)^2 = \Theta\biggl(\frac{\zeta^2_{n}(x)}{n^{1+2/b}}\biggr). \qedhere
\]
\end{enumerate}
\end{proof}

\section{Additional computational results}
\label{app:computres}

\begin{algorithm}[t]
\caption{
{Estimating the $99$\% UCB on the optimality gap of a solution}}
\label{alg:99ucbs}
{
\begin{algorithmic}[1]
\State {\textbf{Input}: Covariate realization $X = x$ and data-driven solution $\hz_n(x)$ for a particular realization of the data~$\D_n$.}

\State {\textbf{Output}: $\hat{B}_{99}(x)$, which is a normalized estimate of the $99$\% UCB on the out-of-sample optimality gap of $\hz_n(x)$.}

\For{{$k = 1, \dots, 30$}}

\State {Draw $10^5$ i.i.d.\ samples $\bar{\D}^k := \{\bar{\varepsilon}^{k,i}\}_{i=1}^{10^5}$ of $\varepsilon$ according to $P_{\varepsilon}$.}

\State {Estimate the optimal value $v^*(x)$ by solving the full-information SAA} 

\Statex \hspace*{0.2in} {problem~\eqref{eqn:fullinfsaa} using the scenarios $\{f^*(x)+\bar{\varepsilon}^{k,i}\}_{i=1}^{10^5}$ constructed with~$\bar{\D}^k$}

\State {Estimate the out-of-sample cost of the solution $\hz_n(x)$ using the first}

\Statex \hspace*{0.2in} {$20{,}000$ samples of~$\bar{\D}^k$, i.e., $\hat{v}^k(x) := \frac{1}{20000} \sum_{i=1}^{2 \times 10^4} c(\hz_n(x),f^*(x)+\bar{\varepsilon}^{k,i})$}

\Statex \vspace*{-0.05in}

\State {Estimate the optimality gap of the $\hz_n(x)$ as $\hat{G}^k(x) = \hat{v}^k(x) - \bar{v}^k(x)$.}

\EndFor

\State {Construct the estimate of the $99$\% UCB on the optimality gap of $\hz_n(x)$ as 
\[
\hat{B}_{99}(x) := 100 \Bigl( \texttt{avg}(\{\hat{G}^k(x)\}) + 2.462\sqrt{\texttt{var}(\{\hat{G}^k(x)\}) / 30}\Bigr),
\]
\mbox{where $\texttt{avg}(\{\hat{G}^k(x)\})$/$\texttt{var}(\{\hat{G}^k(x)\})$ denote the mean/variance of estimates.}}
\end{algorithmic}
}
\end{algorithm}

{
Algorithm~\ref{alg:99ucbs} describes our procedure for estimating the $99$\% UCB on the optimality gap of our data-driven solutions using the multiple replication procedure~\citep{mak1999monte}.
We only use $20{,}000$ of the generated $10^5$ samples from the conditional distribution of $Y$ given $X = x$ to compute these UCBs since they are sufficient to yield an accurate estimate of the optimality gaps.
Unlike~\citep[Algorithm~1]{kannan2020data} that uses {\it relative} optimality gaps, we use {\it absolute} optimality gaps in our $99$\% UCB estimates to avoid division by small quantities when $v^*(x)$ is close to zero.}

We compare Algorithm~\ref{alg:ersaasameradius} with an ``optimal covariate-independent'' specification of the radius~$\zeta_n$.
This optimal covariate-independent radius is determined by choosing~$\zeta_n$ such that the medians of the $99\%$ UCBs over the $20$ different covariate realizations are minimized.
We also benchmark Algorithm~\ref{alg:ersaadiffradius} against an ``optimal covariate dependent'' specification of~$\zeta_n(x)$ that is determined by choosing~$\zeta_n(x)$ such that the $99\%$ UCBs are minimized.
Determining these optimal covariate-independent and covariate-dependent radii~$\zeta_n(x)$ is impractical because it requires {$20{,}000$} i.i.d.\ samples from the conditional distribution of $Y$ given $X = x$ (which a decision-maker does not have).
We consider it only to benchmark the performance of Algorithms~\ref{alg:ersaasameradius} and~\ref{alg:ersaadiffradius}. 
\\

\begin{figure}[t!]
    \centering
    \begin{subfigure}[t]{0.33\textwidth}
        \centering
        \includegraphics[width=\textwidth]{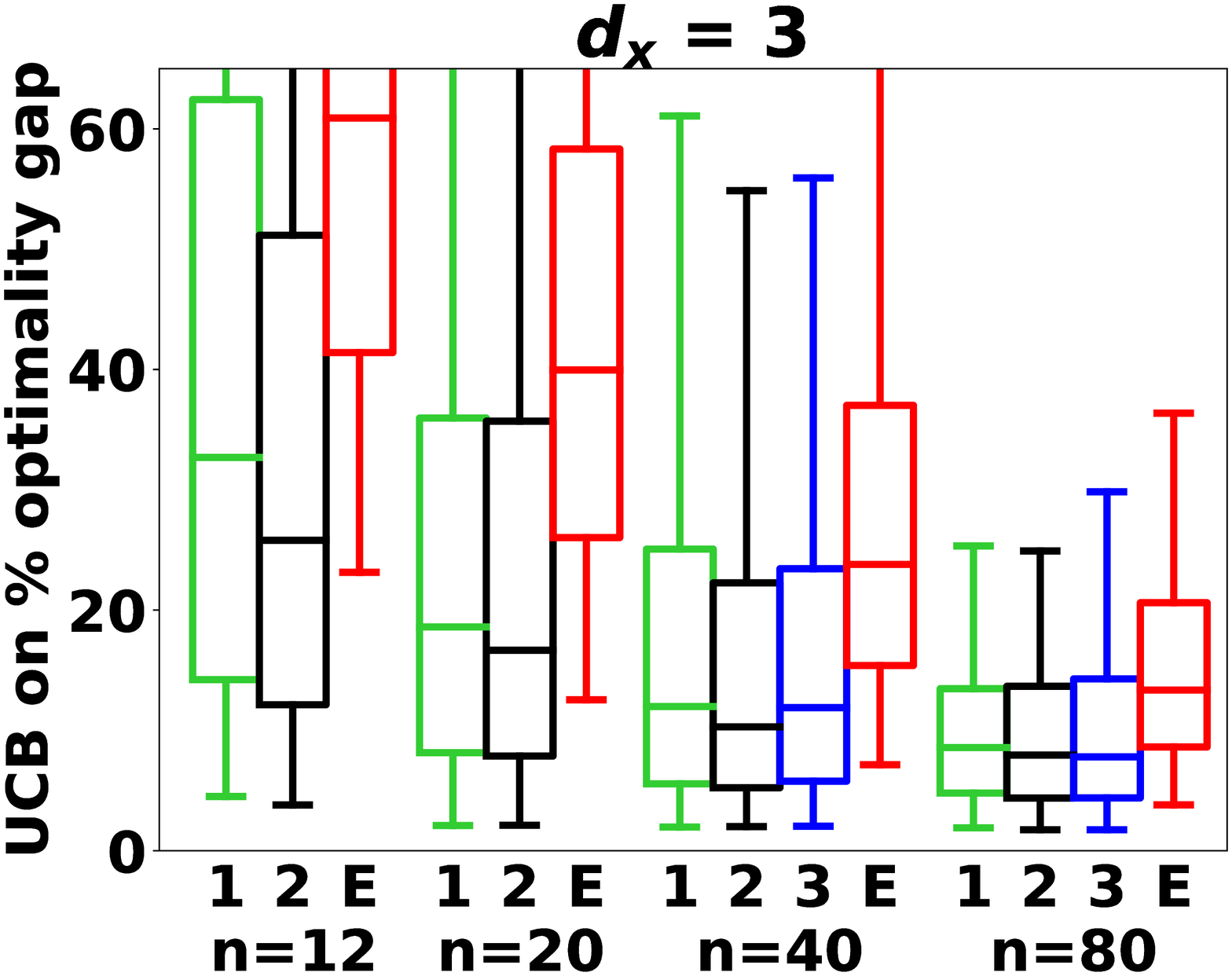}
    \end{subfigure}%
    ~ 
    \begin{subfigure}[t]{0.33\textwidth}
        \centering
        \includegraphics[width=\textwidth]{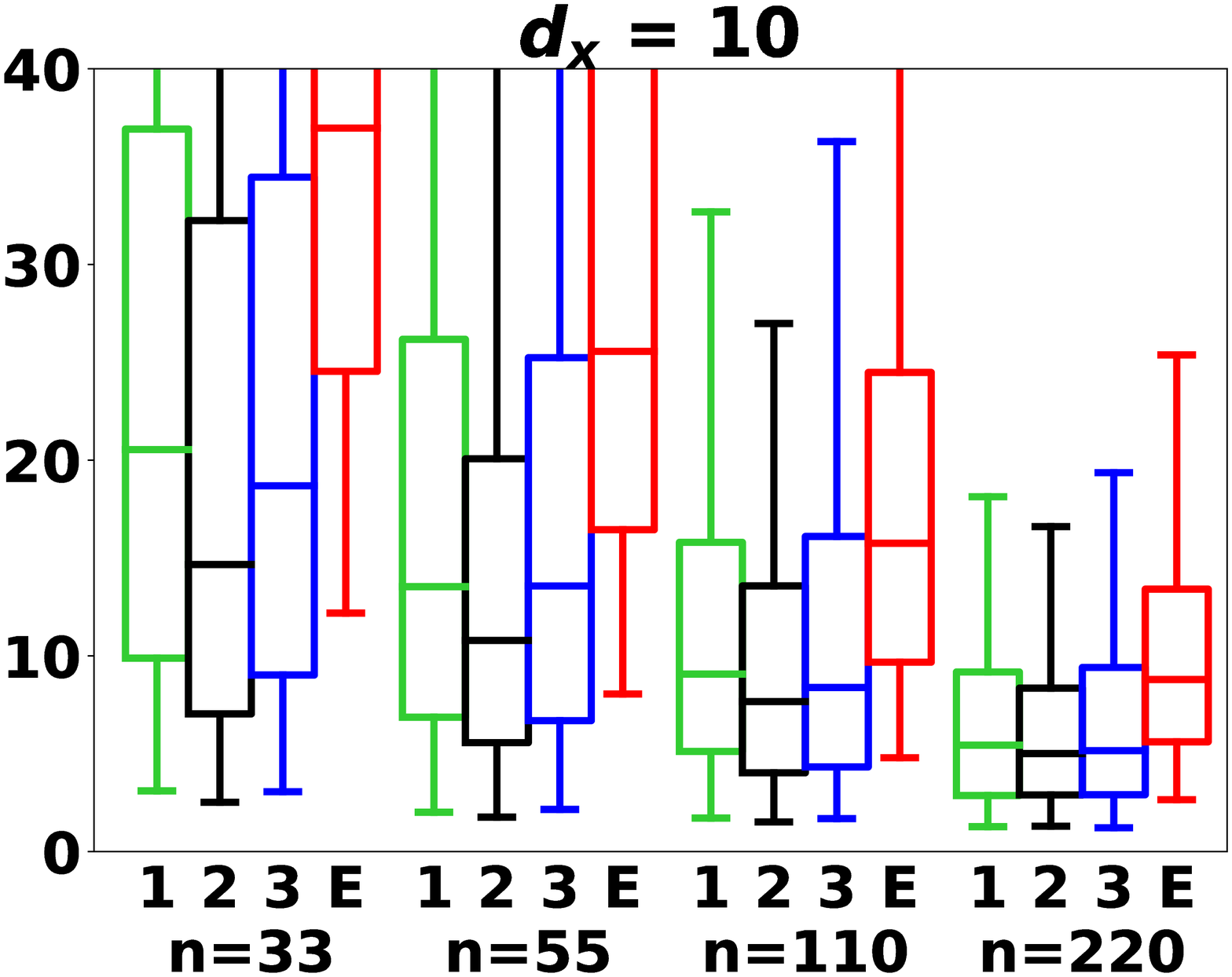}
    \end{subfigure}%
    ~ 
    \begin{subfigure}[t]{0.33\textwidth}
        \centering
        \includegraphics[width=\textwidth]{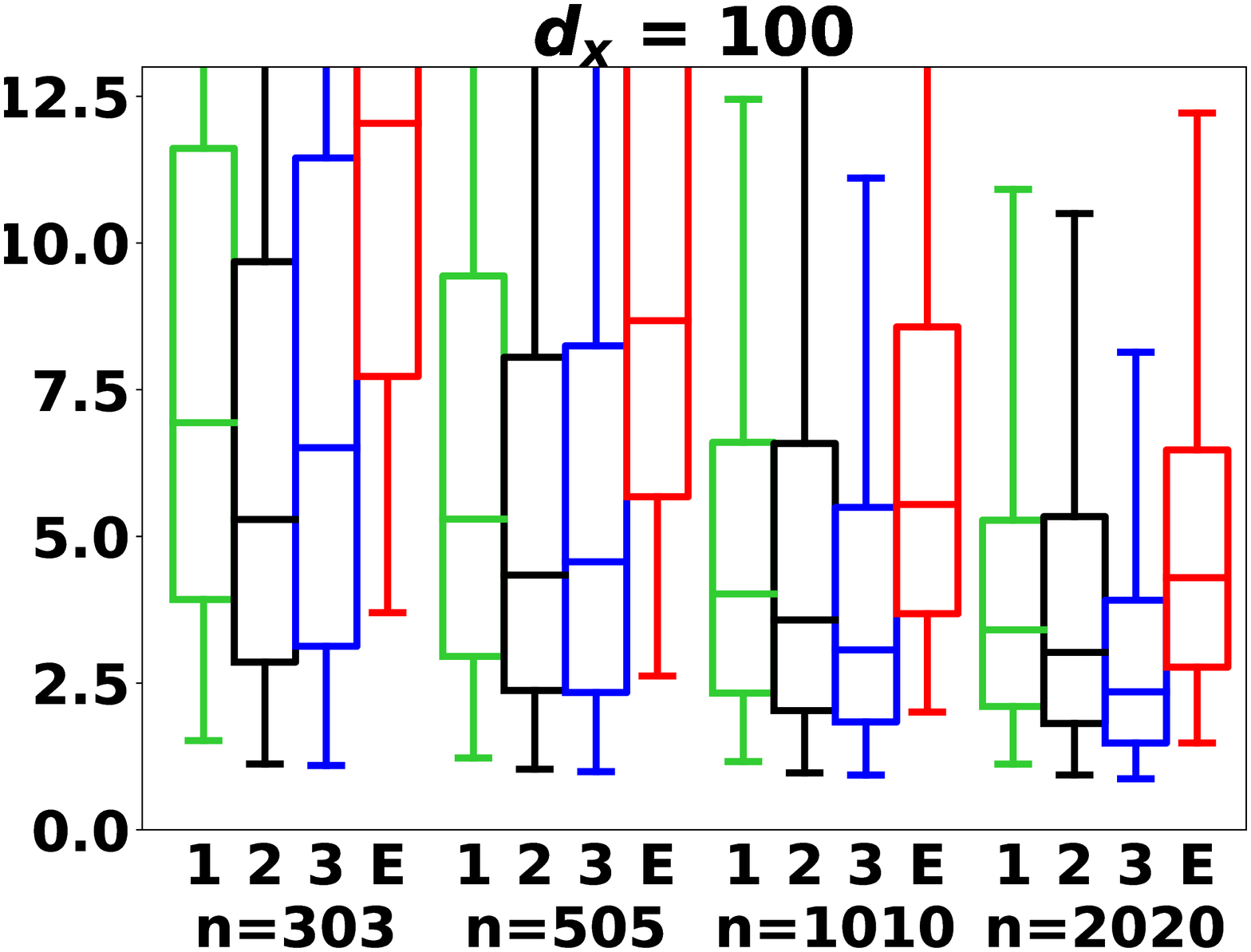}
    \end{subfigure}\\
    \begin{subfigure}[t]{0.33\textwidth}
        \centering
        \includegraphics[width=\textwidth]{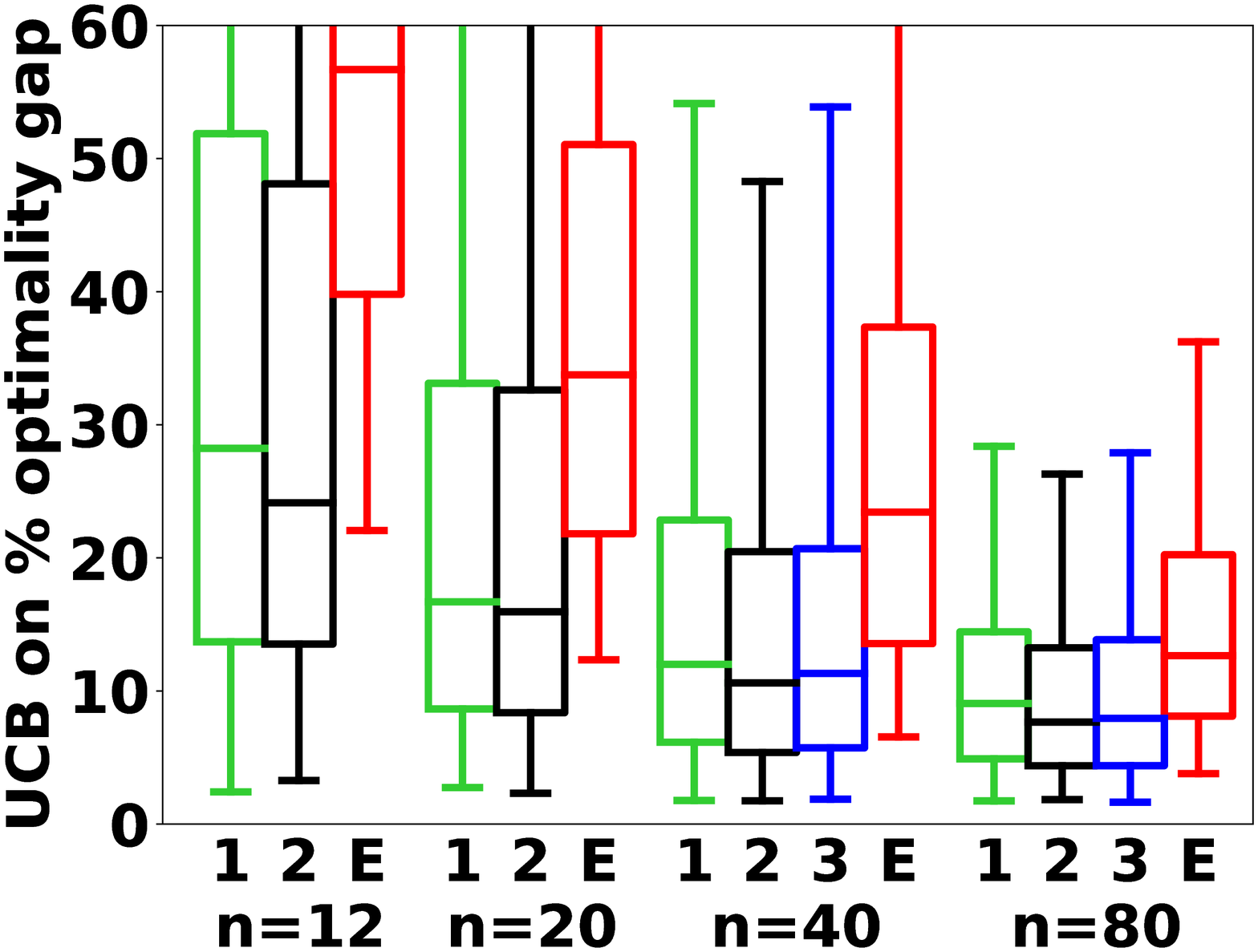}
    \end{subfigure}%
    ~ 
    \begin{subfigure}[t]{0.33\textwidth}
        \centering
        \includegraphics[width=\textwidth]{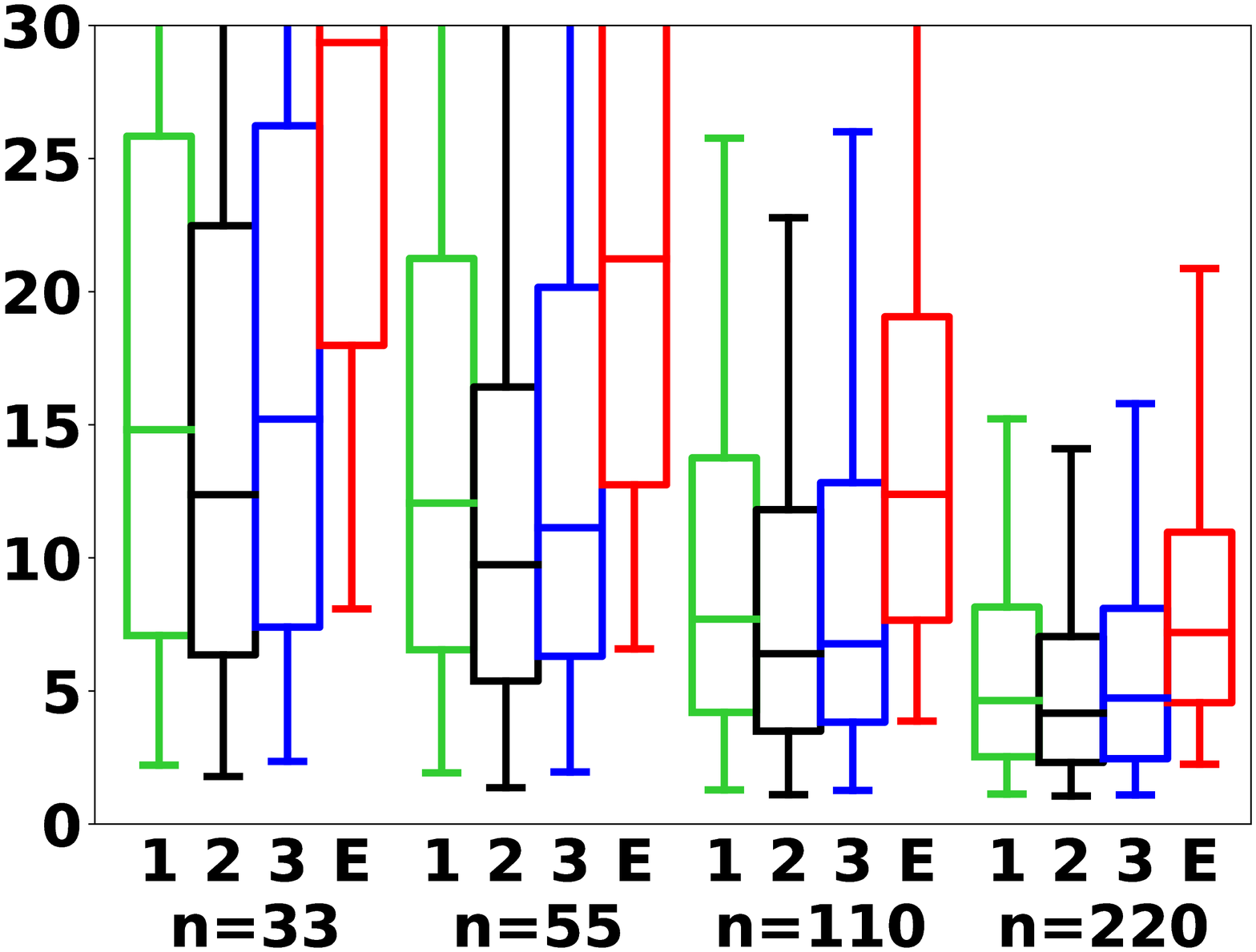}
    \end{subfigure}%
    ~ 
    \begin{subfigure}[t]{0.33\textwidth}
        \centering
        \includegraphics[width=\textwidth]{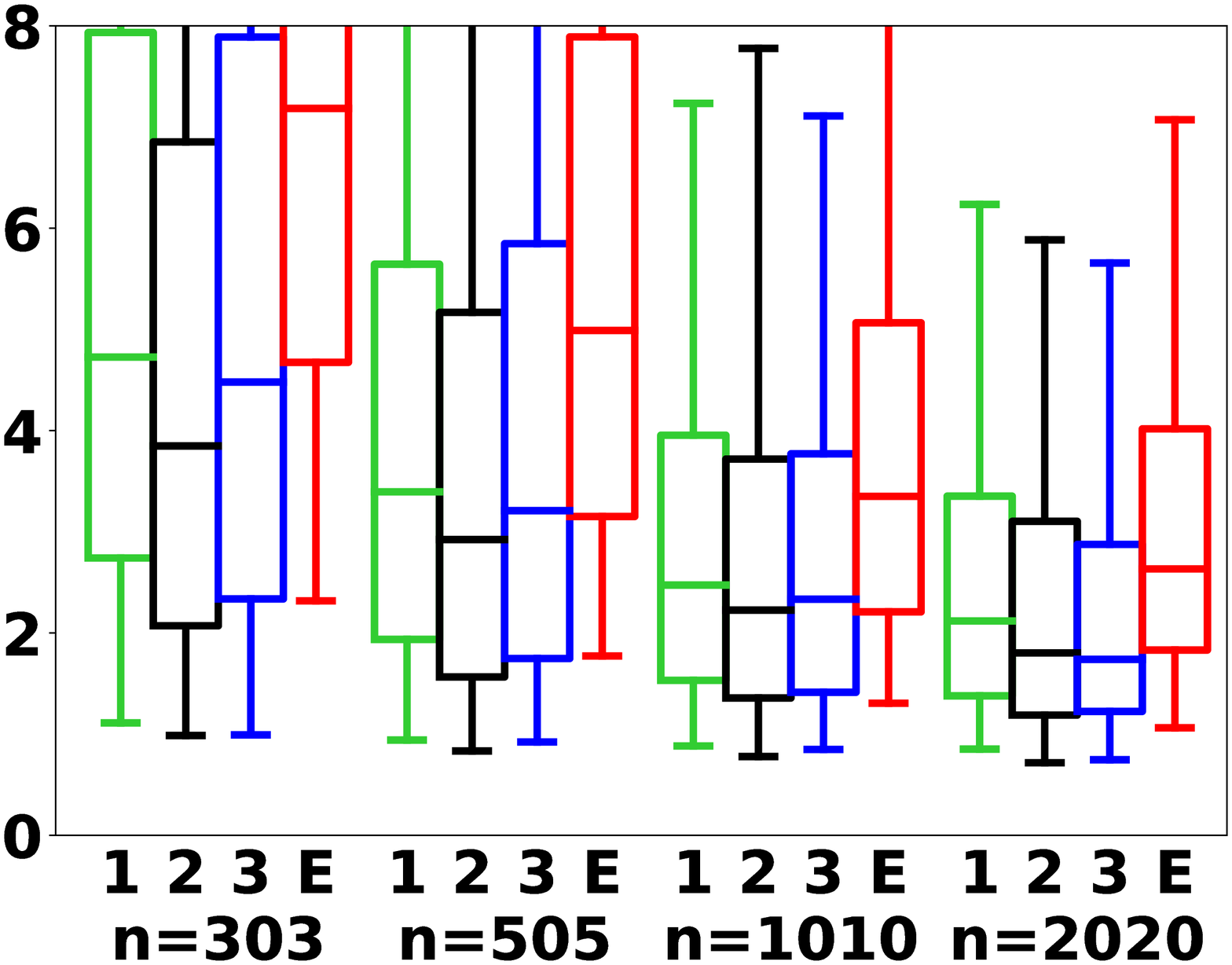}
    \end{subfigure}\\
    \begin{subfigure}[t]{0.33\textwidth}
        \centering
        \includegraphics[width=\textwidth]{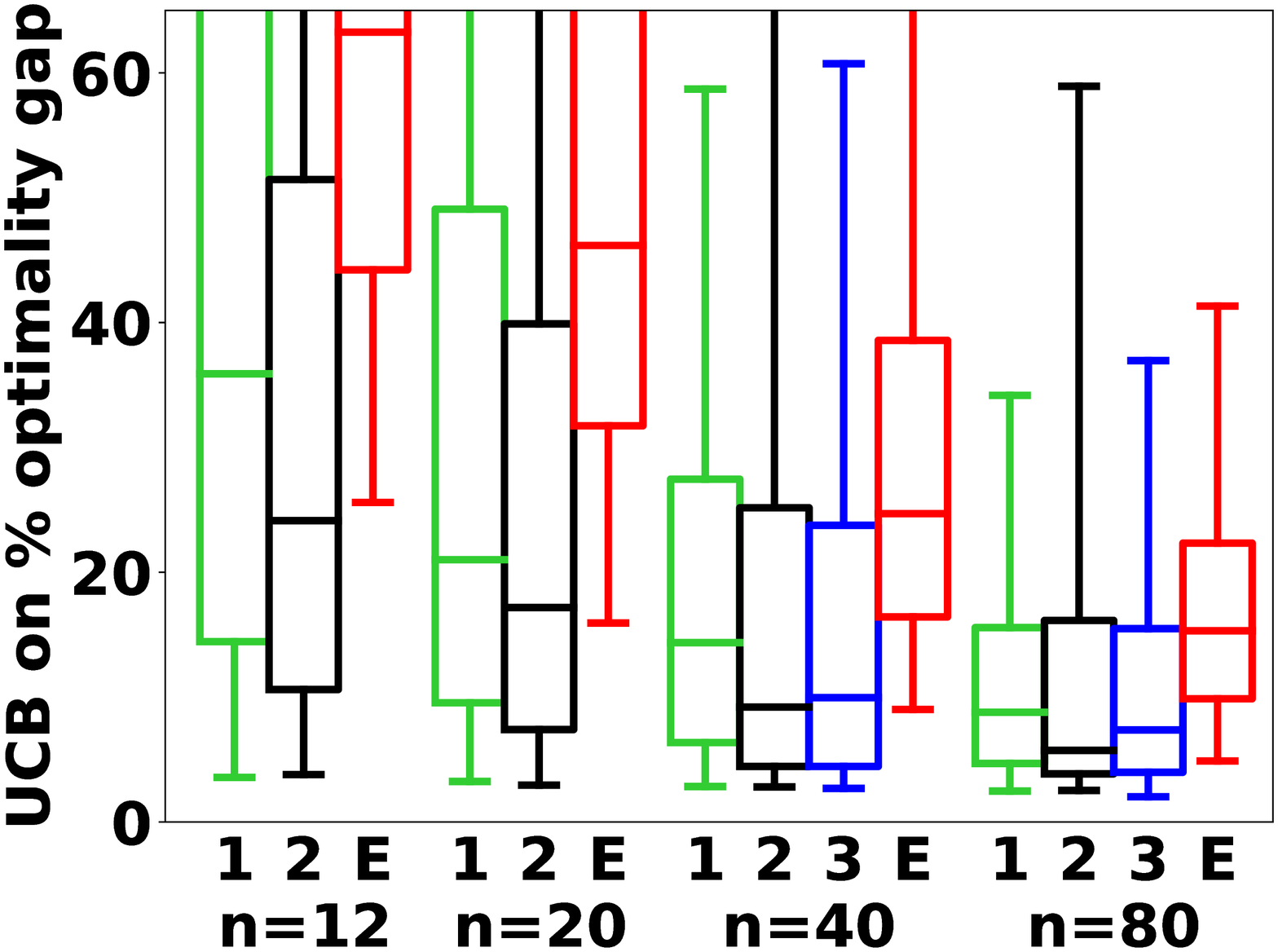}
    \end{subfigure}%
    ~ 
    \begin{subfigure}[t]{0.33\textwidth}
        \centering
        \includegraphics[width=\textwidth]{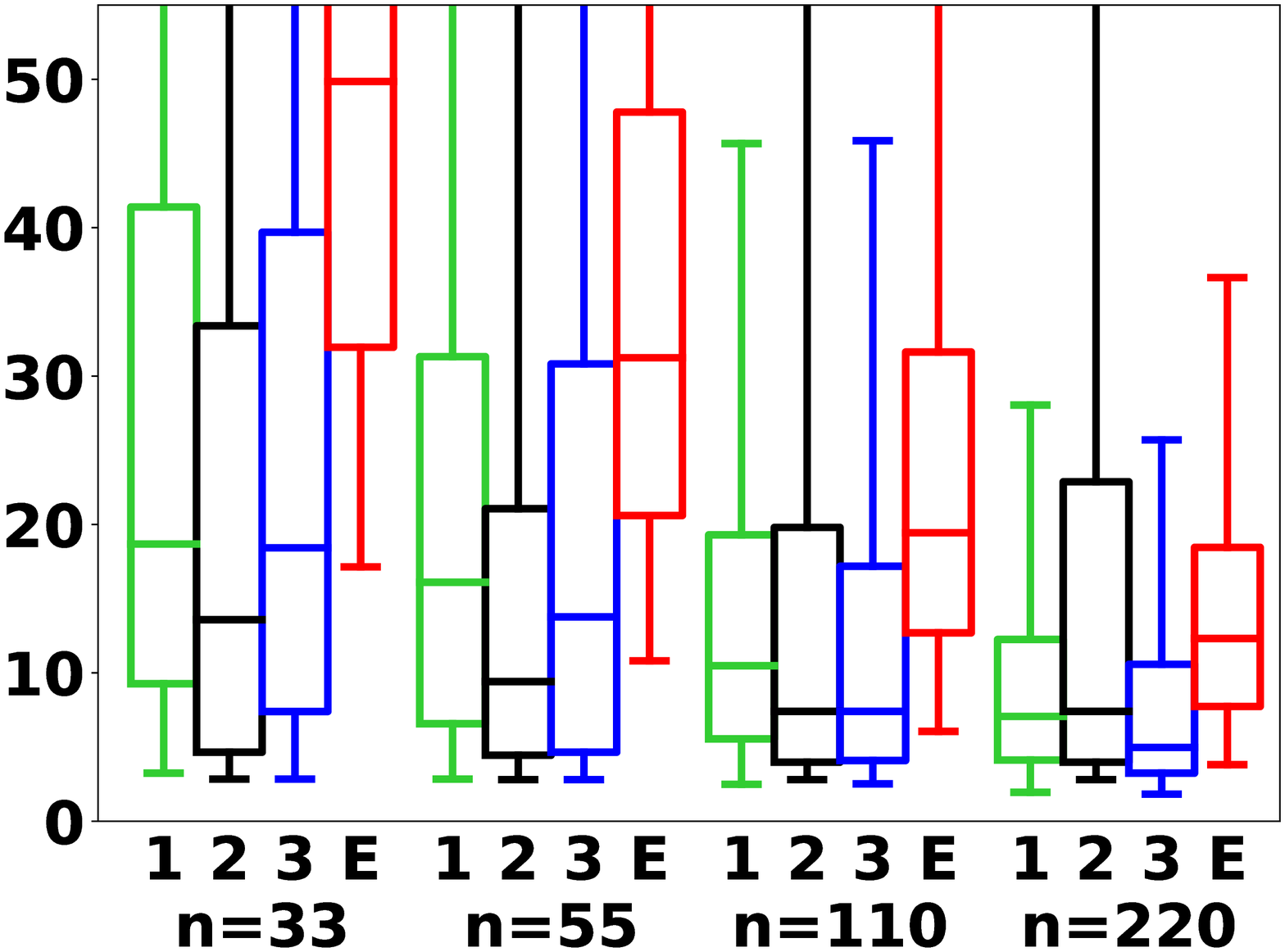}
    \end{subfigure}%
    ~ 
    \begin{subfigure}[t]{0.33\textwidth}
        \centering
        \includegraphics[width=\textwidth]{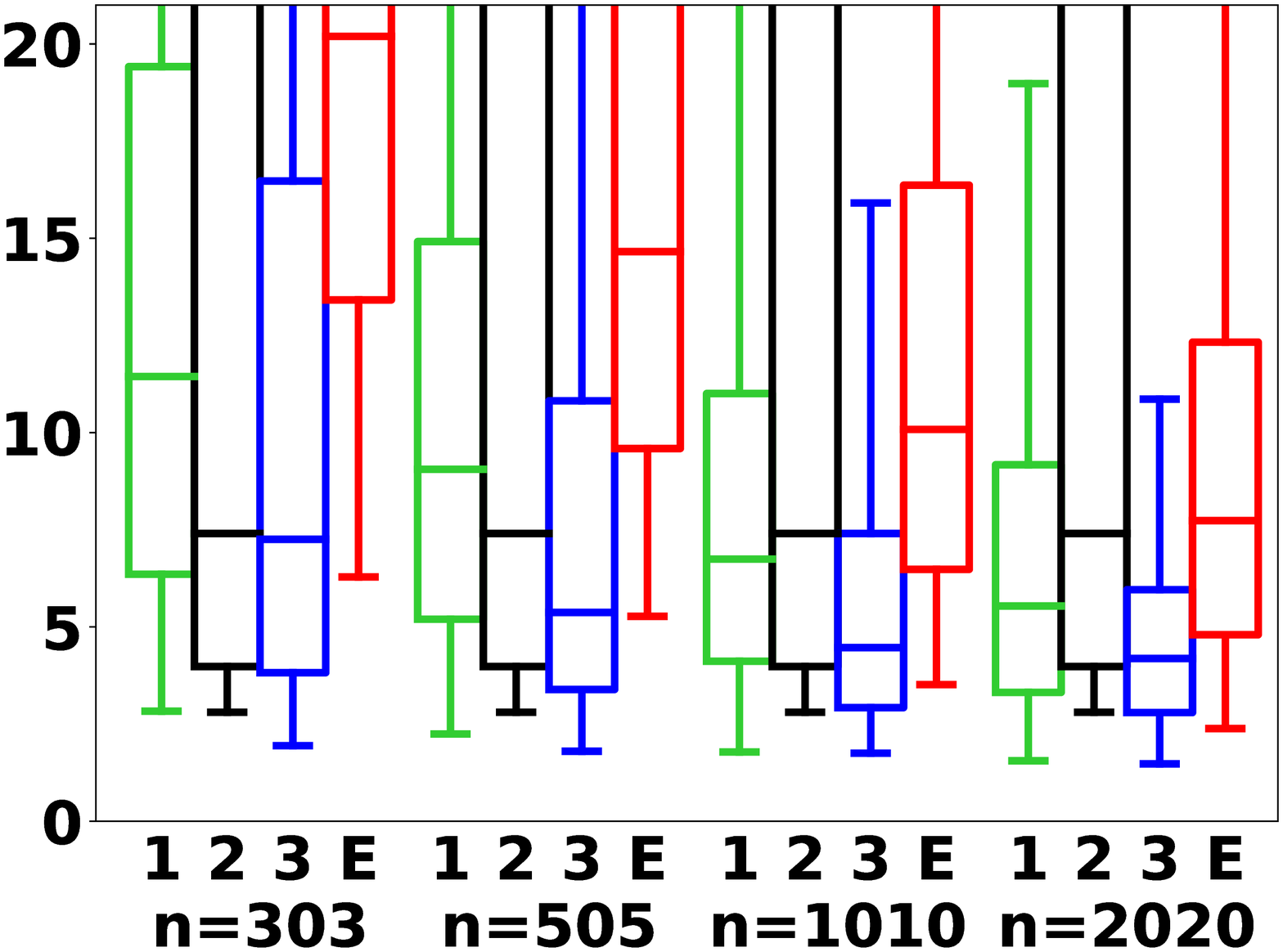}
    \end{subfigure}
    \caption{{\textbf{(Wasserstein-DRO with Ridge regression)} Comparison of the \tE+Ridge approach (\texttt{E}) with tuning of the \tW+Ridge radius using Algorithms~\ref{alg:naivesaaradius} (\tone),~\ref{alg:ersaasameradius} (\ttwo), and~\ref{alg:ersaadiffradius} (\tthree). Top row: $\theta = 1$. Middle row: $\theta = 0.5$. Bottom row: $\theta = 2$. Left column: $d_x = 3$. Middle column: $d_x = 10$. Right column: $d_x = 100$.}}
    \label{fig:comp_ridge_wass}
\end{figure}

\noindent
\textbf{``Optimal'' tuning of the Wasserstein radius.}
Figure~\ref{fig:comp_ols_wass_dep} compares the performance of the \tW+OLS formulations when the radius~$\zeta_n(x)$ of the ambiguity set is determined using Algorithms~\ref{alg:ersaasameradius} and~\ref{alg:ersaadiffradius} and optimal covariate-dependent and covariate-independent tuning.
We vary the model degree $\theta$, the covariate dimension among $d_x \in \{3,10,100\}$, and the {sample size among $n \in \{5(d_x + 1),10(d_x + 1),20(d_x + 1),50(d_x + 1)\}$} in these experiments.
{The radius specified by Algorithm~\ref{alg:ersaasameradius} performs better than the radius specified using Algorithm~\ref{alg:ersaadiffradius} for smaller sample sizes and covariate dimensions, and the converse holds for larger covariate dimensions and sample sizes}.
These results indicate that while covariate-dependent tuning theoretically has potential to yield better results than the covariate-independent tuning of Algorithm~\ref{alg:ersaasameradius}, Algorithm~\ref{alg:ersaadiffradius} is {only} able to obtain good estimates of the optimal covariate-dependent radius~$\zeta_n(x)$ {for relatively large sample sizes $n$}.
The difference between the performance of Algorithm~\ref{alg:ersaasameradius} and the optimal covariate-independent tuning reduces with increasing sample size and covariate dimension {except for $\theta = 2$}.
The difference between the performance of Algorithm~\ref{alg:ersaadiffradius} and optimal covariate-dependent tuning of the radius also {reduces} with increasing covariate dimension and sample size.

\begin{figure}[t!]
    \centering
    \begin{subfigure}[t]{0.33\textwidth}
        \centering
        \includegraphics[width=\textwidth]{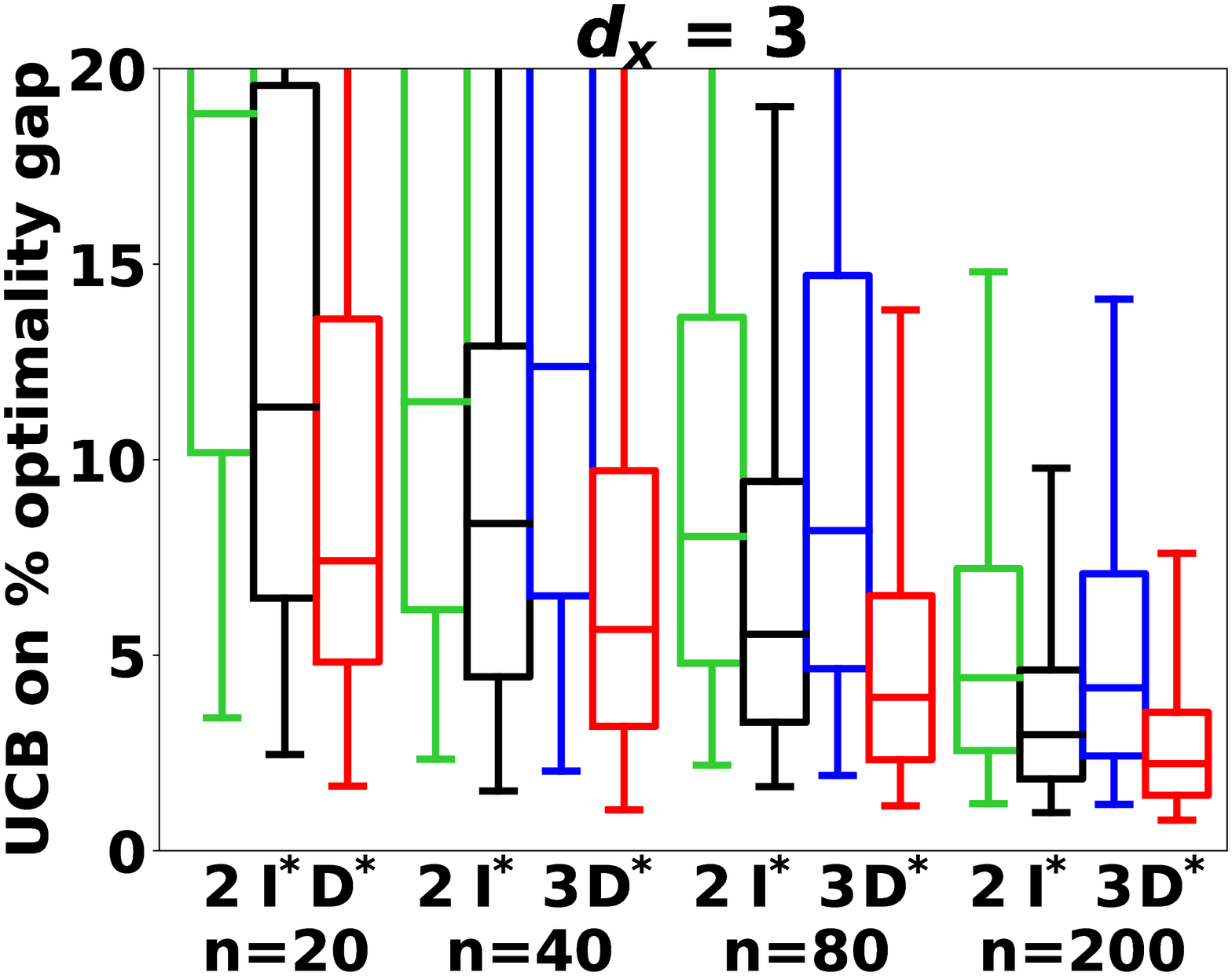}
    \end{subfigure}%
    ~ 
    \begin{subfigure}[t]{0.33\textwidth}
        \centering
        \includegraphics[width=\textwidth]{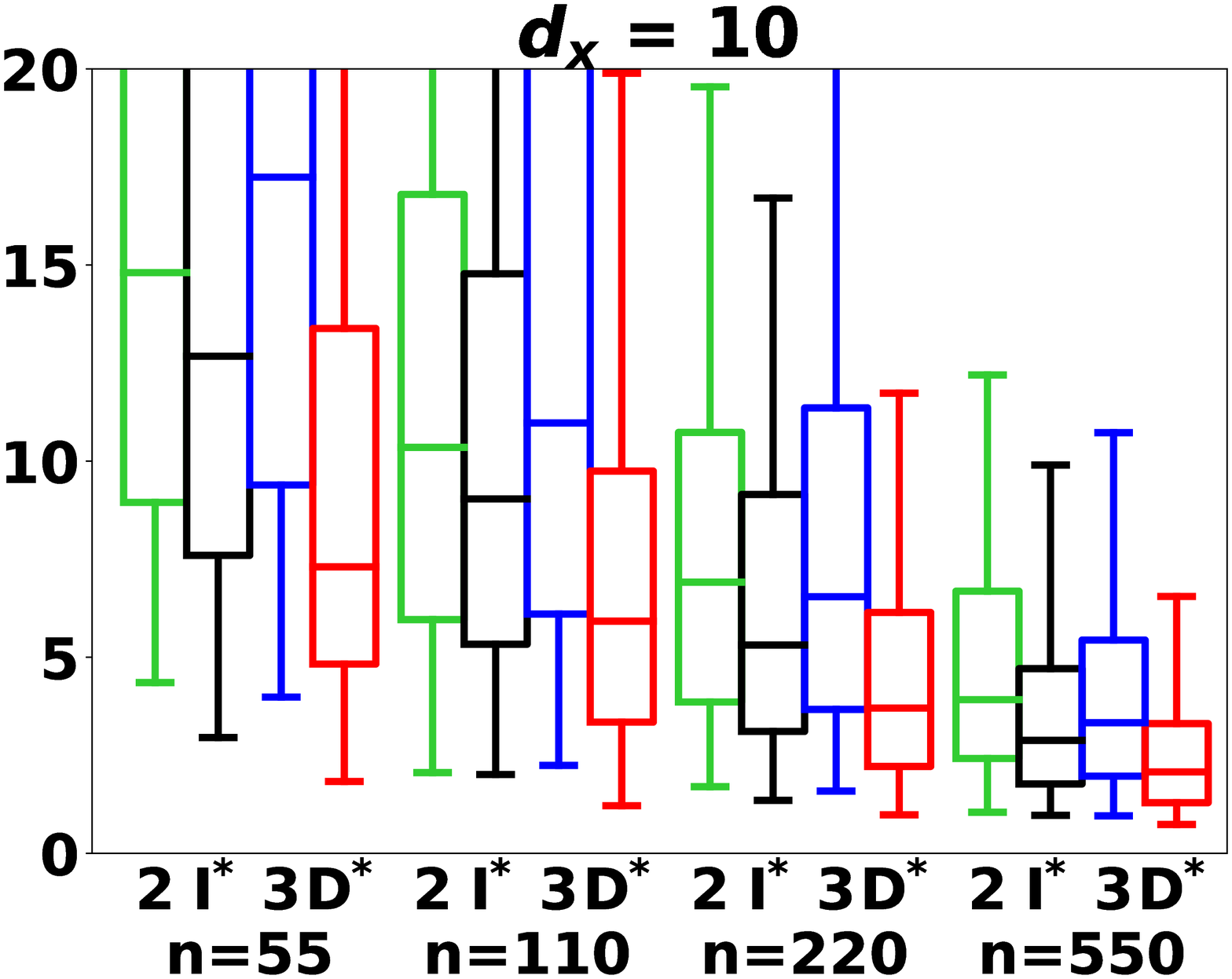}
    \end{subfigure}%
    ~ 
    \begin{subfigure}[t]{0.33\textwidth}
        \centering
        \includegraphics[width=\textwidth]{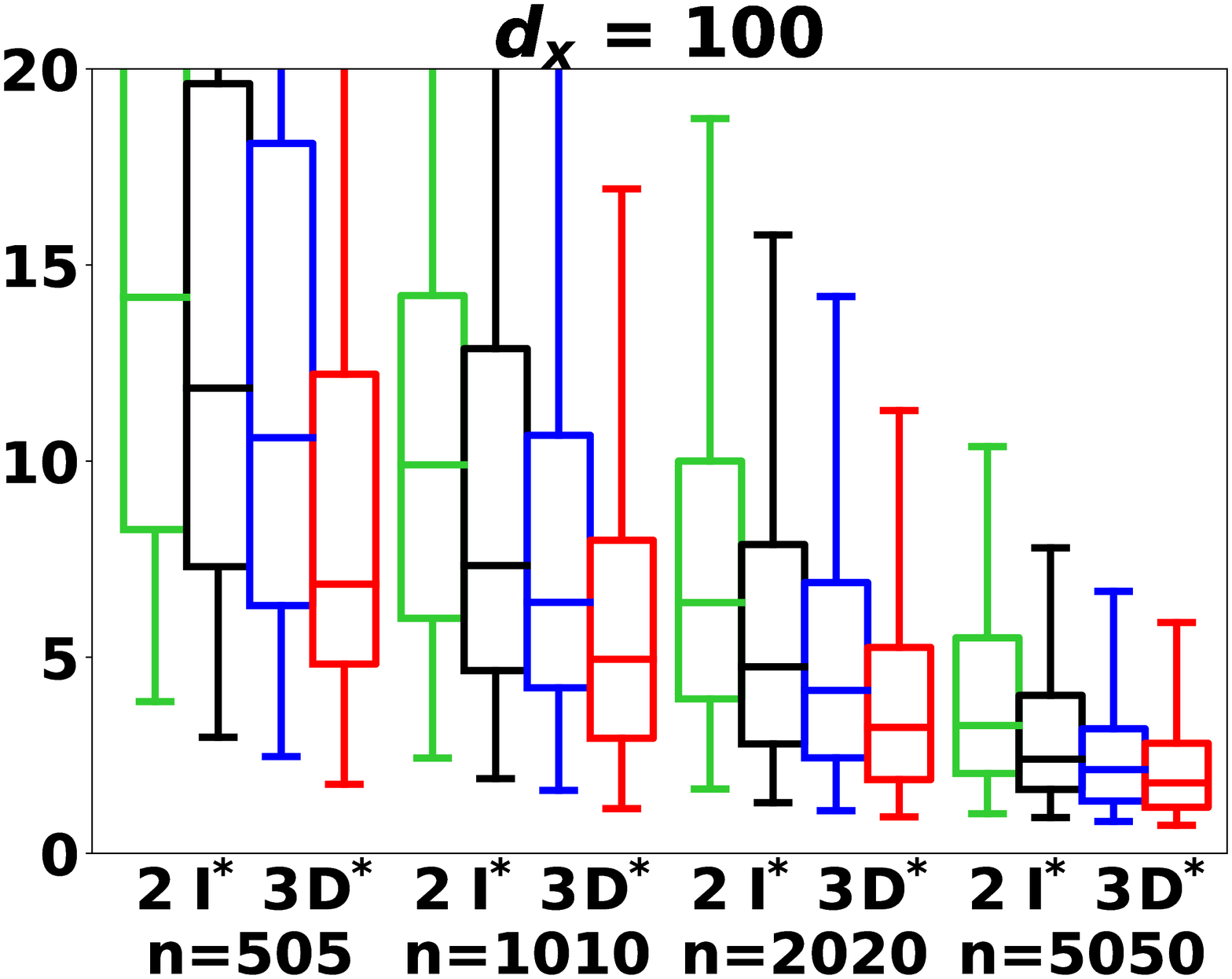}
    \end{subfigure}\\
    \begin{subfigure}[t]{0.33\textwidth}
        \centering
        \includegraphics[width=\textwidth]{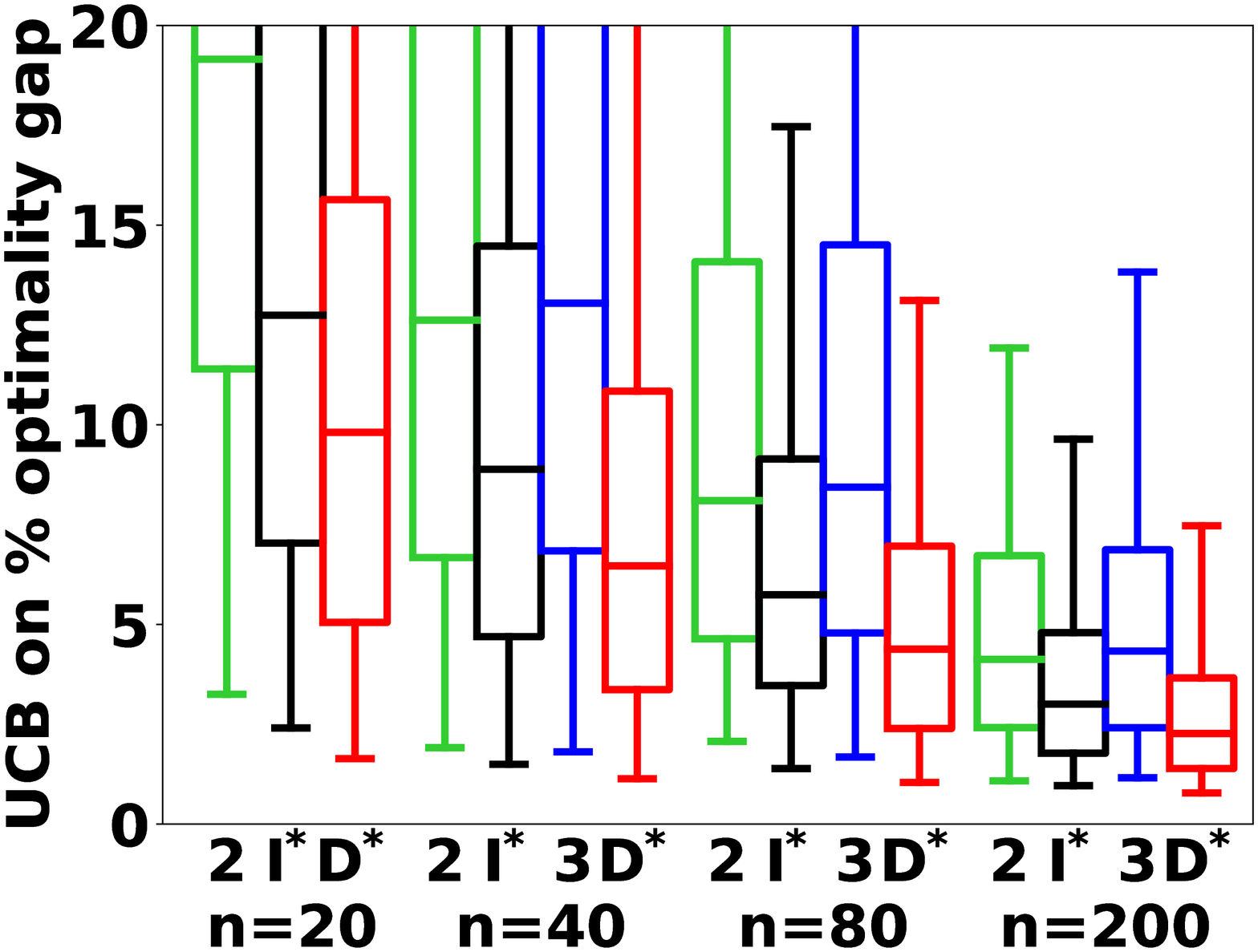}
    \end{subfigure}%
    ~ 
    \begin{subfigure}[t]{0.33\textwidth}
        \centering
        \includegraphics[width=\textwidth]{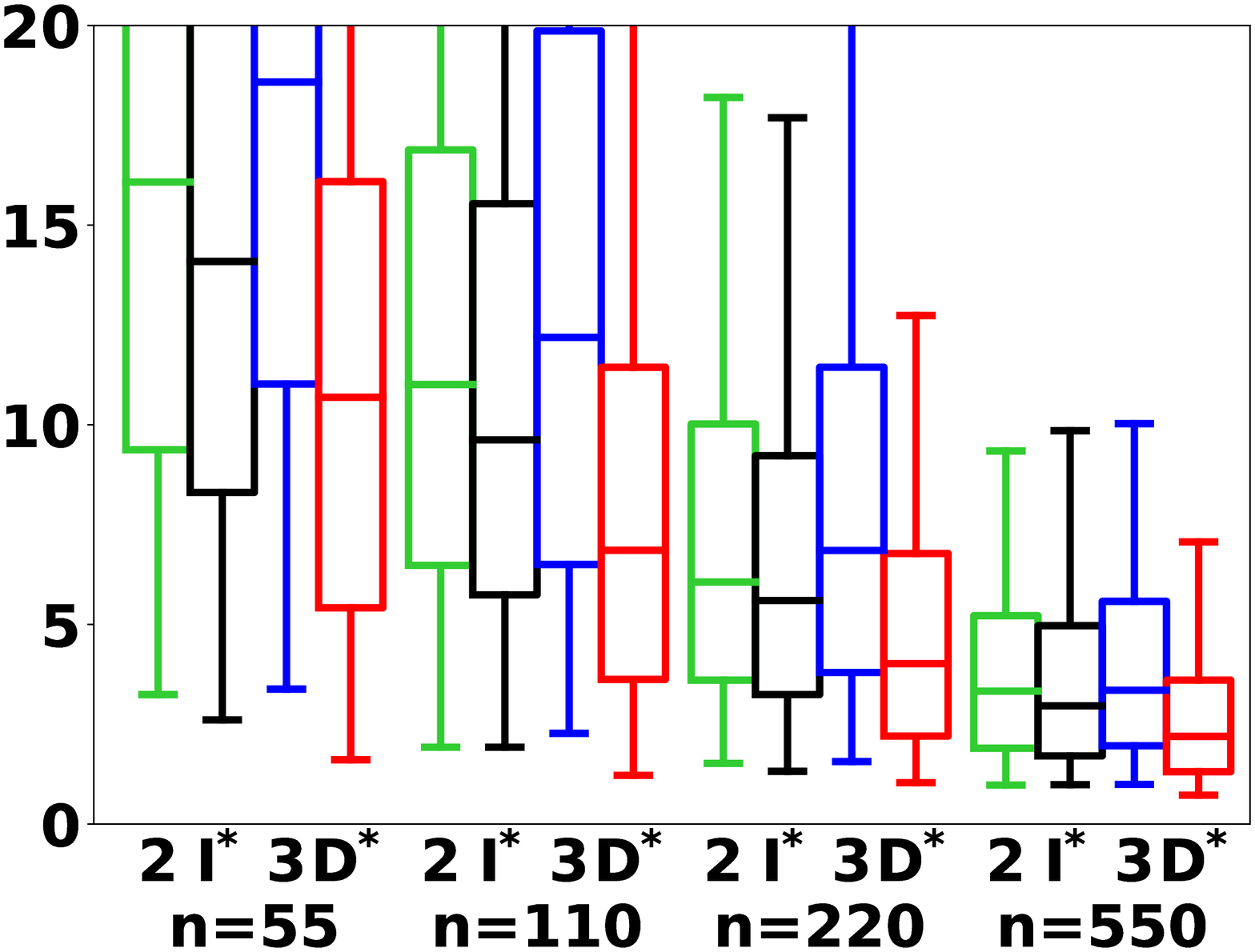}
    \end{subfigure}%
    ~ 
    \begin{subfigure}[t]{0.33\textwidth}
        \centering
        \includegraphics[width=\textwidth]{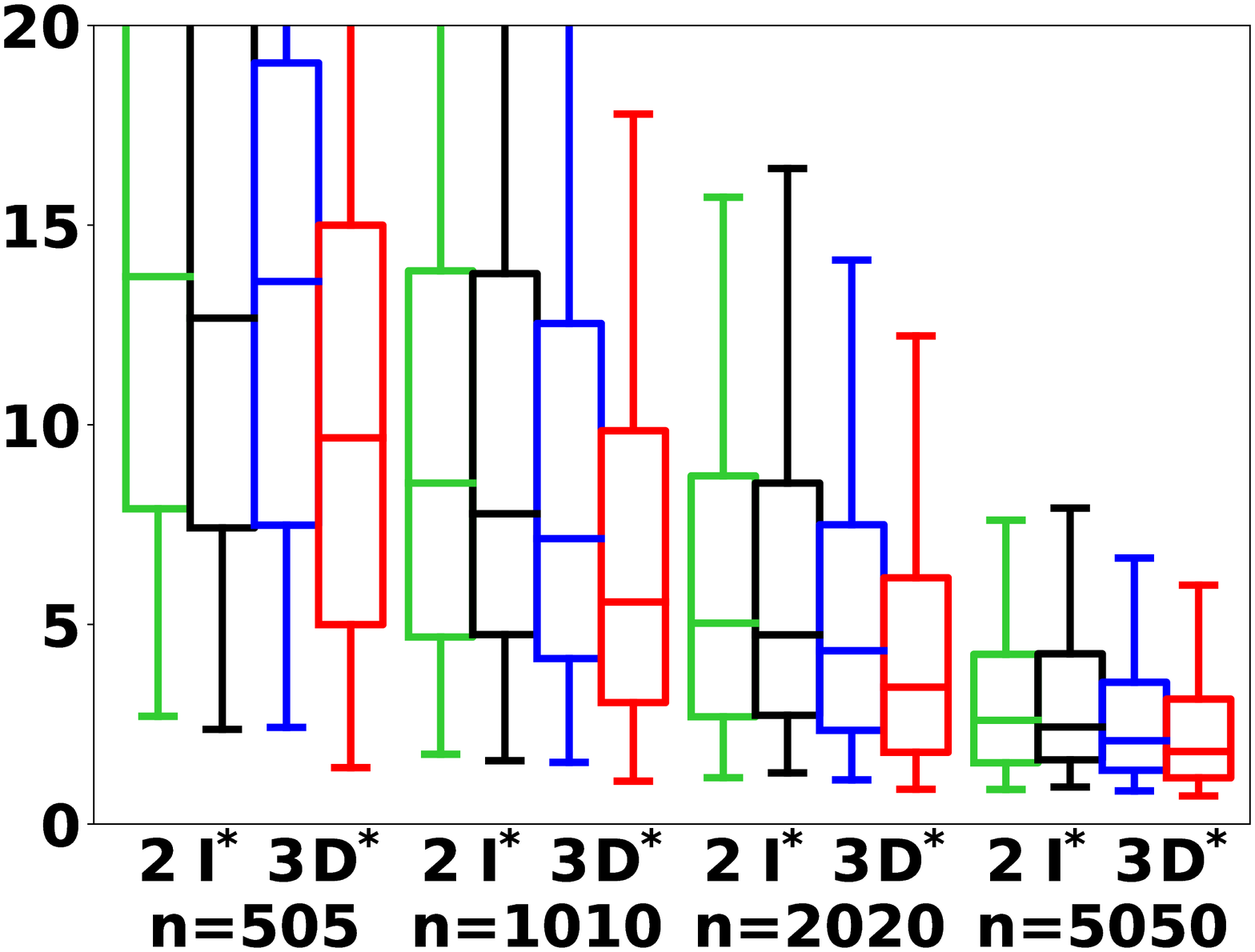}
    \end{subfigure}\\
    \begin{subfigure}[t]{0.33\textwidth}
        \centering
        \includegraphics[width=\textwidth]{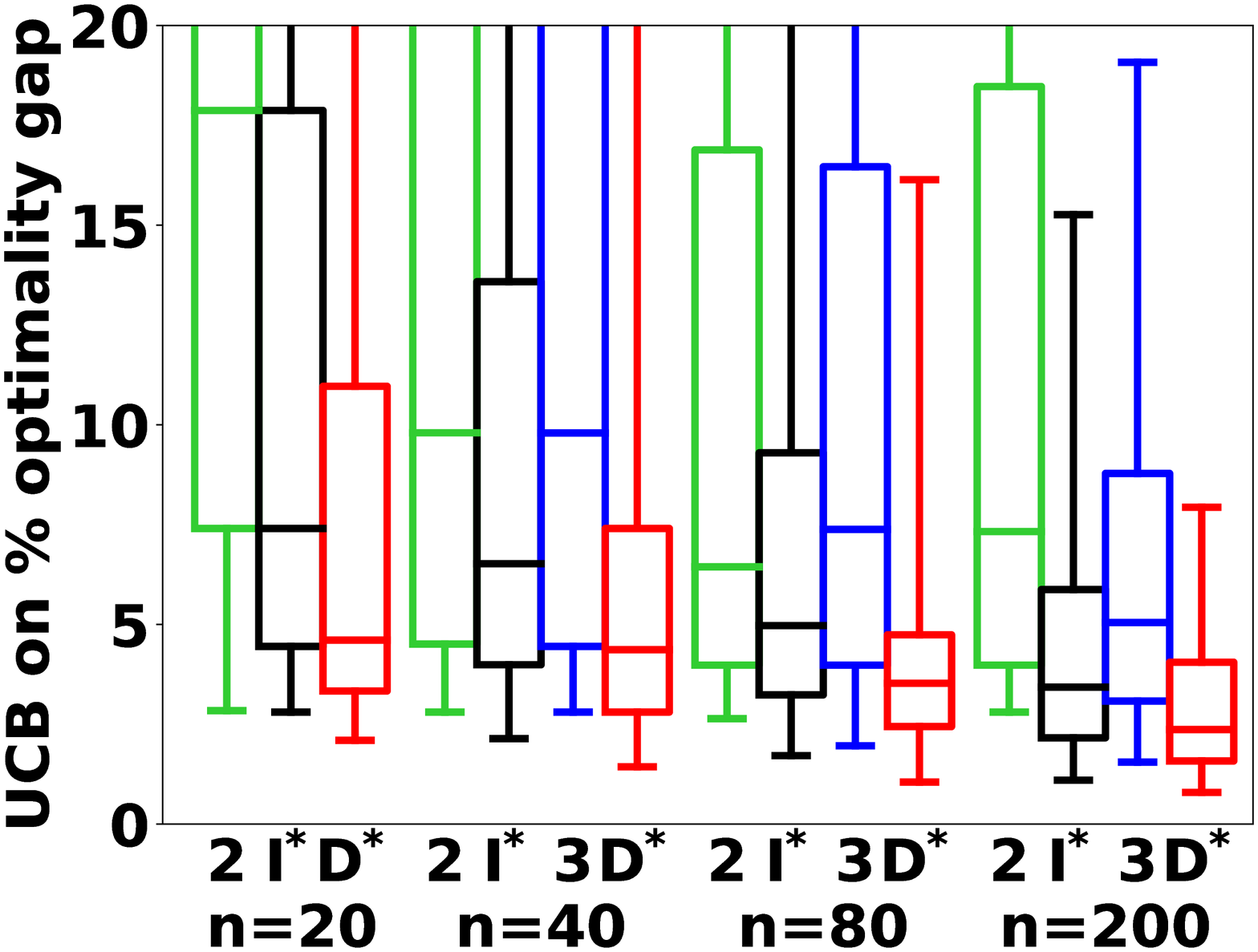}
    \end{subfigure}%
    ~ 
    \begin{subfigure}[t]{0.33\textwidth}
        \centering
        \includegraphics[width=\textwidth]{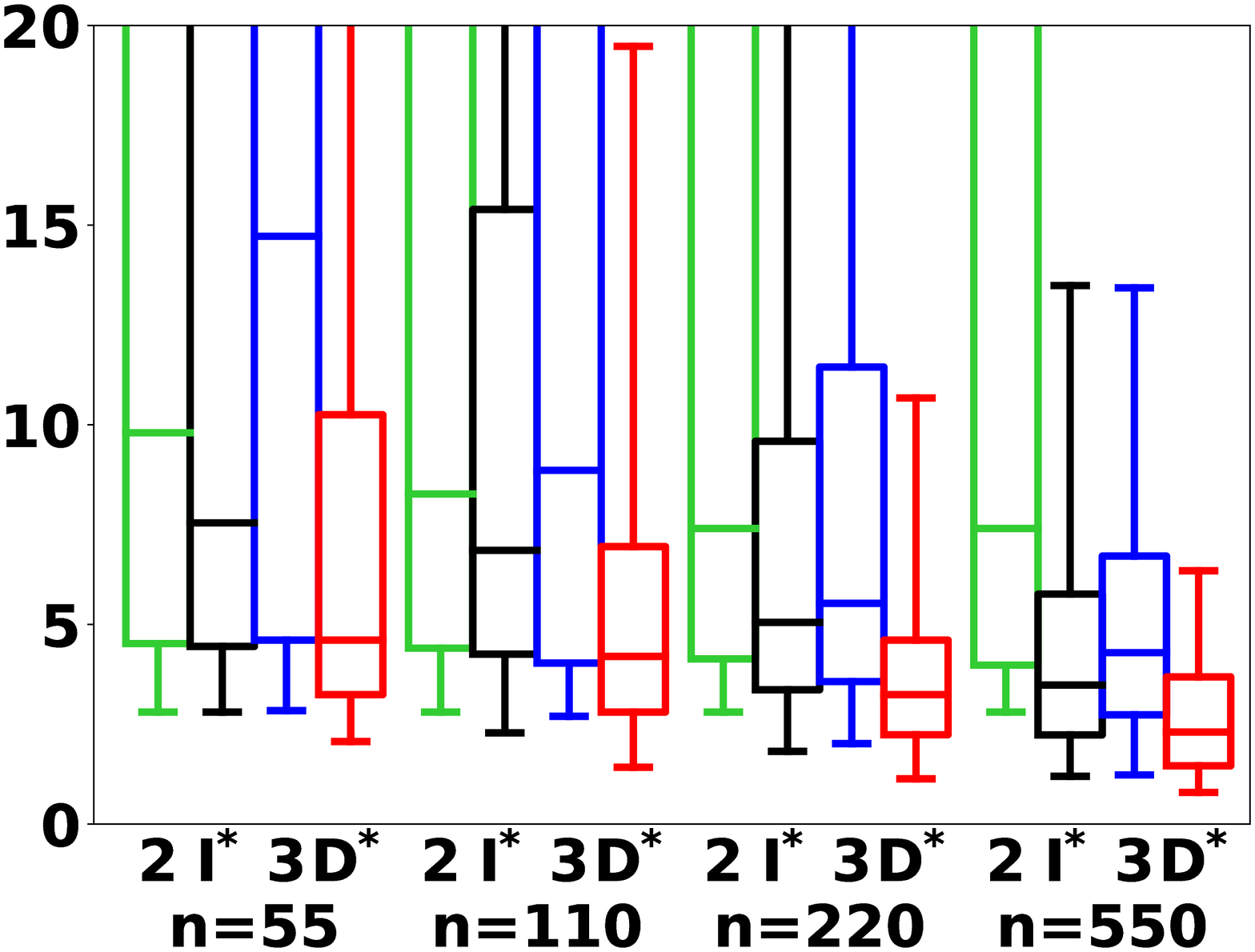}
    \end{subfigure}%
    ~ 
    \begin{subfigure}[t]{0.33\textwidth}
        \centering
        \includegraphics[width=\textwidth]{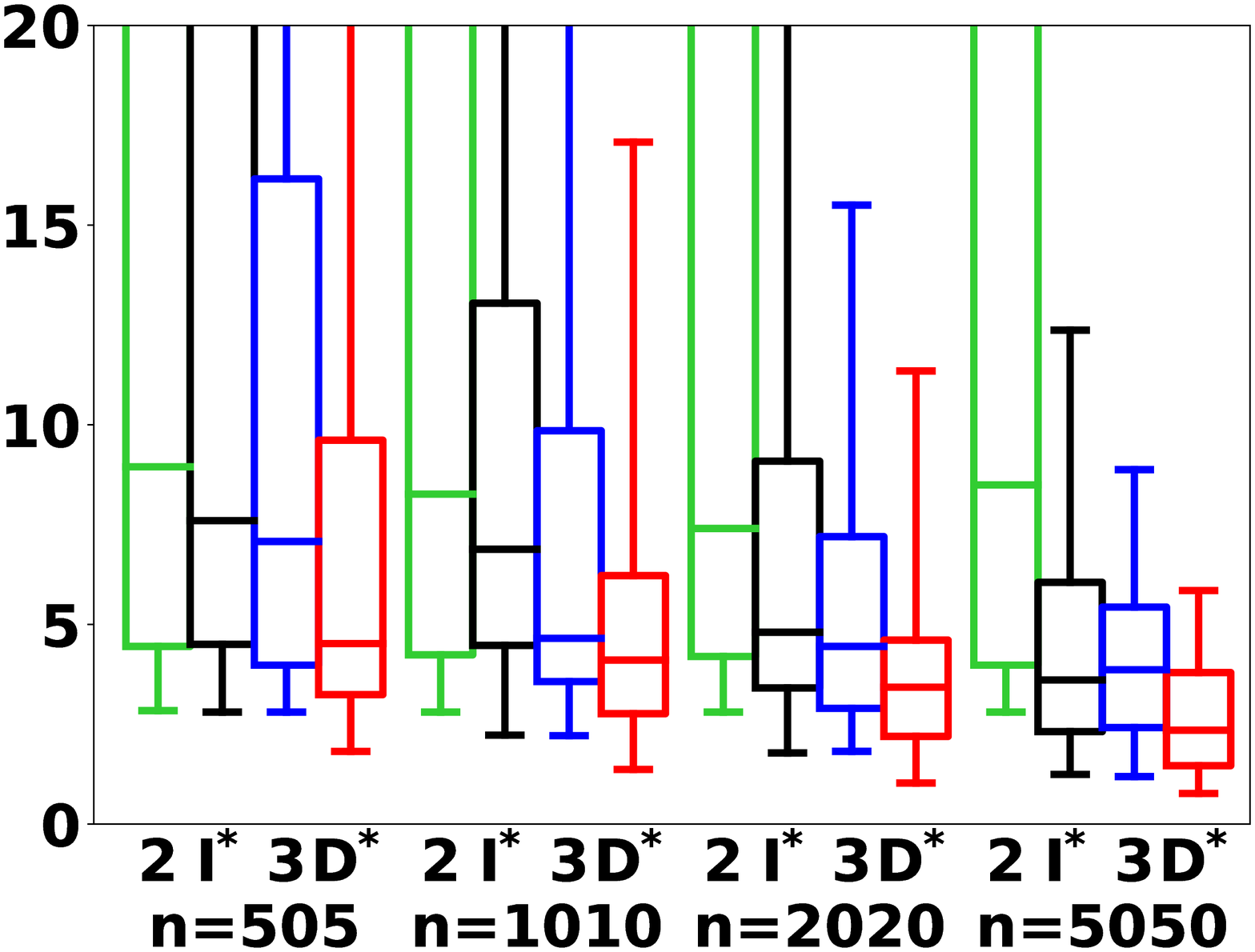}
    \end{subfigure}
    \caption{{\small \textbf{(Comparison with ``optimal'' specification of the Wasserstein radius)} Comparison of the \tW+OLS approach with the optimal covariate-dependent (\texttt{$\text{D}^*$}) and covariate-independent (\texttt{$\text{I}^*$}) tuning of the \tW+OLS radius, and the tuning of the \tW+OLS radius using Algorithm~\ref{alg:ersaadiffradius} (\texttt{3}) and Algorithm~\ref{alg:ersaasameradius} (\texttt{2}). Top row: $\theta = 1$. Middle row: $\theta = 0.5$. Bottom row: $\theta = 2$. Left column: $d_x = 3$. Middle column: $d_x = 10$. Right column: $d_x = 100$.}}
    \label{fig:comp_ols_wass_dep}
\end{figure}

\vspace{0.1in}
\noindent
\textbf{Comparison of the radii specified by Algorithms~\ref{alg:naivesaaradius},~\ref{alg:ersaasameradius}, and~\ref{alg:ersaadiffradius}.}
Figure~\ref{fig:comp_ols_wass_radii} compares the radii specified by Algorithms~\ref{alg:naivesaaradius},~\ref{alg:ersaasameradius}, and~\ref{alg:ersaadiffradius} with the optimal covariate-dependent radius and optimal covariate-independent radius for the \tW+OLS formulation.
We consider $d_x = 100$, vary the model degree $\theta$, and vary the sample size {among $n \in \{5(d_x + 1),10(d_x + 1),20(d_x + 1),50(d_x + 1)\}$ in these experiments}.
{Note that the $y$-axis limits are different across the subplots.}
First, note that the radius specified by Algorithm~\ref{alg:naivesaaradius} shrinks very quickly to zero for all three values of~$\theta$.
{Consequently, we note from Figure~\ref{fig:comp_ols_wass_indep} that the resulting ER-DRO estimators typically do not perform as well as the estimators obtained when the radius~$\zeta_n$ is specified using Algorithms~\ref{alg:ersaasameradius} and~\ref{alg:ersaadiffradius}.}
Second, we see that the covariate-independent specifications of the radius result in more narrow distributions {overall} compared to the covariate-dependent specifications.
This may be because the covariate-independent specifications of the radius attempt to choose a single value of~$\zeta_n(x)$ for all possible covariate realizations~$x \in \X$, whereas the covariate-dependent specifications can choose a different value of~$\zeta_n(x)$ depending on the realization~$x \in \X$.
Third, the distribution of the radius determined using Algorithm~\ref{alg:ersaadiffradius} {appears to converge} to the distribution of the optimal covariate-dependent radius as the sample size increases.
Similarly, the distribution of the radius determined using Algorithm~\ref{alg:ersaasameradius} {also appears to converge} to the distribution of the optimal covariate-independent radius as $n$ increases (except for the case when $\theta = 2$).
Finally, as noted in Section~\ref{subsec:wass_radius}, it may be advantageous to use a positive radius for the ambiguity set when the prediction model is misspecified (e.g., using OLS regression even when $\theta \neq 1$).
This is corroborated by the plots for $\theta = 2$, where the distribution of the optimal covariate-dependent radius is far from the zero distribution even for large sample sizes~$n$.
\\

\noindent
\textbf{Comparison with the {jackknife-based} formulations.}
Figure~\ref{fig:comp_ols_wass_jack} compares the performance of the ER-SAA+OLS approach and the {jackknife-based} SAA (J-SAA+OLS) approach \cite{kannan2020data} with the \tW+OLS formulations when the radius~$\zeta_n(x)$ is specified using Algorithms~\ref{alg:ersaasameradius} and~\ref{alg:ersaadiffradius}.
We consider $d_x = 100$, vary the model degree $\theta$, and vary the {sample size among $n \in \{3(d_x + 1),5(d_x + 1),10(d_x + 1),20(d_x + 1)\}$ in these experiments}.
As observed in~\cite{kannan2020data}, the J-SAA+OLS formulation performs better than the \tE+OLS formulation in the small sample size regime.
Figure~\ref{fig:comp_ols_wass_jack} shows that the \tW+OLS formulations outperform the J-SAA+OLS formulation {(except when using Algorithm~\ref{alg:ersaasameradius} for $\theta = 2$ and large $n$)}.
This is expected because the ER-DRO formulations account for both the errors in the approximation of $f^*$ by $\hf_n$ \textit{and} in the approximation of $P_{Y \mid X = x}$ by $P^*_n(x)$, whereas the J-SAA+OLS formulation only addresses the bias in the residuals obtained from OLS regression (i.e., even if $\hf_n$ is an accurate estimate of $f^*$, the jackknife formulations do not account for the fact that $P^*_n(x)$ may be a crude approximation of $P_{Y \mid X = x}$).
We omit the results for the J+-SAA+OLS formulation because they are similar to those for the J-SAA+OLS formulation.

\begin{figure}
    \centering
    \begin{subfigure}[t]{0.33\textwidth}
        \centering
        \includegraphics[width=\textwidth]{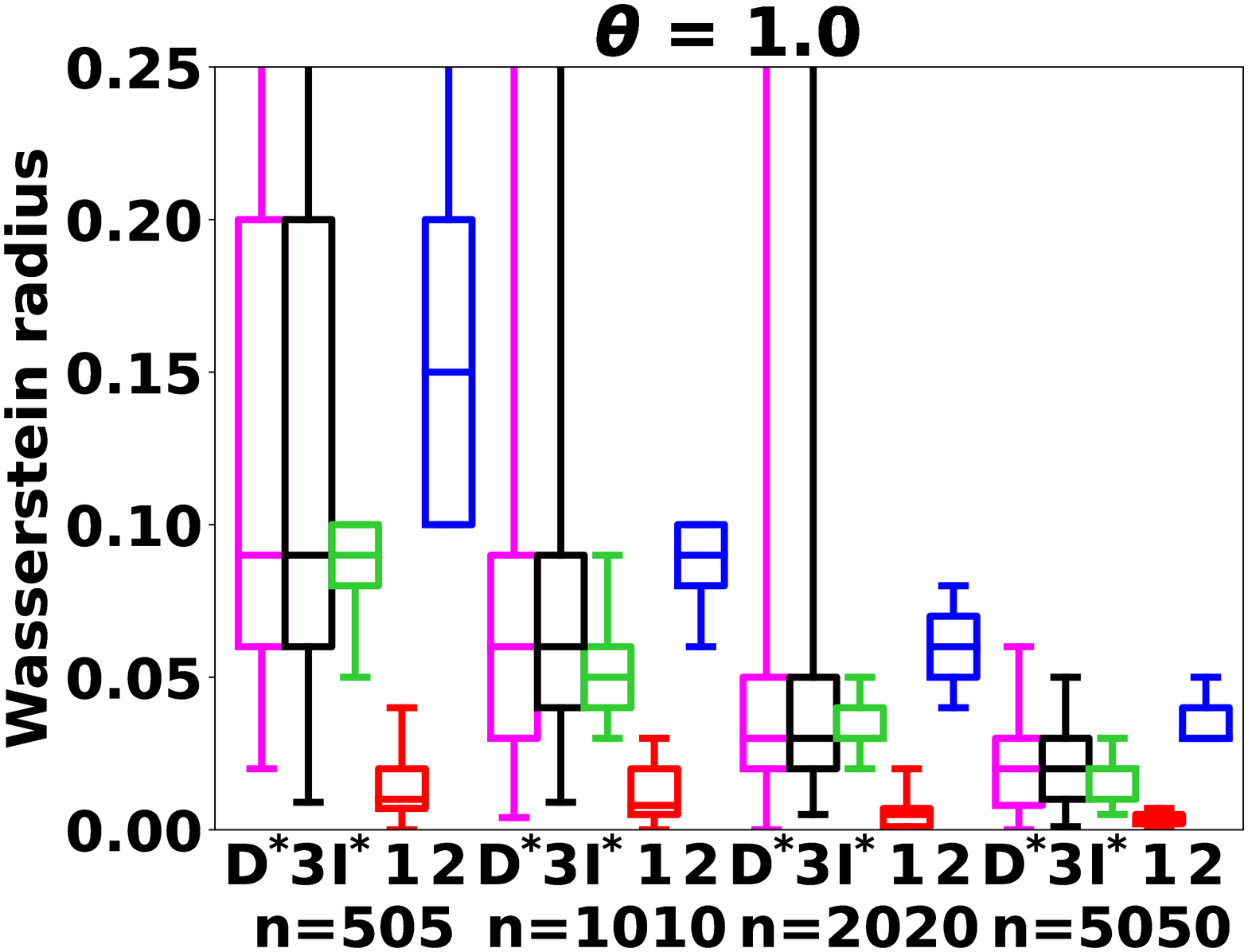}
    \end{subfigure}%
    ~ 
    \begin{subfigure}[t]{0.33\textwidth}
        \centering
        \includegraphics[width=\textwidth]{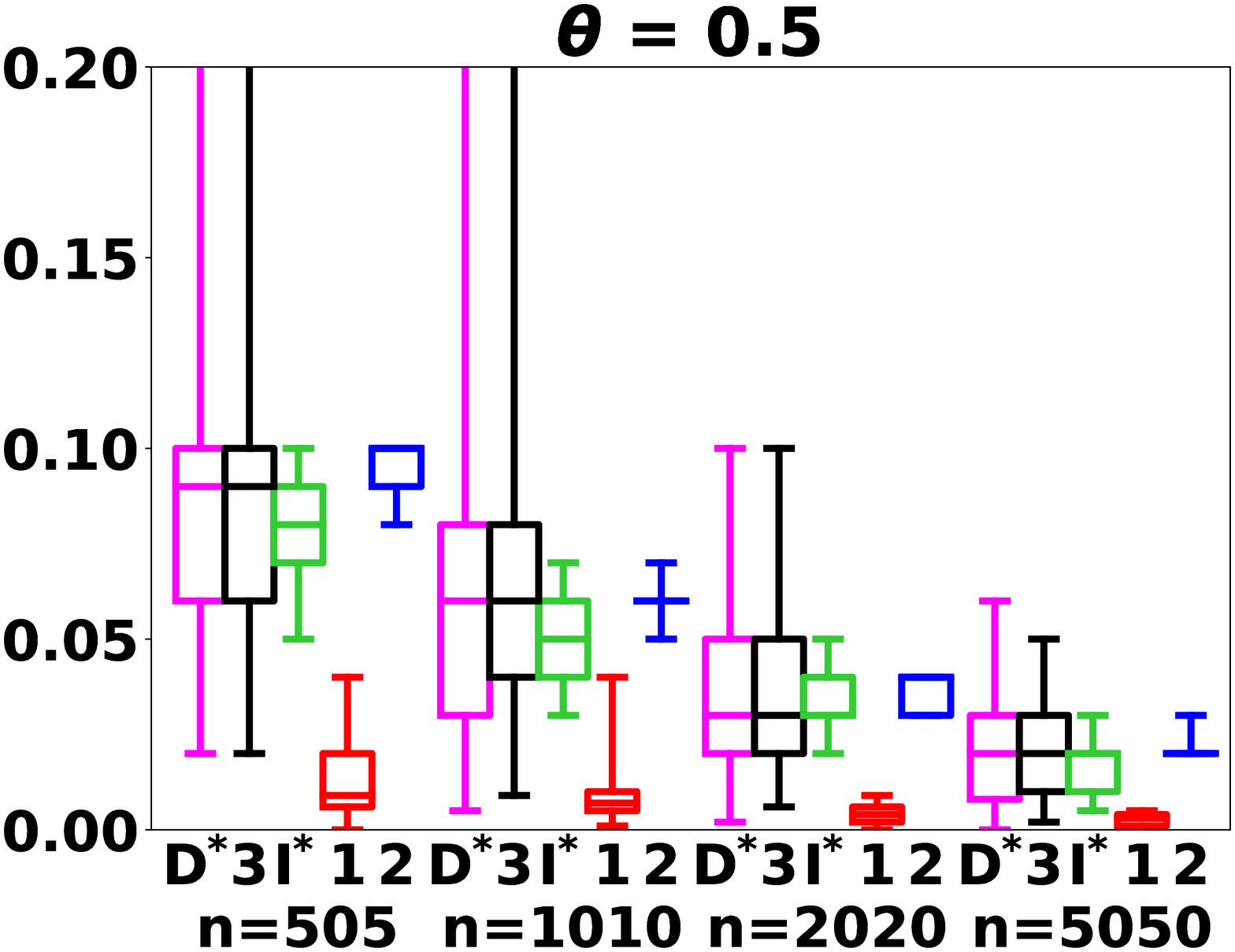}
    \end{subfigure}%
    ~ 
    \begin{subfigure}[t]{0.33\textwidth}
        \centering
        \includegraphics[width=\textwidth]{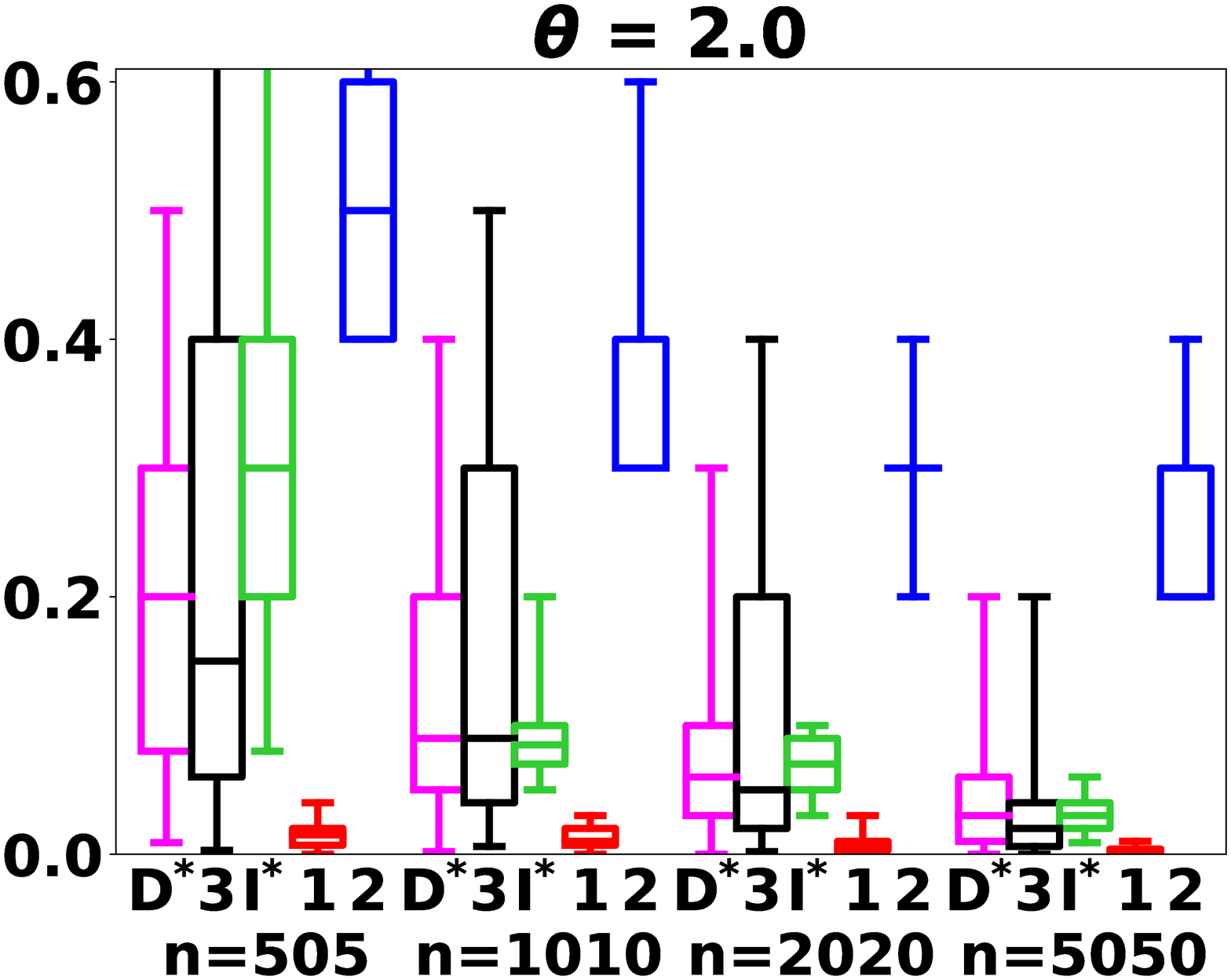}
    \end{subfigure}
    \caption{\small {\textbf{(Comparison of the radii specified by Algorithms~\ref{alg:naivesaaradius},~\ref{alg:ersaasameradius}, and~\ref{alg:ersaadiffradius})} Comparison of the optimal covariate-dependent tuning (\texttt{$\text{D}^*$}) and optimal covariate-independent tuning (\texttt{$\text{I}^*$}) of the \tW+OLS radius, the covariate-dependent tuning of the \tW+OLS radius using Algorithm~\ref{alg:ersaadiffradius} (\texttt{3}), and the covariate-independent tuning of the \tW+OLS radius using Algorithm~\ref{alg:naivesaaradius} (\texttt{1}) and Algorithm~\ref{alg:ersaasameradius} (\texttt{2}) for $d_x = 100$. Left: $\theta = 1$. Middle: $\theta = 0.5$. Right: $\theta = 2$.}}
    \label{fig:comp_ols_wass_radii}
\end{figure}

\begin{figure}
    \centering
    \begin{subfigure}[t]{0.33\textwidth}
        \centering
        \includegraphics[width=\textwidth]{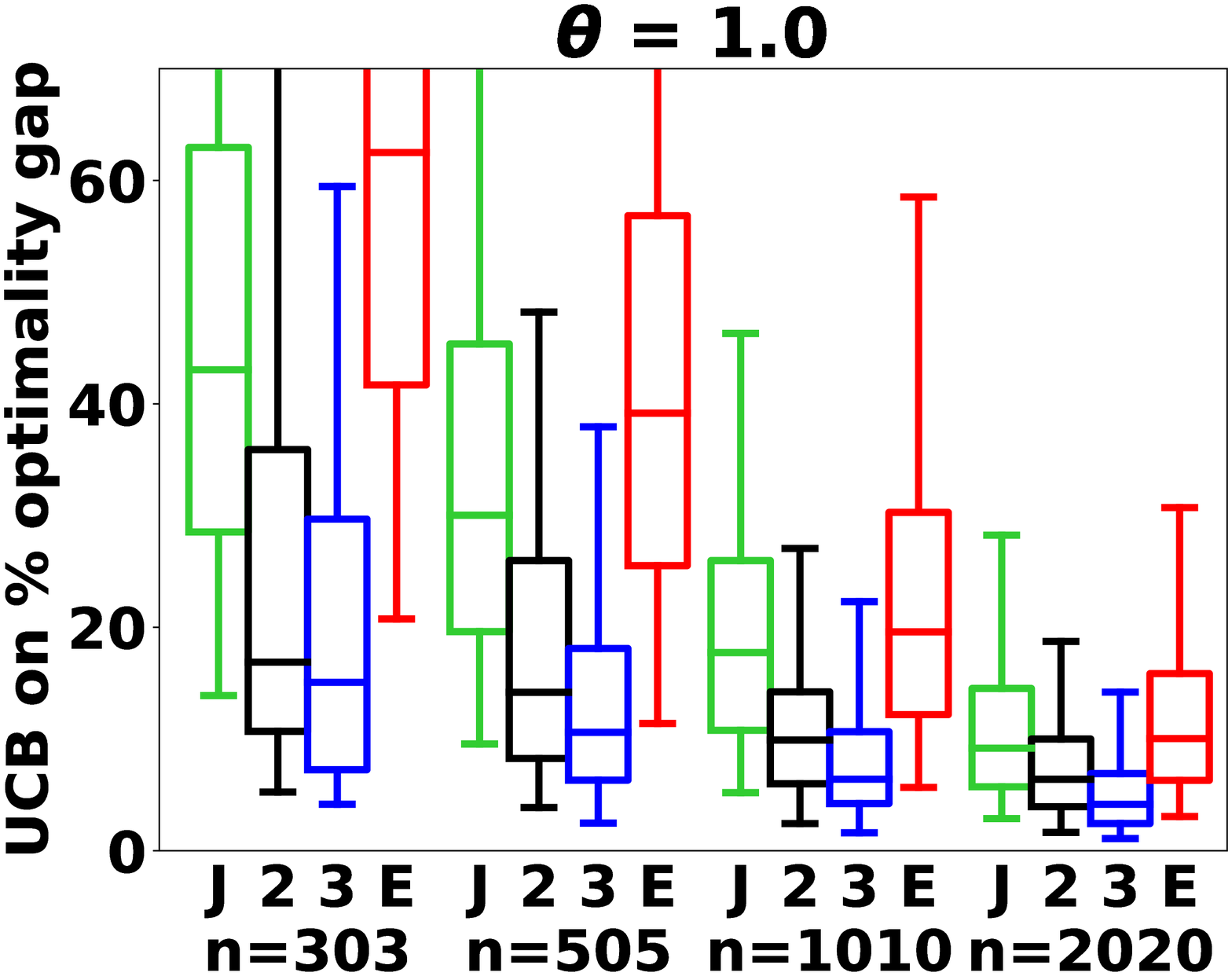}
    \end{subfigure}%
    ~ 
    \begin{subfigure}[t]{0.33\textwidth}
        \centering
        \includegraphics[width=\textwidth]{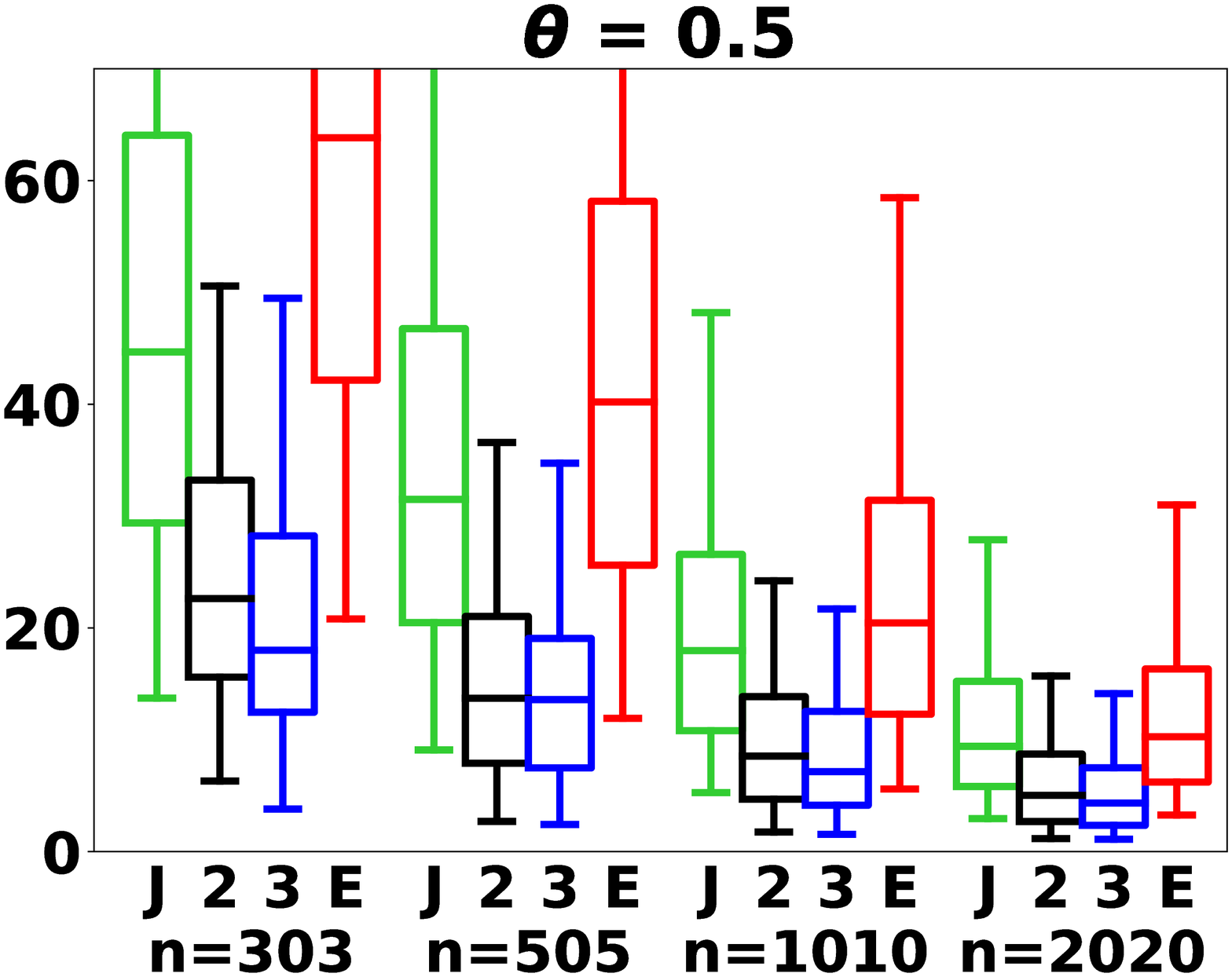}
    \end{subfigure}%
    ~ 
    \begin{subfigure}[t]{0.33\textwidth}
        \centering
        \includegraphics[width=\textwidth]{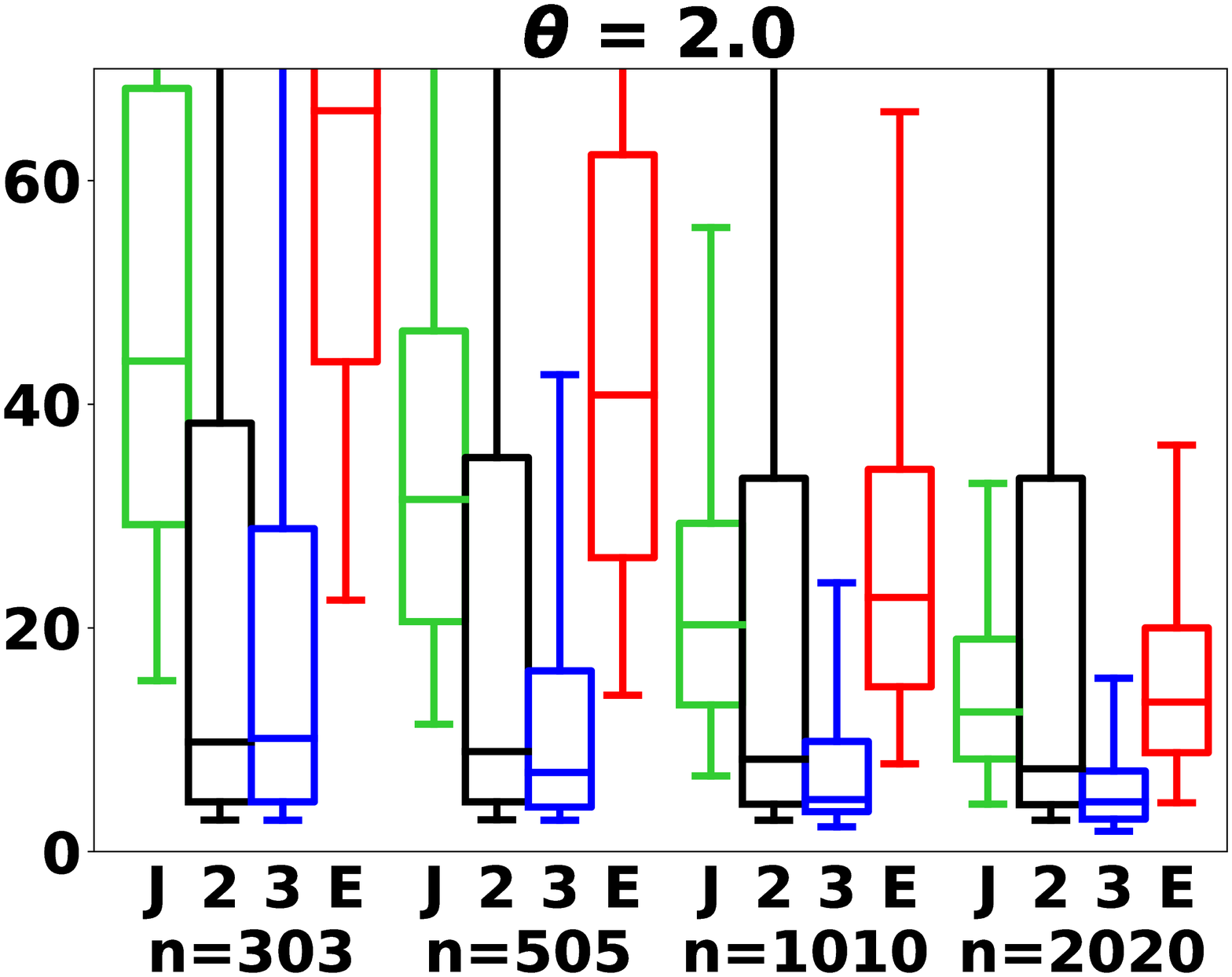}
    \end{subfigure}
    \caption{{\small \textbf{(Comparison of Wasserstein-DRO with J-SAA)} Comparison of the \tE+OLS (\tE) and \tJ+OLS (\tJ) approaches with tuning of the \tW+OLS radius using Algorithms~\ref{alg:ersaasameradius} (\texttt{2}) and~\ref{alg:ersaadiffradius} (\texttt{3}) for $d_x = 100$. Left: $\theta = 1$. Middle: $\theta = 0.5$. Right: $\theta = 2$.}}
    \label{fig:comp_ols_wass_jack}
\end{figure}

\end{document}